\documentclass[a4paper, 11pt]{article}
\usepackage{amsmath,amssymb,amsthm}
\usepackage[scaled]{helvet}
\usepackage{braket}
\usepackage{cleveref}
\usepackage{cases}
\usepackage{mathrsfs}
\usepackage{amsmath}
\usepackage{amsfonts}
\allowdisplaybreaks[4]
\usepackage{amsthm}
\usepackage{amscd}
\usepackage[all]{xy}
\usepackage{comment}
\usepackage{graphicx}
\usepackage[top=30truemm,bottom=30truemm,left=25truemm,right=25truemm]{geometry}
\usepackage{comment}
\usepackage{ascmac}
\usepackage[dvipdfmx]{xcolor}
\usepackage{tikz}
\usepackage[utf8]{inputenc}
\usepackage[T1]{fontenc}
\usepackage{mathtools}
\usetikzlibrary{shapes,arrows,positioning}

\usetikzlibrary{shapes.misc}
\usetikzlibrary{decorations.markings}
\tikzset{cross/.style={cross out, draw=black, minimum size=2*(#1-\pgflinewidth), inner sep=0pt, outer sep=0pt},
cross/.default={1pt}}
\tikzset{->-/.style={decoration={
  markings,
  mark=at position #1 with {\arrow[scale=1.5]{>}}},postaction={decorate}}}

\tikzset{-<-/.style={decoration={
  markings,
  mark=at position #1 with {\arrow[scale=1.5]{<}}},postaction={decorate}}}

\tikzset{midstealth/.style={decoration={
  markings,
  mark=at position #1 with {\arrow{stealth}}},postaction={decorate}}}
\newtheoremstyle{break}
  {}%
  {}%
  {\itshape}
  {}%
  {\bfseries}
  {.}%
  {\newline}%
  {}%

 \makeatletter
    
    \@addtoreset{equation}{section}
  \makeatother

\theoremstyle{plain}
\newtheorem{thm}{Theorem}[section]
\newtheorem{lem}[thm]{Lemma}
\newtheorem{prop}[thm]{Proposition}
\newtheorem{cor}[thm]{Corollary}
\newtheorem{exa}[thm]{Example}
\newtheorem{defn}[thm]{Definition}
\newtheorem{rem}[thm]{Remark}

\theoremstyle{break}

\DeclareMathOperator{\rank}{rank}

\DeclareMathOperator{\vol}{vol}

\DeclareMathOperator{\im}{Im}

\DeclareMathOperator{\Ker}{Ker}
\DeclareMathOperator{\Hom}{Hom}
\DeclareMathOperator{\Ext}{Ext}
\DeclareMathOperator{\Homo}{H}
\DeclareMathOperator{\Int}{Int}
\DeclareMathOperator{\sol}{Sol}
\DeclareMathOperator{\diag}{diag}
\DeclareMathOperator{\DR}{DR}
\DeclareMathOperator{\FL}{FL}
\DeclareMathOperator{\id}{id}
\DeclareMathOperator{\cone}{cone}
\newcommand{\barsigma}{\overline{\sigma}}
\newcommand{\R}{\mathbb{R}}

\newcommand{\C}{\mathbb{C}}
\newcommand{\DD}{\mathcal{D}}
\newcommand{\LL}{\mathbb{L}}
\newcommand{\A}{\mathbb{A}}
\newcommand{\Q}{\mathbb{Q}}
\newcommand{\Z}{\mathbb{Z}}
\newcommand{\T}{\mathbb{T}}
\newcommand{\lef }{\left\{ }
\newcommand{\righ }{\right\} }
\newcommand{\Gm}{\mathbb{G}_m}
\newcommand{\bartau}{\bar{\tau}}
\newcommand{\bm}{{\bf m}}
\newcommand{\ii}{\sqrt{-1}}
\newcommand{\s}{\sigma}
\newcommand{\bt}{\bar{\tau}}
\newcommand{\bs}{\barsigma}
\newcommand{\ts}{\tilde{\sigma}}
\newcommand{\cA}{\mathring{A}}
\newcommand{\ca}{\mathring{\bf a}}

\newcommand*{\transp}[2][-3mu]{\ensuremath{\mskip1mu\prescript{\smash{\mathrm t\mkern#1}}{}{\mathstrut#2}}}

\crefname{thm}{Theorem}{Theorems}
\crefname{bthm}{Theorem}{Theorems}
\crefname{bprop}{Proposition}{Propositions}
\crefname{blem}{Lemma}{Lemmata}
\crefname{prop}{Proposition}{Propositions}
\crefname{lem}{Lemma}{Lemmata}
\crefname{bcor}{Corollary}{Corollary}
\crefname{cor}{Corollary}{Corollary}

\title{Laplace, Residue, and Euler integral representations of GKZ hypergeometric functions }
\author{Saiei-Jaeyeong Matsubara-Heo\footnote{Graduate School of Mathematical Sciences, The University of Tokyo, 3-8-1 Komaba, Meguro, Tokyo, 153-8914 Japan.\newline e-mail: \texttt{saiei@ms.u-tokyo.ac.jp}}}

\begin{document}

\date{}
\maketitle

\begin{abstract}      
We consider four types of representations of solutions of GKZ system: series representations, Laplace integral representations, Euler integral representations, and Residue integral representations which will be introduced in this paper. In the former half of this paper, we provide a method for constructing integration cycles for Laplace, Residue, and Euler integral representations and relate them to series representations. In the latter half, we reformulate our integral representations in terms of direct images of $D$-modules and show their equivalence.
\end{abstract}

\begin{section}{Introduction}

\indent

 It is widely recognised that classical Gauss hypergeometric function 
\begin{equation}\label{series}
{}_2F_1(\alpha,\beta,\gamma ;z)=\displaystyle\sum_{n=0}^\infty\frac{(\alpha)_n(\beta)_n}{(\gamma)_n(1)_n}z^n
\end{equation}
has been a central object in the theory of special functions. On the other hand, ${}_2F_1(\alpha,\beta,\gamma ;z)$ has the so-called Euler integral representation which enables one to perform an analytic continuation:
\begin{equation}\label{Euler}
{}_2F_1(\alpha, \beta, \gamma ; z)=\frac{\Gamma(\gamma)}{\Gamma(\gamma -\alpha)\Gamma(\alpha)}\int^1_0t^{\alpha -1}(1-t)^{\gamma -\alpha -1}(1-zt)^{-\beta}dt\;\; (|z|<1).
\end{equation}
Precisely speaking, we need some conditions of the parameters $\alpha,\; \beta,\; \gamma$ so that the integral (\ref{Euler}) is convergent, but we do not discuss it here. However, by the definition of $\Gamma$ function, this integral can also be transformed into Laplace type integral representation

\begin{align}
 &{}_2F_1(\alpha, \beta, \gamma ; z)\nonumber\\
 =&\frac{\pi\Gamma(\gamma)}{\sin\pi(\gamma-\alpha)\Gamma(\alpha)\Gamma(\beta)}\int^1_0\int_0^\infty\int_0^\infty e^{ -(1-t_1)t_2-(1-zt_1)t_3} t_1^{\alpha-1}t_2^{\alpha-\gamma}t_3^{\beta-1}dt_1dt_2dt_3\label{GaussLaplace}
\end{align}
as long as the inequality $|z|<1$ holds. This trick is known as \lq\lq{}Cayley trick\rq\rq{} in the literature. While the integration contour of (\ref{GaussLaplace}) is unbounded, one can also see that the Euler integral (\ref{Euler}) can also be rewritten as
\begin{align}\label{GaussResidue}
  &{}_2F_1(\alpha, \beta, \gamma ; z)\nonumber\\
 =&\frac{\Gamma(\gamma)}{\Gamma(\gamma-\alpha)(2\pi\sqrt{-1})^2}\int_0^1\oint_{L_2}\oint_{L_1}\frac{t^{\alpha-1}y_1^{\alpha-\gamma}y_2^{\beta-1}}{\Big(1-y_1(1-t)\Big)\Big( 1-y_2(1-zt)\Big)}dy_1dy_2dt,
\end{align}
where $L_1$ and $L_2$ are small circles around $\frac{1}{y_1}=1-t$ and $\frac{1}{y_2}=1-zt$ respectively. Note that the integration contour of (\ref{GaussResidue}) remains bounded. We shall call the integral representation (\ref{GaussResidue}) the Residue integral representation in this paper.

The observation above shows that Gauss hypergeometric function has four different representations: series representation (\ref{series}), Euler integral representation (\ref{Euler}), Laplace integral representation (\ref{GaussLaplace}), and Residue integral representation (\ref{GaussResidue}). In this paper, we are going to establish the relation among these points of views for GKZ hypergeometric functions. Let us first briefly  overview the state of research.

GKZ hypergeometric system $M_A(c)$, which is of our central interest in this paper, was introduced by I.\  M.\  Gelfand and his coworkers as a generalisation of series representation of ${}_2F_1(\alpha,\beta,\gamma;z).$ $M_A(c)$ includes classically important special functions as particular cases. For example, Gauss, Kummer, Bessel, Hermite-Weber, Airy functions, Appell-Lauricella series, or Horn series (\cite{E}) can naturally be grasped in terms of $M_A(c)$. Let us revise the definition of GKZ hypergeometric system: Let $n<N$ be positive integers, $c\in\mathbb{C}^{n\times 1}$ be a fixed parameter, and let $\{ {\bf a}(1),\dots,{\bf a}(N)\}\subset\mathbb{Z}^{N\times 1}$ be lattice points. We set $A=({\bf a}(1)\mid \cdots\mid {\bf a}(N))=(a_{ij}).$ Throughout this paper, we assume that $\Z A=\displaystyle\sum_{j=1}^N\Z {\bf a}(j)=\Z^{n\times 1}.$
The GKZ hypergeometric system is given by the family of equations 

\begin{subnumcases}{M_A(c):}
E_i\cdot f(z)=0 &($i=1,\cdots, n$)\label{EulerEq}\\
\Box_u\cdot f(z)\hspace{-0.8mm}=0& ($u\in L_A=\Ker_{\mathbb{Z}}A$),\label{ultrahyperbolic}
\end{subnumcases}
where $E_i$ and $\Box_u$ are differential operators defined by 

\begin{equation}
E_i=\sum_{j=1}^Na_{ij}z_j\frac{\partial}{\partial z_j}+c_i,\;\;
\Box_u=\prod_{u_j>0}\left(\frac{\partial}{\partial z_j}\right)^{u_j}-\prod_{u_j<0}\left(\frac{\partial}{\partial z_j}\right)^{-u_j}.
\end{equation}
This system is, without any restriction on the parameter $c$ and $A$, a holonomic system (\cite{A}), and thus, has finitely many linearly independent solutions. An important observation is that, the so-called $\Gamma-$series
\begin{equation}\label{gammaseries}
\varphi_v(z)=\displaystyle\sum_{u\in {\rm L}_A}\frac{z^{u+v}}{\Gamma({\bf 1}+u+v)},
\end{equation}
where $v\in\C^N$ satisfies equations (\ref{EulerEq}) and (\ref{ultrahyperbolic}) when $Av=-c$, and that one can construct a basis of holomorphic solutions of $M_A(c)$ consisting only of $\Gamma$-series $\varphi_v(z)$ with the aid of the so-called regular triangulation of the Newton polytope $\Delta_A\overset{\rm def}{=}c.h.\{0,{\bf a}(1),\dots,{\bf a}(N)\}$ (see \cite{GGR} and \cite{FF}). Thus, the viewpoint of series representation (\ref{series}) of ${}_2F_1(\alpha,\beta,\gamma;z)$ was successfully generalised to that of GKZ hypergeometric functions (\ref{gammaseries}).

Concerning a generalisation of Euler integral representation (\ref{Euler}), it was shown in \cite{GKZ} that for any given $k$ Laurent polynomials
$h_{l,z^{(l)}}(x)=\displaystyle\sum_{j=1}^{N_l}z_j^{(l)}x^{{\bf a}^{(l)}(j)}\;\;(l=1,\dots,k),$
where $z_j^{(l)}$ is regarded as a variable, and for any parameters $\gamma_l\in\C\;\;(l=1,\dots k)$ and $c\in\C^{n},$ the Euler integral representation
\begin{equation}\label{EulerInt}
f(z)=\int_\Gamma h_{1,z^{(1)}}(x)^{-\gamma_1}\cdots h_{k,z^{(k)}}(x)^{-\gamma_k}x^{c-1} dx
\end{equation}
satisfies a GKZ system $M_A(d)$ for suitable matrix $A$ and parameters $d=\begin{pmatrix}
\gamma\\
c
\end{pmatrix}
\in\C^{n+k}.$ Note that any regular holonomic GKZ system has such an integral representation for suitable Laurent polynomials $h_{l,z^{(l)}}(x)$. In the case when each Laurent polynomial $h_{l,z^{(l)}}(x)$ is a linear polynomial, it is possible to construct a basis of cycles of (\ref{EulerInt}) based on the geometry of the space $\T^n\setminus\{ h_{1,z^{(1)}}(x)\cdots h_{k,z^{(k)}}(x)=0\}$ as was fully explored by K. Aomoto and others (see e.g. \cite{AK} or \cite{OT}). They showed that if all $z_j^{(l)}$ are real, the space of cycles can be computed by means of combinatorics of hyperplane arrangements. However, one has to know complex cycles in general case and it is still a challenging problem to construct a basis of cycles $\Gamma$ for (\ref{EulerInt}). Moreover, the relation between (\ref{gammaseries}) and (\ref{EulerInt}) remains unclear.

Regarding a generalisation of Laplace integral (\ref{GaussLaplace}), several authors developed a systematic study of Laplace integral representations based on Cayley trick (\cite{A}, \cite{ET}, \cite{H}, \cite{SW}). In particular, under certain genericity assumption (non-resonance) of the parameter $c$, a canonical isomorphism between $M_A(c)$ and a certain direct image of a $D$-module was established in \cite{SW}. This result suggests that, for a  generic parameter $c$, a general solution $f(z)$ of $M_A(c)$ is given by the Laplace integral

\begin{equation}\label{Laplace integral}
f(z)=\int_\Gamma \exp\left\{\sum_{j=1}^Nz_jx^{{\bf a}(j)}\right\}x^{\bf c-1}dx,
\end{equation}
where $\Gamma$ is a suitable (unbounded) cycle. Thus, the viewpoint of Laplace integral representation (\ref{Laplace}) was generalised to that of GKZ hypergeometric functions. 

It is noteworthy that there is a progress concerning the construction of integration cycles $\Gamma$ of (\ref{Laplace integral}). In \S 5 of \cite{ET}, it was examined that under the assumption $0\in\Int\Delta_A$, one can employ a general procedure of constructing Lefschetz thimbles (method of steepest descent) to construct a basis of cycles $\Gamma$ of (\ref{Laplace integral}). As a byproduct, they obtained the leading term of the asymptotic expansion of (\ref{Laplace integral}) when $f(z)$ is restricted to a complex line passing through the origin (stationary phase approximation). However, their assumption is too restrictive since, for example, any regular holonomic GKZ system $M_A(c)$ never satisfies the assumption $0\in\Int\Delta_A$. Moreover, it is not easy to relate these steepest descent contours to series representations (\ref{gammaseries}).

As for a generalisation of Residue integral (\ref{GaussResidue}), the following integral was treated in \cite{B}:
\begin{equation}\label{ResidueInt}
f(z)=\int_\Gamma\frac{y^{\gamma -1}x^{c-1}}{(1-y_1h_{1,z^{(1)}}(x))\cdots (1-y_kh_{k,z^{(k)}}(x))}dydx.
\end{equation}
In \cite{B}, the integral (\ref{ResidueInt}) is called ``Euler integral''. The author indicated how to construct the integration contour $\Gamma$ as a generalisation of Pochhammer cycle when $k=1$, but did not investigate its relation to (\ref{gammaseries}). Moreover, the relation between (\ref{EulerInt}) and (\ref{ResidueInt}) is unclear. 

The purpose of this paper is two-folds: (I) to construct integration cycles explicitly for (\ref{EulerInt}), (\ref{Laplace integral}), and (\ref{ResidueInt}) and to relate them to series representations (\ref{gammaseries}), and (I\hspace{-.1em}I) to show the equivalence of three integral representations (\ref{EulerInt}), (\ref{Laplace integral}), and (\ref{ResidueInt}).

In the former half of this paper, we focus on (I). Our strategy is as follows. First, we take a suitable covering change of coordinate of the torus $\T^n$ associated to a regular triangulation $T$. Then, in this new coordinate, we construct a standard integration cycle $\Gamma$, as a product of Hankel contour and a multidimensional Pochhammer cycle for (\ref{Laplace integral}), and as a product of Pochhammer cycles for (\ref{EulerInt}) and  (\ref{ResidueInt}).  Finally, we consider suitable deck transformations of this cycle to obtain a basis of cycles. Though the construction is almost straightforward, we will find that the transformation matrix between cycles of (\ref{EulerInt}), (\ref{Laplace integral}), or (\ref{ResidueInt}), and a basis consisting of $\Gamma-$series (\ref{gammaseries}) is given in terms of the character matrix of a finite Abelian group associated to the regular triangulation $T$. These results correspond to \cref{thm:fundamentalthm1},  \cref{thm:fundamentalthm2} and \cref{thm:fundamentalthm3}.

In the latter half of this paper, we switch our attention to problem (I\hspace{-.1em}I). We consider the following integral:
\begin{equation}
\int_\Gamma h_{1,z^{(1)}}(x)^{-\gamma_1}\cdots h_{k,z^{(k)}}(x)^{-\gamma_k}x^{c-1} e^{h_{0,z^{(0)}}(x)}dx,\label{KummerInt}
\end{equation}
where each $h_{l,z^{(l)}}(x)\;\;(l=0,\dots,k)$ is again a Laurent polynomial. This can be seen as an interpolation of Euler integral representations (\ref{EulerInt}) and Laplace integral representations (\ref{Laplace integral}). We also consider an interpolation of Laplace integral representations (\ref{Laplace integral}) and Residue integral representations (\ref{ResidueInt}):
\begin{equation}\label{MixedIntegral}
\int_\Gamma\frac{e^{h_{0}(x)}y^{\gamma -1}x^{c-1}}{(1-y_1h_{1}(x))\cdots (1-y_kh_{k}(x))}dydx.
\end{equation}
It can be confirmed that such integrals satisfy GKZ system with a suitable matrix $A$ and a parameter 
$d=\begin{pmatrix}
\gamma\\
c
\end{pmatrix}$. 
Then, (I\hspace{-.1em}I) follows once we formulate and prove the equivalence of three integral representations (\ref{Laplace integral}), (\ref{KummerInt}), and (\ref{MixedIntegral}). At this point, we can formulate the equivalence as canonical isomorphisms of suitable direct images of $D$-modules corresponding to three integral representations (\cref{thm:CayleyTrick} and \cref{thm:ComposedResidue}). It is worth pointing out that we can also obtain an isomorphism between the local system of cycles for (\ref{ResidueInt}) and the solution sheaf of the corresponding GKZ system (\cref{thm:ResidueIntegral}). We will conclude this paper with a general construction of cycles for (\ref{KummerInt}) and  (\ref{MixedIntegral}) and its relation to series representation (\ref{gammaseries}) in terms of a character matrix of a finite Abelian group associated to a regular triangulation which unifies the treatment of \S\ref{SectionLaplace}, \S\ref{SectionResidue} and \S\ref{SectionEuler} (\cref{thm:fundamentalthm4} and \cref{thm:fundamentalthm5}).

The author was stimulated by the computations in \cite{B} and \cite{GG}. The author would like to thank Prof.\  Y.\  Goto and Prof.\  N.\  Takayama for intriguing discussions, thank Prof.\  Y.\   Nozaki and Prof.\  S.\  Wakatsuki for daily discussion, and thank Prof.\  H.\  Sakai for his constant encouragement during the preparation of this paper. 

\underline{General Notation}

For
$v=
\begin{pmatrix}
v_1\\
\vdots\\
v_n\\
\end{pmatrix}
\in\C^n$,
 $A=({\bf a}(1)|\cdots|{\bf a}(N))\in M(n,N;\mathbb{Q})$, $x=(x_1,\cdots,x_n),\; y=(y_1,\dots,y_n)\in\C^n,$ and a complex valued univariate function $F(t)$, we put

\begin{equation*}
|v|=v_1+\cdots +v_n,
\end{equation*}

\begin{equation*}
xy=(x_1y_1,\dots,x_ny_n),
\end{equation*}

\begin{equation*}
x^v=x_1^{v_1}\cdots x_n^{v_n},
\end{equation*}

\begin{equation*}
x^A=(x^{{\bf a}(1)},\cdots,x^{{\bf a}(N)}),
\end{equation*}

\begin{equation*}
F(v)=F(v_1)\times\cdots\times F(v_n).
\end{equation*}

\noindent For example, we write $\Gamma\left(
\begin{pmatrix}
v_1\\
v_2\\
\end{pmatrix}
\right)
=\Gamma(v_1)\Gamma(v_2).$

For any subset $\s$ of $\{ 1,\dots, N\}$, we denote $A_\s$ the matrix consisting of column vectors $\{ {\bf a}(i)\}_{i\in\s}$. We often identify $\s$ with the corresponding set of column vectors $\{ {\bf a}(i)\}_{i\in\s}$. We also denote $\bs$ the complement $\{1,\dots,N\}\setminus\s$. 

For any natural numbers $m,n$, we denote $\Z^{m\times n}$ the set of $m\times n$ matrices with entries in $\Z$. More generally, for any finite sets $\s$ and $\tau$, we denote $\Z^{\s\times\tau}$ the set of matrices with their rows indexed by $\s$ and with their columns indexed by $\tau$. Especially, if $|\s|=n$ and $A\in\Z^{n\times \s}$, we often regard ${}^tA,A^{-1}\in\Z^{\s\times n}.$

The complex torus $\C^\times$ is written in two ways. We use $\Gm$ if it is equipped with Zariski topology, while we use $\T$ if it is equipped with the usual topology.

\vspace{2em}
\underline{Assumption}
Throughout this paper, we assume that the $A$ matrix of a GKZ system $M_A(c)$ satisfies $\Z A=\Z^{n\times 1}.$ 

\end{section}

\begin{section}{Review on $\Gamma$-series (after M.-C. Fern\'andez-Fern\'andez)}\label{sectionseries}
In this section, we review some basic facts on $\Gamma$-series solutions of GKZ hypergeometric functions. Basic references are \cite{FF} and \cite{SST}. First of all recall the following standard notion. Let us consider an Euclidean space $Y=\C^l$ and consider a subset $\tau\subset\{1,\dots,l\}$ with cardinality $r$. We will denote $z_{\tau}=(z_i)_{i\in\tau}$ and $\bar{\tau}=\{1,\cdots,l\}\setminus\tau$. We identify $\C^r$ with the subspace   $Y_{\bartau}=\{z\in\C^l\mid z_j=0\; (\forall j\in\tau)\}.$ Take a point $p\in Y_{\bartau}.$
\begin{defn}
A formal series 
\begin{equation}
\displaystyle\sum_{{\bf m}\in\Z_{\geq 0}^{\bar{\tau}}}f_{\bf m}(z_\tau)z_{\bar{\tau}}^{\bf m}\in\C\{ z_\tau-p\}[\![z_{\bar{\tau}}]\!]
\end{equation}
is said to be of Gevrey multi-order ${\bf s}=(s_j)_{j\in\bar{\tau}}\in\R^{\bar{\tau}}$ along $Y_{\bartau}$ at $p\in Y_{\bartau}$ if one has
\begin{equation}
\displaystyle\sum_{{\bf m}\in\Z_{\geq 0}^{\bar{\tau}}}\frac{f_{\bf m}(z_\tau)}{({\bf m}!)^{\bf s-1}}z_{\bar{\tau}}^{\bf m}\in\C\{ z_\tau-p,z_{\bar{\tau}}\}.
\end{equation}
\end{defn}

Let us consider Gevrey series solutions of GKZ system. For any vector $v\in\C^{N\times 1}$ such that $Av=-c,$ we put
\begin{equation}
\varphi_v(z)=\displaystyle\sum_{u\in L_A}\frac{z^{u+v}}{\Gamma(1+u+v)}
\end{equation}
and call it a $\Gamma$-series. It can readily be seen that $\varphi_{v}(z)$ is a formal solution of $M_A(c)$ (\cite{SST}). For any subset $\tau\subset\{1,\dots,N\}$, we denote $A_\tau$ the matrix given by the columns of $A$ indexed by $\tau.$ In the following, we take $\sigma\subset\{1,\dots,N\}$ such that $|\sigma|=n$ and $\det A_\sigma\neq 0.$
Taking a vector ${\bf k}\in\Z^{\bar{\sigma}},$ we put
\begin{equation}
v_\sigma^{\bf k}=
\begin{pmatrix}
-A_{\sigma}^{-1}(c+A_{\bar{\sigma}}{\bf k})\\
{\bf k}
\end{pmatrix}.
\end{equation}

\noindent
Then, by a direct computation, we have
\begin{equation}\label{seriesphi}
\varphi_{\s,{\bf k}}(z)\overset{\rm def}{=}\varphi_{v_\sigma^{\bf k}}(z)=z_\sigma^{-A_\sigma^{-1}{\bf c}}
\sum_{{\bf k+m}\in\Lambda_{\bf k}}\frac{(z_\sigma^{-A_\sigma^{-1}A_{\bar{\sigma}}}z_{\bar{\sigma}})^{\bf k+m}}{\Gamma({\bf 1}_\sigma-A_\sigma^{-1}({\bf c}+A_{\bar{\sigma}}({\bf k+m}))){\bf (k+m)!}},
\end{equation}
where $\Lambda_{\bf k}$ is given by
\begin{equation}\label{lambdak}
\Lambda_{\bf k}=\Big\{{\bf k+m}\in\Z^{\bar{\sigma}}_{\geq 0}\mid A_{\bar{\sigma}}{\bf m}\in\Z A_\sigma\Big\}.
\end{equation}
The following lemmata are very easily confirmed (\cite{FF}).

\begin{lem}
For any ${\bf k},{\bf k^\prime}\in\Z^{\bar{\sigma}}$, the following statements are equivalent

\begin{enumerate}
\item $v^{\bf k}-v^{\bf k^\prime}\in\Z^{N\times 1}$\\
\item $[A_{\barsigma} {\bf k}]=[A_{\barsigma} {\bf k^\prime}]$ in $\Z^{n\times 1}/\Z A_\sigma$\\
\item $\Lambda_{\bf k}=\Lambda_{\bf k^\prime}.$
\end{enumerate}

\end{lem}

\begin{lem}
Take a representative $\{ [A_{\barsigma}{\bf k}(i)]\}_{i=1}^r$ of a finite Abelian group $\Z^{n\times 1}/\Z A_\sigma.$ Then, we have a decomposition
\begin{equation}
\Z^{\barsigma}_{\geq 0}=\bigsqcup_{j=1}^{r}\Lambda_{{\bf k}(j)}.
\end{equation}
\end{lem}

\noindent
Note that we always assume $\Z A=\Z^{n\times 1}.$ 

Thanks to these lemmata, we can observe that $\{\varphi_{\sigma,{\bf k}(i)}(z)\}_{i=1}^r$ is a set of $r$ linearly independent formal solutions of $M_A(c)$ unless $\varphi_{\sigma,{\bf k}(i)}(z)=0$ for some $i$. In order to ensure the non-vanishing of $\varphi_{\s,{\bf k}(i)}$, we say that a parameter vector $c$ is very generic with respect to $\sigma$ if $A_\sigma^{-1}(c+A_{\bar{\sigma}}{\bf m})$ does not contain any integer entry for any ${\bf m}\in\mathbb{Z}_{\geq 0}^{\bar{\sigma}}.$ Using this terminology, we can rephrase the observation above as follows:

\begin{prop}\label{prop:independence}
If $c\in\C^{n\times 1}$ is very generic with respect to $\sigma$, 
\begin{equation}
\Big\{\varphi_{\s,{\bf k}(i)}\Big\}_{i=1}^r
\end{equation}
is a linearly independent set of formal solutions of $M_A(c)$.
\end{prop}

\noindent
Let us put
\begin{equation}
H_\sigma=\Big\{ y\in\R^n\mid |A_\sigma^{-1}y|=1\Big\},
\end{equation}
and
\begin{equation}
U_\sigma=\left\{z\in(\C^*)^N\mid |z_\sigma^{-A_\sigma^{-1}{\bf a}(j)}z_{j}|<R, \forall {\bf a}(j)\in H_\sigma\setminus\sigma\right\},
\end{equation}
where $R>0$ is a small positive number. We also put 
\begin{equation}
\barsigma_+=\Big\{j\in\barsigma\mid |A_\sigma^{-1}{\bf a}(j)|>1\Big\}.
\end{equation}
The following fact was observed in \cite{FF}.

\begin{prop}\label{prop:GevreyOrder}
Put $s_j=|A_\sigma^{-1}{\bf a}(j)|$ for $j\in\barsigma_+.$ Then, $\varphi_{\sigma,{\bf k}(i)}(z)$ is of Gevrey multi-order ${\bf s}=(s_j)_{j\in\barsigma_+}$ along $Y_{\barsigma_+}$ at any point $p\in U_\sigma\cap Y_{\barsigma_+}.$
\end{prop}

As is well-known in the literature, under a genericity condition, we can construct a basis of holomorphic solutions of  GKZ system $M_A(c)$ consisting of $\Gamma$-series with the aid of regular triangulation. Let us revise the definition of a regular triangulation. In general, for any subset $\sigma$ of $\{1,\dots,N\},$ we denote $\cone(\sigma)$ the positive span of $\{{\bf a}(1),\dots,{\bf a}(N)\}$ i.e., $\cone(\sigma)=\displaystyle\sum_{i\in\sigma}\R_{\geq 0}{\bf a}(i).$ We often identify a subset $\sigma\subset\{1,\dots,N\}$ with the corresponding set of vectors $\{{\bf a}(i)\}_{i\in\sigma}$ or with the set $c.h.\{0,\{{\bf a}(i)\}_{i\in\sigma}\}$. A collection $T$ of subsets of $\{1,\dots,N\}$ is called a triangulation if $\{\cone(\sigma)\mid \sigma\in T\}$ is the set of cones in a simplicial fan whose support equals $\cone(A)$. For any generic choice of a vector $\omega\in\R^{N\times 1},$ we can define a triangulation $T(\omega)$ as follows. A subset $\sigma\subset\{1,\dots,N\}$ belongs to $T(\omega)$ if there exists a vector ${\bf n}\in\R^{n\times 1}$ such that
\begin{align}
{\bf n}\cdot{\bf a}(i)=\omega_i &\text{ if } i\in\sigma\\
{\bf n}\cdot{\bf a}(j)<\omega_j &\text{ if } j\in\barsigma.
\end{align}
A triangulation $T$ is called a regular triangulation if $T=T(\omega)$ for some $\omega\in\R^{N\times 1}.$ Let us consider a regular triangulation $T$ such that for any element $\sigma\in T$, one has 
\begin{equation}
s_j=|A_\sigma^{-1}{\bf a}(j)|\leq 1\;\; (\forall j\in\barsigma).
\end{equation}
Note that such a regular triangulation always exists. Suppose that the parameter vector $c$ is very generic with respect to any $\sigma\in T$. Then, it was shown in \cite{FF} that we have $\rank M_A(c)=\vol_\Z(\Delta_A).$ Since $\vol_\Z(\Delta_A)=\displaystyle\sum_{\substack{\sigma\in T\\ |\sigma|=n}}\vol_\Z(\sigma),$ we can conclude that  

\begin{equation}
\displaystyle\bigcup_{\sigma\in T}\left\{ \varphi_{v_\sigma^{{\bf k}(i)}}\right\}_{i=1}^{r}
\end{equation}
is a basis of holomorphic solutions of $M_A(c)$ on $U_{T}\overset{def}{=}\displaystyle\bigcap_{\sigma\in T}U_\sigma\neq\varnothing$ where $r=\vol_\Z(\sigma)=|\Z/\Z A_\sigma|.$ 

\begin{rem}
We define an $N\times(N-n)$ matrix $B_\sigma$ by
\begin{equation}
B_\sigma=
\begin{pmatrix}
-A_\sigma^{-1}A_{\barsigma}\\
{\bf I}_{\barsigma}
\end{pmatrix}
\end{equation}
and a cone $C_\omega$ by
\begin{equation}
C_\omega=\Big\{ \omega\in\R^{N\times 1}\mid {}^t\omega\cdot B_\sigma>0\Big\}.
\end{equation}
Then, one can verify that $C_{T}\overset{def}{=}\displaystyle\bigcap_{\sigma\in T}C_\sigma$ is a non-empty open cone characterised by the formula
\begin{equation}
C_T=\Big\{ \omega\in\R^{N\times 1}\mid T(\omega)=T\Big\}.
\end{equation}
From the definition of $U_\sigma$, we can confirm that $z$ belongs to $U_T$ if $(\log|z_1|,\dots,\log|z_N|)$ belongs to a sufficiently far translation of $C_T$ inside itself, which implies $U_T\neq\varnothing.$
\end{rem}

\begin{rem}
It was obtained in \cite{Ho} and \cite{SW2} COROLLARY 3.16 that $M_A(c)$ is regular holonomic if and only if  there exists a linear function $l:\Q^n\rightarrow\Q$ such that $l({\bf a}(j))=1$ holds for any $j=1,\dots,N.$ Thus, when $M_A(c)$ is regular, any regular triangulation $T(\omega)$ and any simplex $\sigma\in T(\omega)$ gives only convergent series $\varphi_{v_\sigma^{{\bf k}(i)}}$ on $U_{T(\omega)}.$
\end{rem}

\end{section}

\begin{section}{Construction of integration contours I: Laplace integral representations}\label{SectionLaplace}
In this section, we construct a basis of integration cycles for Laplace integral representations
\begin{equation}\label{LaplaceInt2}
\frac{1}{(2\pi\ii)^{n+1}}\int_\Gamma e^{h_z(x)}x^{c-{\bf 1}}dx
\end{equation}
in a combinatorial way. Our construction is based on the idea of \cite{GG}. For this purpose, let us take any $\sigma\subset\{ 1,\cdots ,N \}$ such that $|\sigma|=n$, $\det A_\sigma\neq 0$, and $s_j=|A_\sigma^{-1}{\bf a}(j)|\leq 1$ for any $j\in\barsigma$. In view of \cref{prop:GevreyOrder}, this ensures that all $\Gamma$-series $\{\varphi_{v^{{\bf k}(i)}}(z)\}_{i=1}^r$ associated to the simplex $\sigma$ are convergent. We consider a covering map
\begin{equation}
\mathbb{T}^n_x\overset{p}{\rightarrow}\mathbb{T}^\s_{\xi_\sigma}
\end{equation}
given by
\begin{equation}
x\mapsto \xi_\sigma=z_\s x^{A_\sigma}.
\end{equation}
Here, we always reorder elements $\{ i_1,\dots,i_n\}$ of $\sigma$ so that $i_1<\dots<i_n$. By a straightforward computation, we have a transformation formula of volume forms:
\begin{equation}
\frac{dx}{x}=\frac{1}{\det (A_\sigma)}\frac{d\xi_\sigma}{\xi_\s}.
\end{equation}
Now, we are going to construct our integration cycle in $x$-space as a pull-back cycle of $\xi_\s$-space. In order to clarify the meaning of pull-back of a cycle, we need the following construction. Let $X,Y$ be oriented smooth $n$-dimensional manifolds and let $\pi:X\rightarrow Y$ be a covering map of degree $d$. Suppose that we are given a local system $\mathcal{L}$ on $Y$. We denote the Poincar\'e duality morphism by $\Phi_Y:\Homo^{n-p}_c(Y,\mathcal{L})\rightarrow\Homo_p(Y,\mathcal{L})$ and by $\Phi_X:\Homo^{n-p}_c(X,\pi^*\mathcal{L})\rightarrow\Homo_p(X,\pi^*\mathcal{L})$. Then the composition of the following morphisms is denoted by $\pi^*$:
\begin{equation}
\Homo_p(Y,\mathcal{L})\overset{\Phi_Y^{-1}}{\rightarrow}\Homo^{n-p}_c(Y,\mathcal{L})\overset{\pi^*}{\rightarrow}\Homo^{n-p}_c(X,\pi^*\mathcal{L})\overset{\Phi_X}{\rightarrow}\Homo_p(X,\pi^*\mathcal{L}).
\end{equation}
Note that the pull-back $\Homo^{n-p}_c(Y,\mathcal{L})\overset{\pi^*}{\rightarrow}\Homo^{n-p}_c(X,\pi^*\mathcal{L})$ is well-defined since $\pi$ is proper.

\begin{lem}
For any $p$-cycle $[\gamma]\in\Homo_p(Y,\mathcal{L})$ and for any $p$-cocycle $[\omega]\in\Homo^p(Y,\mathcal{L}^\vee)$, one has an identity
\begin{equation}
\int_{\pi^*\gamma}\pi^*\omega=d\int_\gamma\omega.
\end{equation}
\end{lem}

\begin{proof}
By the definition of Poincar\'e duality, for any $[\xi]\in\Homo^{n-p}_c(X,\pi^*\mathcal{L})$ and $[\eta]\in\Homo^p(X,\pi^*\mathcal{L}^\vee)$, we have 
\begin{equation}
\int_X\xi\wedge\eta=\int_{\Phi_X(\xi)}\eta.
\end{equation}
Therefore, we obtain a sequence of equalities
\begin{align}
\int_{\pi^*\gamma}\pi^*\omega&=\int_{\Phi_X\circ\pi^*\circ\Phi_Y^{-1}(\gamma)}\pi^*\omega\\
 &=\int_X\pi^*\Phi_Y^{-1}(\gamma)\wedge\pi^*\omega\\
 &=d\int_Y\Phi_Y^{-1}(\gamma)\wedge\omega\\
 &=d\int_\gamma\omega.
\end{align}

\end{proof}

\noindent
We turn back to the construction of integration cycle of (\ref{LaplaceInt2}). We do not specify our integration contour $\Gamma$ for the moment, but we suppose that it is a pull-back contour of some cycle $\gamma$ in $\T^\s_{\xi_\sigma}$, i.e., $\Gamma=p^*\gamma.$ By the construction of $\Gamma$, we have

\begin{align}
f_{\sigma,0}(z)&\overset{\rm def}{=}\frac{1}{(2\pi\sqrt{-1})^{n+1}}\int_\Gamma e^{h_z(x)}x^{\bf c-1}dx\\
 &=\frac{z_\sigma^{-A_\sigma^{-1}{\bf c}}}{(2\pi\sqrt{-1})^{n+1}}\int_{\gamma}\exp\{\sum_{i\in\sigma}\xi_i+\sum_{j\in\bar{\sigma}}z_\sigma^{-A_\sigma^{-1}{\bf a}(j)}z_j\xi_\sigma^{A_\sigma^{-1}{\bf a}(j)}\}\xi_\sigma^{A_\sigma^{-1}{\bf c}-{\bf 1}}d\xi_\sigma.
\end{align}
Now we apply the plane wave expansion formula. Let us introduce a new coordinate by
\begin{equation}
\xi_i=\rho u_i\;\;(i\in\sigma),
\end{equation}
where $\rho\in\C^*$ and $u_i$ are coordinates of $\{u_\sigma=(u_i)_{i\in\sigma}\in(\mathbb{C}^*)^\s\mid\displaystyle\sum_{i\in\sigma}u_i=1\}.$ Then, it is standard that we have an equality of volume forms 
\begin{equation}
d\xi_\sigma=\rho^{n-1}d\rho du_\sigma,
\end{equation}
where $du_\sigma=\displaystyle\sum_{k=1}^n(-1)^{k-1}u_kdu_{\widehat{i_k}}$ with $du_{\widehat{i_k}}=du_{i_1}\wedge\cdots\wedge \widehat{du_{i_k}}\wedge\cdots\wedge du_{i_n}$. In this new coordinate, we have

\begin{align}
f_{\sigma,0}(z)&=\frac{z_\sigma^{-A_\sigma^{-1}{\bf c}}}{(2\pi\sqrt{-1})^{n+1}}\int_{\gamma}\exp\{\rho+\sum_{j\in\bar{\sigma}}z_\sigma^{-A_\sigma^{-1}{\bf a}(j)}z_ju_\sigma^{A_\sigma^{-1}{\bf a}(j)}\rho^{|A_\sigma^{-1}{\bf a}(j)|}\}\times\nonumber\\
 &\quad \rho^{|A_\sigma^{-1}{\bf c}|-1}u_\sigma^{A_\sigma^{-1}{\bf c}-{\bf 1}_\sigma}d\rho du_\sigma.
\end{align}

\begin{figure}[t]
\begin{center}
\begin{tikzpicture}
\draw (0,0) node[below right]{O}; 
\draw[thick, ->] (-1.5,0)--(3.2,0) node[right]{$\xi_1$} ; 
\draw[thick, ->] (0,-1.2)--(0,3.2) node[above]{$\xi_2$} ; 
\draw[-] (3,-1)--(-1,3) node[left]{$\{ u_1+u_2=1\}$};
\draw[->] (-0.5,-0.5)--(3,3) node[right]{$\rho$};
\end{tikzpicture}
\caption{plane wave coordinate for $n=2$ and $\sigma=\{ 1,2\}$}
\end{center}
\end{figure}
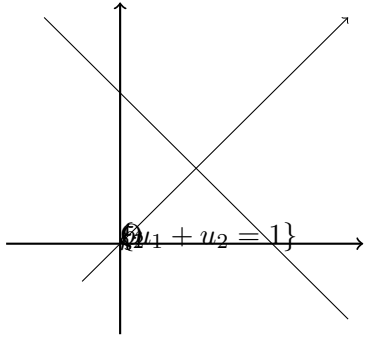

\noindent
At this moment, we can construct our integration cycle $\gamma$ in $(\rho,u_\sigma)$ coordinate as a product of a cycle in $\rho$ direction and that in $u_\sigma$ direction.

In $\rho$ direction, we take the so-called Hankel contour $\Gamma_0$. $\Gamma_0$ is given by the formula
$\Gamma_0=(-\infty,-\delta]e^{-\pi\sqrt{-1}}+l_{(0+)}-(-\infty,-\delta]e^{\pi\sqrt{-1}},$
where $e^{\pm\pi\sqrt{-1}}$ stands for the argument of the variable and $l_{(0+)}$ is a small loop which encircles the origin in the counter-clockwise direction starting from and ending at the point $-\delta.$ 
Using this notation, we have 

\begin{lem}\label{lemma:lemma}
Suppose $\alpha\in\mathbb{C}.$ Then we have

$$\int_{\Gamma_{0}}\xi^{\alpha-1}e^\xi d\xi=\frac{2\pi\sqrt{-1}}{\Gamma(1-\alpha)}.$$
\end{lem}

\noindent
The proof of this formula is straightforward from the definition of $\Gamma$ function and the reflection formula
\begin{equation}\label{ReflectionFormula}
\Gamma(\alpha)\Gamma(1-\alpha)=\frac{\pi}{\sin(\pi\alpha)}.
\end{equation}

As for the construction of a cycle in $u_\sigma$ direction, we need the multidimensional Pochhammer contour. We cite following lemma due to \cite{B}. 

\begin{lem}[\cite{B} Proposition 6.1]\label{lemma:Beukers}
Let $\Delta^k\subset\mathbb{R}^k$ be the k simplex $\Delta^k=\{(x_1,\cdots ,x_k)\in\mathbb{R}^k|x_1,\cdots,x_k\geq 0, \sum_{j=1}^kx_j\leq 1\}$ and $\alpha_1,\cdots ,\alpha_{k+1}\in\mathbb{C}$. If we denote by $P_k$ the Pochhammer cycle associated to $\Delta^k$ in the sense of \cite{B}, we have
\begin{equation}\label{DirichletInt}
\int_{P_k}t_1^{\alpha_1-1}\cdots t_{k}^{\alpha_k-1}(1-t_1\cdots -t_k)^{\alpha_{k+1}-1}dt_1\cdots dt_k=\prod_{j=1}^{k+1}(1-e^{-2\pi\sqrt{-1}\alpha_j})\frac{\Gamma(\alpha_1)\cdots\Gamma(\alpha_{k+1})}{\Gamma(\alpha_1+\cdots +\alpha_{k+1})}.
\end{equation}
\end{lem}
\noindent
For the construction of $P_k$, see \cite{B}. We only note that $P_k$ is a compact cycle which does not intersect with any hyperplane $\{ x_j=0\}$ $(j=1,\dots,k)$ and $\{ x_1+\cdots+x_k=1\}$. For later use, it is more convenient to rewrite (\ref{DirichletInt}) as
\begin{equation}\label{ConvenientFormula}
\int_{P_k}t_1^{\alpha_1-1}\cdots t_{k}^{\alpha_k-1}(1-t_1\cdots -t_k)^{\alpha_{k+1}-1}dt_1\cdots dt_k=\frac{e^{-\pi\sqrt{-1}(\alpha_1+\dots+\alpha_{k+1})}(2\pi\sqrt{-1})^{k+1}}{\Gamma(1-\alpha_1)\cdots\Gamma(1-\alpha_{k+1})\Gamma(\alpha_1+\cdots+\alpha_{k+1})},
\end{equation}
where we used (\ref{ReflectionFormula}).
\begin{figure}[t]
\begin{minipage}{0.5\hsize}
\begin{center}
\begin{tikzpicture}
\draw[->-=.5,domain=-175:175] plot ({-1+cos(\x)}, {sin(\x)});
\draw[-<-=.5] ({-1+cos(-175)},{sin(-175)}) -- (-6, {sin(-175)});
\draw[->-=.5] ({-1+cos(175)},{sin(175)}) -- (-6, {sin(175)});
\node at (-1,0){$\cdot$};
\draw (-1,0) node[below right]{O};
\end{tikzpicture}
\caption{Hankel contour}
\end{center}
\end{minipage}
\begin{minipage}{0.5\hsize}
\begin{center}
\begin{tikzpicture}
\node at (0,0){$\cdot$};
\node at (4,0){$\cdot$};
\draw (0,0) node[below]{$t_1=0$};
\draw (4,0) node[below]{$t_1=1$};
\draw[-<-=.5,domain=-140:140] plot ({4+cos(\x)}, {0.4+sin(\x)});
\draw[-<-=.5,domain=40:320] plot ({cos(\x)}, {0.4+sin(\x)});
\draw[->-=.5,domain=-140:140] plot ({4+cos(\x)}, {-0.4+sin(\x)});
\draw[->-=.5,domain=40:320] plot ({cos(\x)}, {-0.4+sin(\x)});
    \coordinate (A1) at ({cos(40)}, {0.4+sin(40)});
    \coordinate (A2) at ({4+cos(-140)}, {-0.4+sin(-140)});
    \coordinate (B1) at ({4+cos(140)}, {-0.4+sin(140)});
    \coordinate (B2) at ({cos(40)}, {-0.4+sin(40)});
    \coordinate (C1) at ({cos(320)}, {-0.4+sin(320)});
    \coordinate (C2) at ({4+cos(140)}, {0.4+sin(140)});
    \coordinate (D1) at ({4+cos(-140)}, {0.4+sin(-140)});
    \coordinate (D2) at ({cos(320)}, {0.4+sin(320)});
\draw[->-=.75] (A1) -- (A2);
\draw[->-=.5] (B1) -- (B2);
\draw[->-=.75] (C1) -- (C2);
\draw[->-=.5] (D1) -- (D2);
\end{tikzpicture}
\caption{Pochhammer cycle $P_1$}
\end{center}
\end{minipage}
\end{figure}

Now, choose our $\gamma$ in a product form $\gamma=\Gamma_{0}\times P_{u_\sigma},$ where $P_{u_\sigma}$ is the Pochhammer cycle in $\Big\{u_\sigma=(u_i)_{i\in\sigma}\in(\mathbb{C}^*)^\s\mid\displaystyle\sum_{i\in\sigma}u_i=1\Big\}.$ Then, it can be readily seen that when $z\in U_\sigma,$ our integral $f_0(z)$ is absolutely convergent since $s_j\leq 1$ for any $j\in\bs$. Moreover, expanding the term 
\begin{equation}
\exp\Bigg\{\sum_{j\in\bar{\sigma}}z_\sigma^{-A_\sigma^{-1}{\bf a}(j)}z_ju_\sigma^{A_\sigma^{-1}{\bf a}(j)}\rho^{|A_\sigma^{-1}{\bf a}(j)|}\Bigg\},
\end{equation}denoting ${\bf e}(i)$ the column vector ${}^t(0,\dots,\overset{i}{\overset{\smile}{1}},\dots,0)$, and using \cref{lemma:lemma} and the formula (\ref{ConvenientFormula}), we have

\begin{align}
f_{\sigma,0}(z)&=\frac{z_\sigma^{-A_\sigma^{-1}{\bf c}}}{(2\pi\sqrt{-1})^{n+1}}\int_{\gamma}\exp\Big\{\rho+\sum_{j\in\bar{\sigma}}z_\sigma^{-A_\sigma^{-1}{\bf a}(j)}z_ju_\sigma^{A_\sigma^{-1}{\bf a}(j)}\rho^{|A_\sigma^{-1}{\bf a}(j)|}\Big\}\times\nonumber\\
 & \quad \rho^{|A_\sigma^{-1}{\bf c}|-1}u_\sigma^{A_\sigma^{-1}{\bf c}-{\bf 1}_\sigma}d\rho du_\sigma\\
 &=\frac{z_\sigma^{-A_\sigma^{-1}{\bf c}}}{(2\pi\sqrt{-1})^{n+1}}\sum_{{\bf m}\in\mathbb{Z}_{\geq 0}^{\bar{\sigma}}}\frac{(z_\sigma^{-A_\sigma^{-1}A_{\bar{\sigma}}}z_{\bar{\sigma}})^{\bf m}}{\bf m!}\times\nonumber\\
 &\quad \int\int e^\rho\rho^{|A_\sigma^{-1}{\bf c}|+|A_\sigma^{-1}A_{\bar{\sigma}}{\bf m}|-1}u_\sigma^{A_\sigma^{-1}{\bf c}+A_\sigma^{-1}A_{\bar{\sigma}}{\bf m}-{\bf 1}_\sigma}d\rho du_\sigma\\
 &=\frac{z_\sigma^{-A_\sigma^{-1}{\bf c}}}{(2\pi\sqrt{-1})^{n+1}}\sum_{{\bf m}\in\mathbb{Z}_{\geq 0}^{\bar{\sigma}}}\frac{(z_\sigma^{-A_\sigma^{-1}A_{\bar{\sigma}}}z_{\bar{\sigma}})^{\bf m}}{\bf m!}\frac{2\pi\sqrt{-1}}{\det A_\sigma\Gamma(1-|A_\sigma^{-1}{\bf c}|-|A_\sigma^{-1}A_{\bar{\sigma}}{\bf m}|)}\times\nonumber\\
 & \quad \frac{(2\pi\sqrt{-1})^ne^{-\pi\sqrt{-1}\left(|A^{-1}_\sigma c|+|A_\sigma^{-1}A_{\barsigma}{\bf m}|\right)}}{\Gamma({\bf 1}_\sigma-A_\sigma^{-1}(c+A_{\barsigma}{\bf m}))\Gamma(|A_\sigma^{-1}{\bf c}|+|A_\sigma^{-1}A_{\bar{\sigma}}{\bf m}|)}\\
 &=z_\sigma^{-A_\sigma^{-1}{\bf c}}
\sum_{{\bf m}\in\mathbb{Z}_{\geq 0}^{\bar{\sigma}}}\frac{(z_\sigma^{-A_\sigma^{-1}A_{\bar{\sigma}}}z_{\bar{\sigma}})^{\bf m}}{\Gamma({\bf 1}_\sigma-A_\sigma^{-1}({\bf c}+A_{\bar{\sigma}}{\bf m})){\bf m!}}
(1-e^{-2\pi\sqrt{-1}|A_\sigma^{-1}({\bf c}+A_{\bar{\sigma}}{\bf m})|})\\
 &=\sum_{i=1}^r(1-e^{-2\pi\sqrt{-1}|A_\sigma^{-1}({\bf c}+A_{\bar{\sigma}}{\bf k}(i))|})\varphi_{\s,{\bf k}(i)}(z).
\end{align}
We denote $\Gamma_{\sigma,0}$ the integration cycle defined by the relation $\Gamma_{\sigma,0}=p^*\gamma$. Precisely speaking our cycle $\Gamma_{\s,0}$ is non-compact so we have to proceed with some care to justify the argument above. Consider the $\rho$ coordinate as a level function $\rho:\T^n_x\ni x\mapsto \displaystyle\sum_{i\in\s}z_ix^{{\bf a}(i)}\in\C.$ We compactify $\T^n_x$ to a smooth projective variety $X$ so that $\rho$ can be prolonged to a meromorphic function, i.e., so that we have a commutative diagram
\begin{equation}
\xymatrix{
& \T^n_x \ar[r]^{\rho} \ar[d]_{\rm inclusion}&\C \ar[d]^{\rm inclusion}\\
& X   \ar[r]^{\tilde{\rho}}                             &\mathbb{P}^1.}
\end{equation}
By (a corollary of) Thom-Mather's 1st isotopy lemma (\cite{Ver}, (5.1) Corollaire), perturbing the Hankel contour, we can assume that $\rho$ is a trivial fiber bundle when restricted to $\Gamma_0$. Now, take a point $0<\rho_0\in\Gamma_0$. By rescalling, we may assume $\rho_0=1.$ Consider the restricted covering 
\begin{equation}
p:\left\{\sum_{i\in\s}z_ix^{{\bf a}(i)}=1\right\}\setminus\bigcup_{i\in\s}\{ z_ix^{{\bf a}(i)}=0\}\rightarrow\left\{\sum_{i\in\s}u_i=1\right\}\setminus\bigcup_{i\in\s}\{ u_i=0\}.
\end{equation} Since Pochhammer cycle $P_{u_\s}$ belongs to the homology group 
\begin{equation}
\Homo_{n-1}\left(\left\{\sum_{i\in\s}u_i=1\right\}\setminus\bigcup_{i\in\s}\{ u_i=0\};\underline{\C} u_\s^{A_\s^{-1}c}\right),
\end{equation} we can define its pull-back 
\begin{equation}
p^*(P_{u_\s})\in\Homo_{n-1}\left(\left\{\sum_{i\in\s}z_ix^{{\bf a}(i)}=1\right\}\setminus\bigcup_{i\in\s}\{ z_ix^{{\bf a}(i)}=0\};\underline{\C}x^c\right).
\end{equation}
Since $\rho$ is a trivial fiber bundle over $\Gamma_0$, we can prolong the lifted cycle $p^*(P_{u_\s})$ along $\Gamma_0$, which yields to the precise definition of our integration cycle $\Gamma_{\s,0}$. We can summerize the discussion above as a proposition.

\begin{prop}
Fix a subset $\sigma\subset\{ 1,\cdots ,N \}$ such that $|\sigma|=n$, $\det A_\sigma\neq 0$, and $s_j=|A_\sigma^{-1}{\bf a}(j)|\leq 1$ for any $j\in\barsigma$. Then, 
\begin{equation}\label{fundamentalformula1}
f_{\sigma,0}(z)\overset{def}{=}\frac{1}{(2\pi\sqrt{-1})^{n+1}}\int_{\Gamma_{\sigma,0}} e^{h_z(x)}x^{\bf c-1}dx\end{equation}
is absolutely convergent and for any parameter $c\in\C^{n\times 1},$ and we have an equality
\begin{equation}
f_{\sigma,0}(z)=\sum_{i=1}^r\Big(1-e^{-2\pi\sqrt{-1}|A_\sigma^{-1}({\bf c}+A_{\bar{\sigma}}{\bf k}(i))|}\Big)\varphi_{\s,{\bf k}(i)}(z).
\end{equation}
\end{prop}

Now, let us construct $r=|\Z^n/\Z A_\sigma|$ linearly independent cycles for Laplace integral representation of GKZ hypergeometric function (\ref{Laplace integral}). We are going to show that suitable deck transformations of the cycle above with respect to $p$ give $r$ linearly independent cycles. Let us take a vector $\tilde{\bf k}\in\Z^n$, and consider a deck transformation of $\Gamma_{\sigma,0}$ along $\tilde{\xi}_\sigma\mapsto e^{2\pi\sqrt{-1}\tilde{\bf k}}\tilde{\xi}_\sigma=(e^{2\pi\sqrt{-1}\tilde{k}_1}\tilde{\xi}_{1},\dots,e^{2\pi\sqrt{-1}\tilde{k}_n}\tilde{\xi}_{n}).$ We denote this cycle by $\Gamma_{\sigma,\tilde{\bf k}}$. From (\ref{fundamentalformula1}), (\ref{seriesphi}), and (\ref{lambdak}), we have 
\begin{align}
f_{\sigma,\tilde{\bf k}}(z)&\overset{def}{=}\frac{1}{(2\pi\ii)^{n+1}}\int_{\Gamma_{\sigma,\tilde{\bf k}}}e^{h_z(x)}x^{c-1}dx\\
 &=\frac{z_\sigma^{-A_\sigma^{-1}{\bf c}}}{(2\pi\sqrt{-1})^{n+1}}\int_{\Gamma_{0}\times P_{u_\sigma}}\exp\left\{\sum_{i\in\sigma}\xi_i+\sum_{j\in\bar{\sigma}}z_\sigma^{-A_\sigma^{-1}{\bf a}(j)}z_je^{2\pi\sqrt{-1}{}^t\tilde{\bf k}A_\sigma^{-1}{\bf a}(j)}\xi_\sigma^{A_\sigma^{-1}{\bf a}(j)}\right\}\times\nonumber\\
 & \quad e^{2\pi\sqrt{-1}{}^t\tilde{\bf k}A_\sigma^{-1}c}\xi_\sigma^{A_\sigma^{-1}{\bf c}-{\bf 1}}d\xi_\sigma\\
 &=e^{2\pi\sqrt{-1}{}^t\tilde{\bf k}A_\sigma^{-1}c}\sum_{i=1}^r\Big(1-e^{-2\pi\sqrt{-1}|A_\sigma^{-1}({\bf c}+A_{\bar{\sigma}}{\bf k}(i))|}\Big)e^{2\pi\sqrt{-1}{}^t\tilde{\bf k}A_\sigma^{-1}{\bf k}(i)}\varphi_{\s,{\bf k}(i)}(z).
\end{align}

\noindent
Under the assumption that 
$$|A_\sigma^{-1}({\bf c}+A_{\bar{\sigma}}{\bf k}(i))|\notin\mathbb{Z}\;\;(\forall i=1,\dots,r),$$
we are reduced to find suitable vectors $\tilde{\bf k}(1),\dots,\tilde{\bf k}(r)\in\Z^n$ such that the matrix 
\begin{equation}
\Big(e^{2\pi\sqrt{-1}{}^t\tilde{\bf k}(i)A_\sigma^{-1}{\bf k}(j)}\Big)_{i,j=1}^{r}
\end{equation}
is invertible in view of \cref{prop:independence}. The following observation is of fundamental importance though it is elementary.

\begin{lem}\label{lem:pairing}
The pairing $\langle\; ,\;\rangle:\mathbb{Z}^{n\times 1}/\mathbb{Z}{}^tA_\sigma\times\mathbb{Z}^{n\times 1}/\mathbb{Z}A_\sigma\rightarrow\mathbb{Q}/\mathbb{Z}$ defined by $\langle v,w\rangle={}^tvA_\sigma^{-1}w$ is a non-degenerate pairing of finite abelian groups, i.e., for any fixed $v\in\mathbb{Z}^{n\times 1}/\mathbb{Z}{}^tA_\sigma,$ if one has $<\langle v,w\rangle=0$ for all $w\in\mathbb{Z}^{n\times 1}/\mathbb{Z}A_\sigma,$ then we have $v=0$ and vice versa.
\end{lem}
\noindent
The proof of \cref{lem:pairing} is quite elementary because we know that there are invertible integer matrices $P,Q\in GL(n,\mathbb{Z})$ such that $P^{-1}A_\sigma Q=diag(1,\cdots ,1,d_1,\cdots ,d_l)$ where $d_1,\cdots ,d_l\in\mathbb{Z}^\times$ satisfy $d_1|d_2,\cdots ,d_{l-1}|d_l,$ and the lemma is obvious when $A_\sigma$ is replaced by $P^{-1}A_\sigma Q.$ Thanks to this lemma, we have the following

\begin{prop}
$\mathbb{Z}^{n\times 1}/\mathbb{Z}{}^tA_\sigma\tilde{\rightarrow}\widehat{\mathbb{Z}^{n\times 1}/\mathbb{Z}A_\sigma}$, where the isomorphism is induced from the pairing $\langle\; ,\;\rangle.$
\end{prop}

\begin{proof}
Since we have a group embedding $\phi:\mathbb{Q}/\mathbb{Z}\hookrightarrow\mathbb{C}^{*}$ defined by 
\begin{equation}
\phi(\alpha)=e^{2\pi\sqrt{-1}\alpha}\;\; (\alpha\in\Q/\Z),
\end{equation}
\noindent
the pairing $\langle\; ,\;\rangle$ induces an embedding
\begin{equation}
\mathbb{Z}^{n\times 1}/\mathbb{Z}{}^tA_\sigma\tilde{\rightarrow}\Hom_{\mathbb{Z}}(\mathbb{Z}^{n\times 1}/\mathbb{Z}A_\sigma,\mathbb{Q}/\mathbb{Z})\hookrightarrow\widehat{\mathbb{Z}^{n\times 1}/\mathbb{Z}A_\sigma}.
\end{equation}
Counting the number of elements, we have the proposition.
\end{proof}
\noindent
Thus, if we take a complete system of representatives $\{\tilde{\bf k}(i)\}_{i=1}^r$ of $\Z^n/\Z {}^tA_\sigma$, the matrix 
\begin{equation}
\frac{1}{r}\left( e^{2\pi\sqrt{-1}{}^t\tilde{\bf k}(i)A_\sigma^{-1}A_{\overline{\sigma}}{\bf k}(j)} \right)
\end{equation}is a unitary matrix by the orthogonality of irreducible characters.

\begin{rem}
For any finite abelian group $G$, we have an isomorphism 
\begin{equation}
\Hom (G,\mathbb{Q}/\mathbb{Z})\simeq\Ext^1(G,\mathbb{Z}).
\end{equation}
The pairing of \cref{lem:pairing} is induced naturally from this isomorphism.
\end{rem}

\noindent
In sum, we reached the following

\begin{thm}\label{thm:fundamentalthm1}
Take a regular triangulation $T$ of $A$ such that for any $\sigma\in T,$ one has $s_j=|A_\sigma^{-1}{\bf a}(j)|\leq 1$ for all $j\in\barsigma$. Assume that the parameter vector $c$ is very generic with respect to any $\sigma\in T$ and for any $i=1,\dots,r$, one has $|A_\sigma^{-1}({\bf c}+A_{\bar{\sigma}}{\bf k}(i))|\notin\Z$. Then, if one puts
\begin{equation}
f_{\sigma,\tilde{\bf k}(j)}(z)=\frac{1}{(2\pi\ii)^{n+1}}\int_{\Gamma_{\sigma,\tilde{\bf k}(j)}}e^{h_z(x)}x^{c-1}dx,
\end{equation}
$\bigcup_{\s\in T}\{ f_{\sigma,\tilde{\bf k}(j)}(z)\}_{j=1}^{r}$ is a basis of solutions of $M_A(c)$ on the non-empty open set $U_T$, where $\{\tilde{\bf k}(j)\}_{j=1}^{r}$ is a complete system of representatives of $\Z^{n\times 1}/\Z{}^tA_\sigma$. Moreover, for each $\sigma\in T,$ one has a transformation formula 
\begin{equation}
\begin{pmatrix}
f_{\sigma,\tilde{\bf k}(1)}(z)\\
\vdots\\
f_{\sigma,\tilde{\bf k}(r)}(z)
\end{pmatrix}
=
T_\sigma
\begin{pmatrix}
\varphi_{\sigma,{\bf k}(1)}(z)\\
\vdots\\
\varphi_{\sigma,{\bf k}(r)}(z)
\end{pmatrix}.
\end{equation}
Here, $T_\sigma$ is an invertible $r\times r$ matrix given by 
\begin{equation}
T_\sigma=\diag\Big( e^{2\pi\sqrt{-1}{}^t\tilde{\bf k}(i)A_\sigma^{-1}c}\Big)_{i=1}^{r}\Big(e^{2\pi\sqrt{-1}{}^t\tilde{\bf k}(i)A_\sigma^{-1}{\bf k}(j)}\Big)_{i,j=1}^{r}\diag\Big(1-e^{-2\pi\sqrt{-1}|A_\sigma^{-1}({\bf c}+A_{\bar{\sigma}}{\bf k}(j))|}\Big)_{j=1}^{r}.
\end{equation}
\end{thm}

\begin{exa}
We consider a matrix 
\begin{equation}
A=
\Big({\bf a}(1)\mid{\bf a}(2)\mid{\bf a}(3)\Big)
=
\begin{pmatrix}
1&0&-1\\
0&2&3
\end{pmatrix},
\end{equation}and a parameter
$c=
\begin{pmatrix}
c_1\\
c_2
\end{pmatrix}
\in\mathbb{C}^{2\times 1}
$.
\begin{figure}[t]
\begin{center}
\begin{tikzpicture}
\draw (0,0) node[below left]{O}; 
\draw[thick, ->] (-1.5,0)--(3.2,0) node[right]{} ; 
\draw[thick, ->] (0,-1.2)--(0,3.2) node[above]{} ; 
\draw (0,0) -- (1,0) node[below]{\;\;\;\;\;\;\;\;\;\;${\bf a}(1)=(1,0)$}-- (0,2) node[right]{${\bf a}(2)=(0,2)$} -- (-1,3) node[left]{${\bf a}(3)=(-1,3)$} -- cycle;
\node at (0.4,0.65){$\sigma_1$};
\node at (-0.35,1.8){$\sigma_2$};
\end{tikzpicture}
\end{center}
\end{figure}
\noindent
GKZ system associated to $A$ and $c$ is

\begin{numcases}{M_A(c):}
(z_1\partial_1-z_3\partial_3+c_1) \cdot f(z)&$=0$  \\
(2z_2\partial_2+3z_3\partial_3+c_2) \cdot f(z)&$=0$ \\
(\partial_2^3-\partial_1^2\partial_3^2) \cdot f(z)&$=0.$ 
\end{numcases}
\noindent
For any generic parameter $c$, the rank of $M_A(c)$ is $\vol_\mathbb{Z}\Delta_A=4.$ If we take any real vector $\omega\in\{\omega\in\R^3\mid \omega_1-\frac{3}{2}\omega_2+\omega_3>0,-\frac{3}{2}\omega_2+\omega_3+\omega_1>0\},$ $T(\omega)$ does not depend on the choice of $\omega$ and it consists of two simplexes $\sigma_1=\{ 1,2\}$ and $\sigma_2=\{ 2,3\}.$ Moreover, for this triangulation, we can easily confirm that $|A_{\sigma_1}^{-1}{\bf a}(3)|<1$ and $|A_{\sigma_2}{\bf a}(1)|<1$. We only consider $\Gamma$-series associated to $\sigma_1.$ The vectors $v^{k}$ $(k\in\Z)$ associated to $\sigma_1$ are given by 
\begin{equation}
v^{k}=
\begin{pmatrix}
-k\\
\frac{3}{2}k\\
k
\end{pmatrix}.
\end{equation}
Thus, $\Gamma$-series associated to $\sigma_1$ are

\begin{equation}
\varphi_{v^k}(z)=z_1^{-c_1}z_2^{-\frac{c_2}{2}}\displaystyle\sum_{m\in\Lambda_k}\frac{(z_1z_2^{-\frac{3}{2}}z_3)^m}{\Gamma(1-c_1+m)\Gamma(1-\frac{c_2}{2}-\frac{3}{2}m)m!},
\end{equation}
where
\begin{equation}
\Lambda_k=
\begin{cases}
2\mathbb{Z}_{\geq 0} & (k\equiv 0\mod 2)\\
2\mathbb{Z}_{\geq 0}+1 & (k\equiv 1\mod 2).
\end{cases}
\end{equation}
Note that these series are convergent on $(\C^*)^3.$ Next, we consider the Laplace integral representation
\begin{equation}\label{LapEx}
f(z)=\frac{1}{(2\pi\sqrt{-1})^3}\int_\Gamma e^{z_1x_1+z_2x_2^2+z_3x_1^{-1}x_2^2}x_1^{c_1-1}x_2^{c_2-1}dx_1\wedge dx_2.
\end{equation}
As before, we consider the covering transform associated to $\sigma_1$ given by
\begin{equation}
\xi_1=z_1x_1,\;\xi_2=z_2x_2^2.
\end{equation}
Applying the change of coordinate 
\begin{equation}
\xi_i=\rho u_i,\;\;u_1+u_2=1\;\;(i=1,2),
\end{equation}
we have
\begin{equation}\label{ex1}
f(z)=\frac{z_1^{-c_1}z_2^{-\frac{c_2}{2}}}{(2\pi\sqrt{-1})^3}\displaystyle\int_{\gamma} \exp\left\{ \rho+(z_1z_2^{-\frac{3}{2}}z_3)u_1^{-1}u_2^{\frac{3}{2}}\rho^{\frac{1}{2}}\right\}\rho^{c_1+\frac{c_2}{2}-1}u_1^{c_1-1}u_2^{\frac{c_2}{2}-1}d\rho du
\end{equation}
If we take $\gamma$ as a direct product $\Gamma_{0,1}\times P_1,$ the last integral $(\ref{ex1})$ is convergent and is equal to
\begin{equation}
(1-e^{-2\pi\sqrt{-1}(c_1+\frac{c_2}{2})})\varphi_{v^0}+(1+e^{-2\pi\sqrt{-1}(c_1+\frac{c_2}{2})})\varphi_{v^1}.
\end{equation}
As before, we denote $\Gamma_{\sigma_1,0}$ the pull back of $\gamma$ in $x$-space. Relation between $(\ref{LapEx})$ and $\Gamma$-series is given by
\begin{equation}
\begin{pmatrix}
f_{\sigma_1,0}(z)\\
f_{\sigma_1,1}(z)
\end{pmatrix}
=
\begin{pmatrix}
1&0\\
0&e^{\pi\sqrt{-1}c_2}
\end{pmatrix}
\begin{pmatrix}
1&1\\
1&-1
\end{pmatrix}
\begin{pmatrix}
(1-e^{-2\pi\sqrt{-1}(c_1+\frac{c_2}{2})})&0\\
0&(1+e^{-2\pi\sqrt{-1}(c_1+\frac{c_2}{2})})
\end{pmatrix}
\begin{pmatrix}
\varphi_{\sigma_1,0}(z)\\
\varphi_{\sigma_1,1}(z)
\end{pmatrix}.
\end{equation}
\end{exa}

\end{section}

\begin{section}{Construction of integration contours I\hspace{-.1em}I: Residue integral representations}\label{SectionResidue}


\noindent
In this section, we consider the following integral representation and relate it to series solutions in such a way analogous to \S\ref{SectionLaplace}:
\begin{equation}\label{RegResInt}
f(z)=\int\frac{y^{\gamma -1}x^{c-1}}{(1-y_1h_{1,z^{(1)}}(x))\cdots (1-y_kh_{k,z^{(k)}}(x))}dydx,
\end{equation}
where $h_{l,z^{(l)}}(x)$ are Laurent polynomials. The reason we treat Residue integrals before we treat Euler integrals lies on the fact that the construction of integration contours for Residue integrals turns out to be more symmetric than that for Euler integrals, as we shall see in \S\ref{SectionEuler}. Note first that Residue integral is indeed a solution of GKZ system. We introduce an $n\times N_l$ matrix $A_l=({\bf a}^{(l)}(1)\mid \cdots\mid {\bf a}^{(l)}(N_l))$, and an $(n+k)\times N$ $(N=N_1+\dots+N_k)$ matrix
\begin{equation}
A
=
\left(
\begin{array}{ccc|ccc|c|ccc}
1&\cdots&1&0&\cdots&0&\cdots&0&\cdots&0\\
\hline
0&\cdots&0&1&\cdots&1&\cdots&0&\cdots&0\\
\hline
 &\vdots& & &\vdots& &\ddots& &\vdots& \\
\hline
0&\cdots&0&0&\cdots&0&\cdots&1&\cdots&1\\
\hline
 &A_1& & &A_2& &\cdots & &A_k& 
\end{array}
\right).
\end{equation}
Putting $d=\begin{pmatrix}\gamma\\ c\end{pmatrix}$, we obtain a basic
\begin{prop}\label{prop:ResIntIsSol}
(\ref{RegResInt}) is a solution of $M_A(d).$
\end{prop}

\begin{proof}
First, let us observe that (\ref{RegResInt}) satisfies (\ref{EulerEq}). We have, for any $(\tau,t)\in(\mathbb{C}^*)^k\times (\mathbb{C}^*)^n$,
\begin{align}
f((\tau, t)^Az)&=\int\frac{y^{\gamma -1}x^{c-1}}{(1-y_1h_{1,\tau_1t^{A_1}z^{(1)}}(x))\cdots (1-y_kh_{k,\tau_kt^{A_k}z^{(k)}}(x))}dydx \\
 &=\int\frac{y^{\gamma -1}x^{c-1}}{(1-\tau_1y_1h_{1,z^{(1)}}(t\cdot x))\cdots (1-\tau_k y_kh_{k,z^{(k)}}(t\cdot x))}dydx\\
 &=\tau^{-\gamma}t^{-c}\int\frac{y^{\gamma -1}x^{c-1}}{(1-y_1h_{1,z^{(1)}}(x))\cdots (1-y_kh_{k,z^{(k)}}(x))}dydx.
\end{align}
Applying $(\frac{\partial}{\partial\tau_l})\restriction_{(\tau,t)=({\bf 1},{\bf 1})}$ or $(\frac{\partial}{\partial t_i})\restriction_{(\tau,t)=({\bf 1},{\bf 1})}$ to both sides, we obtain (\ref{EulerEq}).

Next, we verify (\ref{ultrahyperbolic}). For any $u\in \Z^N\simeq\Z^{N_1}\oplus\dots\oplus\Z^{N_k}$, we decompose it as a sum  $u=u^{(1)}+\cdots+u^{(k)}$ so that $u_l\in\Z^{N_l}$. Then, for any $u\in\Z_{\geq 0}^N$, we have 
\begin{equation}
\partial^uf(z)=|u^{(1)}|!\cdots |u^{(k)}|!\int \frac{y_1^{\gamma_1+|u^{(1)}|-1}\cdots y_k^{\gamma_k+|u^{(k)}|-1}x^{c+A_1u^{(1)}+\cdots+A_ku^{(k)} -1}}{(1-y_1h_{1,z^{(1)}}(x))^{|u^{(1)}|+1}\cdots (1-y_kh_{k,z^{(k)}}(x))^{|u^{(k)}|+1}}dydx.
\end{equation}
By the definition of $A$, if we denote by $\left\{ 
\begin{pmatrix}
{\bf e}_1\\
\hline 
O
\end{pmatrix},\cdots,
\begin{pmatrix}
{\bf e}_k\\
\hline 
O
\end{pmatrix},
\begin{pmatrix}
O\\
\hline
{\bf e}_1
\end{pmatrix},\cdots, 
\begin{pmatrix}
0\\
\hline
{\bf e}_n
\end{pmatrix}\right\}$ the standard basis of $\mathbb{Z}^{k\times 1}\times\mathbb{Z}^{n\times 1},$ we have, for any $u\in L_A,$ 
\begin{equation}
0=\transp{
\begin{pmatrix}
{\bf e}_l\\
\hline 
O
\end{pmatrix}
}
\cdot Au=|u^{(l)}|.
\end{equation}
This also implies an equality
\begin{equation}
|u_+^{(l)}|=|u_-^{(l)}|.
\end{equation}
Since $\displaystyle\sum_{l=1}^kA_lu_+^{(l)}=\displaystyle\sum_{l=1}^kA_lu_-^{(l)}$ for any $u\in L_A$, we have
\begin{equation}
\partial^{u_+}f(z)=\partial^{u_-}f(z)
\end{equation}
for any $u\in L_A$.
\end{proof}

Now, we are going to construct linearly independent cycles which correspond to series solutions as in \S\ref{SectionLaplace}. For this purpose, let us fix any $n+k$ simplex $\sigma\subset\{ 1,\dots,N\}$ and put $\tilde{x}=(y,x)$. From the form of the matrix $A$, we automatically have $s_j\leq 1\;(j\in\bs)$ unlike \S\ref{SectionLaplace}. Then, we put
\begin{equation}
f_{\s,0}(z)=\frac{1}{(2\pi\sqrt{-1})^{n+2k}}\int_{\Gamma_{\s,0}}\frac{\tilde{x}^{d-1}}{(1-\tilde{h}_{1,z^{(1)}}(\tilde{x}))\cdots (1-\tilde{h}_{k,z^{(k)}}(\tilde{x}))}d\tilde{x},
\end{equation}
where, by abuse of notation, $\tilde{h}_{l,z^{(l)}}(\tilde{x})$ denotes a Laurent polynomial

\begin{equation}
\tilde{h}_{l,z^{(l)}}(\tilde{x})=y_l\sum_{j=1}^{N_l}z_{j}^{(l)}x^{{\bf a}^{(l)}(j)}=\sum_{j\in I_l}z_{j}\tilde{x}^{{\bf a}(j)},
\end{equation}
where $I_l$ is a set of indices, and $\Gamma_{\s,0}$ is a cycle to be specified later. As in  \S\ref{SectionLaplace}, consider a covering map
\begin{equation}
\mathbb{T}^{n+k}_{\tilde{x}}\overset{p}{\rightarrow}\mathbb{T}^{\s}_{\xi_\sigma}
\end{equation}
given by
\begin{equation}
\tilde{x}\mapsto \xi_\sigma=z_\s\tilde{x}^{A_\sigma}.
\end{equation}
\noindent
Then, putting $\sigma^{(l)}=\sigma\cap I_l,$ and $\bs^{(l)}=I_l\setminus\s^{(l)}$ and assuming $\Gamma_{\s,0}=p^*\gamma$ for some $\gamma\in\Homo_{n+k}(\T^\s_{\xi_\s};\C\xi_\s^{A_\s^{-1}d})$, we have 

\begin{equation}
f_{\s,0}(z)=\frac{z_\sigma^{-A_\sigma^{-1}d}}{(2\pi\sqrt{-1})^{n+2k}}\int_\gamma
\frac{\xi_\sigma^{A_\sigma^{-1}d-{\bf 1}}}
{
\displaystyle\prod_{l=1}^k
\left(
1-\sum_{i\in\sigma^{(l)}}\xi_i-\sum_{j\in\bar{\sigma}^{(l)}}(z_\sigma^{-A_\sigma^{-1}{\bf a}(j)}z_j)\xi_\sigma^{A_\sigma^{-1}{\bf a}(j)}
\right)
}d\xi_\sigma\label{res1}.
\end{equation}
\noindent
Note that the integral well-defined and convergent if $z\in U_\s$. Let us take our integration contour $\gamma$ as a product form:
\begin{equation}
\gamma=P_\sigma\overset{\rm def}{=}\displaystyle\prod_{l=1}^kP_{\sigma^{(l)}},
\end{equation}
\noindent
where $P_{\s^{(l)}}$ is a Pochhammer cycle in $\T^{\s^{(l)}}_{\xi_{\s^{(l)}}}$. Substituting the expansion

\begin{equation}
\frac{1}
{
\displaystyle
\left(
1-\sum_{i\in\sigma^{(l)}}\xi_i-\sum_{j\in\bar{\sigma}^{(l)}}(z_\sigma^{-A_\sigma^{-1}{\bf a}(j)}z_j)\xi_\sigma^{A_\sigma^{-1}{\bf a}(j)}
\right)
}
=
\sum_{{\bf m}_l\in\mathbb{Z}^{\bar{\sigma}_l}_{\geq 0}}
\frac{|{\bf m}_l|!}{{\bf m}_l!}
\frac{
\xi_\sigma^{A_\sigma^{-1}A_{\bar{\sigma}^{(l)}}{\bf m}_l}
(z_\sigma^{-A_\sigma^{-1}A_{\bar{\sigma}^{(l)}}}z_{\bar{\sigma}^{(l)}})^{{\bf m}_l}
}
{
\left(
\displaystyle
1-\sum_{i\in\sigma^{(l)}}\xi_i
\right)^{|{\bf m}_l|+1}
},
\end{equation}
into (\ref{res1}), we can conclude that our integral $f(z)$ is given by the formula

\begin{equation}
f(z)=\frac{z_\sigma^{-A_\sigma^{-1}d}}{(2\pi\sqrt{-1})^{n+2k}}\displaystyle\sum_{{\bf m}\in\Z_{\geq 0}^{\barsigma}}\frac{|{\bf m}_1|!\cdots|{\bf m}_k|!}{{\bf m}!}\prod_{l=1}^k{\rm I}_l({\bf m}_l)(z_\sigma^{-A_\sigma^{-1}A_{\bar{\sigma}}}z_{\bar{\sigma}})^{\bf m},
\end{equation}
where ${\rm I}_l({\bf m}_l)$ is given by the formula

\begin{equation}
{\bf I}_l({\bf m}_l)=
\int_{P_{\sigma^{(l)}}}
\frac{
\displaystyle
\prod_{i\in\sigma^{(l)}}
\xi_i^{{}^t{\bf e}_iA_\sigma^{-1}(d+A_{\bar{\sigma}}{\bf m})-1}
}
{
\left(
\displaystyle
1-\sum_{i\in\sigma^{(l)}}\xi_i
\right)^{|{\bf m}_l|+1}
}
d\xi_{\sigma^{(l)}}.
\end{equation}
Note that $\tilde{\bf e}_i$ above is the standard basis vector of $\mathbb{Z}^\sigma\simeq\mathbb{Z}^{(k+n)\times 1}.$ A direct computation employing (\ref{ConvenientFormula}) yields to

\begin{equation}
I_l({\bf m}_l)=
\frac{
(2\pi\sqrt{-1})^{|\sigma^{(l)}|+1}\exp\left\{-\pi\sqrt{-1}\Bigg(\displaystyle\sum_{i\in\sigma^{(l)}}{}^t{\bf e}_iA_\sigma^{-1}({\bf d}+A_{\bar{\sigma}}{\bf m})-|{\bf m}_l|\Bigg)\right\}
}
{
\displaystyle{ \prod_{i\in\sigma^{(l)}}\Gamma\Big(1-{}^t{\bf e}_iA_\sigma^{-1}({\bf d}+A_{\bar{\sigma}}{\bf m})\Big)\Gamma(1+|{\bf m}_l|)\displaystyle\Gamma\Bigg(\sum_{i\in\sigma^{(l)}}{}^t{\bf e}_iA_\sigma^{-1}({\bf d}+A_{\bar{\sigma}}{\bf m})-|{\bf m}_l|\Bigg)
}
}
\end{equation}
Now, let us prove the following

\begin{lem}\label{lem:sum}
For any $j\in\bar{\sigma}$, one has

\begin{equation}
\sum_{i\in\sigma^{(l)}}{}^t{\bf e}_iA_\sigma^{-1}{\bf a}(j)=
\begin{cases}
1\; (j\in\bar{\sigma}^{(l)})\\
0\; (j\notin\bar{\sigma}^{(l)}).
\end{cases}
\end{equation}
\end{lem}

\begin{proof}
Observe first that, if we write $A$ as $A=({\bf a}(1)|\cdots|{\bf a}(N)),$ then for any $j\in\bar{\sigma}^{(l)},$ we have
\begin{equation}
\transp{
\begin{pmatrix}
{\bf e}_m\\
\hline
O
\end{pmatrix}
}
{\bf a}(j)=
\begin{cases}
1\; (m=l)\\
0\; (m\neq l).
\end{cases}
\end{equation}

\noindent
This can be written as 
\begin{equation}
\left(
\begin{array}{c|c}
I_k& \\
\hline
 &\mathbb{O}_n
\end{array}
\right)
{\bf a}(j)=
\begin{pmatrix}
{\bf e}_l\\
\hline
O
\end{pmatrix}.
\end{equation}

\noindent
We thus have

\begin{align}
\begin{pmatrix}
{\bf e}_l\\
\hline
O
\end{pmatrix}
&=
\left(
\begin{array}{c|c}
I_k& \\
\hline
 &\mathbb{O}_n
\end{array}
\right)
{\bf a}(j)\\
 &=
\left(
\begin{array}{c|c}
I_k& \\
\hline
 &\mathbb{O}_n
\end{array}
\right)
A_\sigma
A_\sigma^{-1}
{\bf a}(j)\\
 &=
\left(
\begin{array}{ccc|ccc|c|ccc}
1&\cdots&1&0&\cdots&0&\cdots&0&\cdots&0\\
\hline
0&\cdots&0&1&\cdots&1&\cdots&0&\cdots&0\\
\hline
 &\vdots& & &\vdots& &\ddots& &\vdots& \\
\hline
0&\cdots&0&0&\cdots&0&\cdots&1&\cdots&1\\
\end{array}
\right)
A_\sigma^{-1}
{\bf a}(j).
\end{align}

\noindent
The formula above clearly shows the lemma.

\end{proof}
\noindent
Thanks to the lemma above, we have

\begin{equation}
I_l({\bf m}_l)=
\frac{
(2\pi\sqrt{-1})^{|\sigma^{(l)}|+1}\exp\left\{-\pi\sqrt{-1}\Bigg(\displaystyle\sum_{i\in\sigma^{(l)}}{}^t{\bf e}_iA_\sigma^{-1}{\bf d}\Bigg)\right\}
}
{
\displaystyle{ \prod_{i\in\sigma^{(l)}}\Gamma\Big(1-{}^t{\bf e}_iA_\sigma^{-1}({\bf d}+A_{\bar{\sigma}}{\bf m})\Big)\displaystyle\Gamma\Bigg(\sum_{i\in\sigma^{(l)}}{}^t{\bf e}_iA_\sigma^{-1}{\bf d}\Bigg)|{\bf m}_l|!
}
}
\end{equation}

Therefore, we have

\begin{align}
f_{\s,0}(z)&=\frac{
e^{-\pi\sqrt{-1}|A_\sigma^{-1}d|}
}
{
\displaystyle
\prod_{l=1}^k
\Gamma\left(
\sum_{i\in\sigma^{(l)}}{}^t{\bf e}_iA_\sigma^{-1}d
\right)
}
z_\sigma^{-A_\sigma^{-1}d}
\sum_{{\bf m}\in\mathbb{Z}^{\bar{\sigma}}}
\frac{
(z_\sigma^{-A_\sigma^{-1}A_{\bar{\sigma}}}z_{\bar{\sigma}})^{{\bf m}}
}
{
\Gamma({\bf 1}-A_\sigma^{-1}(d+A_{\bar{\sigma}}{\bf m})){\bf m}!
}\\
&=\frac{
e^{-\pi\sqrt{-1}|A_\sigma^{-1}d|}
}
{
\displaystyle
\prod_{l=1}^k
\Gamma\left(
\sum_{i\in\sigma^{(l)}}{}^t{\bf e}_iA_\sigma^{-1}d
\right)
}\displaystyle\sum_{i=1}^r\varphi_{\s,{\bf k}(i)}(z).\label{ResExpansion}
\end{align}

As in \S\ref{SectionLaplace}, if we denote the deck transformation of $\Gamma_{\sigma,0}$ associated to $\xi_\sigma\mapsto e^{2\pi\sqrt{-1}\tilde{\bf k}}\xi_\sigma$ $(\tilde{\bf k}\in\Z^{n+k})$ by $\Gamma_{\sigma,\tilde{\bf k}}$, we obtain an analogue of \cref{thm:fundamentalthm1}.

\begin{thm}\label{thm:fundamentalthm2}
Take a regular triangulation $T$ of $A$. Assume that the parameter vector $d$ is very generic with respect to any $\sigma\in T$ and that for any $l=1,\dots,k$, one has $\displaystyle\sum_{i\in\sigma^{(l)}}{}^t{\bf e}_iA_\sigma^{-1}d\notin\Z_{\leq 0}$. Then, if one puts
\begin{equation}
f_{\sigma,\tilde{\bf k}(j)}(z)=\frac{1}{(2\pi\sqrt{-1})^{n+2k}}\int_{\Gamma_{\s,\tilde{\bf k}(j)}}\frac{y^{\gamma -1}x^{c-1}}{(1-y_1h_{1,z^{(1)}}(x))\cdots (1-y_kh_{k,z^{(k)}}(x))}dydx,
\end{equation}
$\bigcup_{\s\in T}\{ f_{\sigma,\tilde{\bf k}(j)}(z)\}_{j=1}^{r}$ is a basis of solutions of $M_A(d)$ on the non-empty open set $U_T$, where $\{\tilde{\bf k}(j)\}_{j=1}^{r}$ is a complete system of representatives $\Z^{n+k}/\Z{}^tA_\sigma$. Moreover, for each $\sigma\in T,$ one has a transformation formula 
\begin{equation}
\begin{pmatrix}
f_{\sigma,\tilde{\bf k}(1)}(z)\\
\vdots\\
f_{\sigma,\tilde{\bf k}(r)}(z)
\end{pmatrix}
=
T_\sigma
\begin{pmatrix}
\varphi_{\sigma,{\bf k}(1)}(z)\\
\vdots\\
\varphi_{\sigma,{\bf k}(r)}(z)
\end{pmatrix}.
\end{equation}
Here, $T_\sigma$ is an $r\times r$ matrix given by 
\begin{equation}
T_\sigma=\frac{
e^{-\pi\sqrt{-1}|A_\sigma^{-1}d|}
}
{
\displaystyle
\prod_{l=1}^k
\Gamma\left(
\sum_{i\in\sigma^{(l)}}{}^t{\bf e}_iA_\sigma^{-1}d
\right)
}
\diag\Big( e^{2\pi\sqrt{-1}{}^t\tilde{\bf k}(i)A_\sigma^{-1}d}\Big)_{i=1}^{r}
\Big(e^{2\pi\sqrt{-1}{}^t\tilde{\bf k}(i)A_\sigma^{-1}A_{\bs}{\bf k}(j)}\Big)_{i,j=1}^{r}.
\end{equation}
\end{thm}

\begin{exa}
We consider Residue integral
\begin{equation}\label{ResEx}
f(z)=\frac{1}{(2\pi\sqrt{-1})^4}\int_\Gamma \frac{y^{\gamma-1}x_1^{c_1-1}x_2^{c_2-1}}{1-y(z_0+z_1x_1+z_2x_2^2+z_3x_1^{-1}x_2^2)}dy\wedge dx_1\wedge dx_2.
\end{equation}
A matrix of (\ref{ResEx}) is given by a $3\times 4$ matrix
\begin{equation}
A
=
\begin{pmatrix}
1&1&1&1\\
0&1&0&-1\\
0&0&2&3
\end{pmatrix},
\end{equation}and a parameter vector $d$ is 
$d=
\begin{pmatrix}
\gamma\\
c_1\\
c_2
\end{pmatrix}
\in\mathbb{C}^{3\times 1}
$.

\noindent
GKZ system associated to $A$ and $d$ is

\begin{numcases}{M_A(d):}
(z_0\partial_0+z_1\partial_1+z_2\partial_2+z_3\partial_3+\gamma)\cdot f(z)&$=0$\\
(z_1\partial_1-z_3\partial_3+c_1) \cdot f(z)&$=0$ \\
(2z_2\partial_2+3z_3\partial_3+c_2) \cdot f(z)&$=0$ \\
(\partial_0\partial_2^3-\partial_1^2\partial_3^2) \cdot f(z)&$=0.$
\end{numcases} 
For any generic parameter $d$, the rank of $M_A(d)$ is $\vol_\mathbb{Z}\Delta_A=4.$ If we take any real vector $\omega\in\{\omega\in\R^4\mid -\frac{\omega_0}{2}+\omega_1-\frac{3}{2}\omega_2+\omega_3>0,-\frac{\omega_0}{2}-\frac{3}{2}\omega_2+\omega_3+\omega_1>0\},$ $T(\omega)$ does not depend on the choice of $\omega$ and it consists of two simplexes $\sigma_1=\{ 0,1,2\}$ and $\sigma_2=\{ 0,2,3\}.$ We only consider $\Gamma$-series associated to $\sigma_1.$ They are given by
\begin{equation}
\varphi_{w^k}(z)=z_0^{c_1+\frac{c_2}{2}-\gamma}z_1^{-c_1}z_2^{-\frac{c_2}{2}}\displaystyle\sum_{m\in\Lambda_k}\frac{(z_0^{-\frac{1}{2}}z_1z_2^{-\frac{3}{2}}z_3)^m}{\Gamma(1-\gamma+c_1+\frac{c_2}{2}-\frac{m}{2})\Gamma(1-c_1+m)\Gamma(1-\frac{c_2}{2}-\frac{3}{2}m)m!},
\end{equation}
where
\begin{equation}
\Lambda_k=
\begin{cases}
2\mathbb{Z}_{\geq 0} & (k\equiv 0\mod 2)\\
2\mathbb{Z}_{\geq 0}+1 & (k\equiv 1\mod 2).
\end{cases}
\end{equation}
Note that these series are convergent on $(\C^*)^3.$ As before, we consider the covering transform associated to $\sigma_1$ given by
\begin{equation}
\xi_0=z_0y,\;\; \xi_1=z_1yx_1,\;\xi_2=z_2yx_2^2.
\end{equation}
Relation between $(\ref{ResEx})$ and $\Gamma-$series is given by
\begin{equation}
\begin{pmatrix}
f_{\sigma_1,0}(z)\\
f_{\sigma_1,1}(z)
\end{pmatrix}
=
\frac{e^{-\pi\sqrt{-1}\gamma}}{\Gamma(\gamma)}
\begin{pmatrix}
1&0\\
0&e^{\pi\sqrt{-1}c_2}
\end{pmatrix}
\begin{pmatrix}
1&1\\
1&-1
\end{pmatrix}
\begin{pmatrix}
\varphi_{\sigma_1,0}(z)\\
\varphi_{\sigma_1,1}(z)
\end{pmatrix}.
\end{equation}
\end{exa}

\begin{rem}
This example can be considered as a non-confluent version of the example we saw before. The readers should notice the resemblance of all computations examined above.
\end{rem}
\end{section}

\begin{section}{Construction of integration contours I\hspace{-.1em}I\hspace{-.1em}I: Euler integral representations}\label{SectionEuler}
In this section, we construct integration contours associated to Euler integral representation
\begin{equation}\label{EulerInt2}
f(z)=\frac{1}{(2\pi\ii)^{n+k}}\int h_{1,z^{(1)}}(x)^{-\gamma_1}\cdots h_{k,z^{(k)}}(x)^{-\gamma_k}x^{c-{\bf 1}}dx,
\end{equation}
where $h_{l,z^{(l)}}(x)=\displaystyle\sum_{j=1}^{N_l}z^{(l)}_jx^{{\bf a}^{(l)}(j)}$ is a Laurent polynomial. Before starting a concrete discussion, let us first remark that we can always transform Euler integral representation to Laplace integral representation via Cayley trick. By \cref{lemma:lemma}, we have
\begin{equation}
h_{l,z^{(l)}}(x)^{-\gamma_l}=\frac{\Gamma(1-\gamma_l)}{2\pi\sqrt{-1}}\int_{\Gamma_{0,1}}y_l^{\gamma_l-1}e^{y_lh_{l,z^{(l)}}(x)} dy_l
\end{equation}
as long as this integral is absolutely convergent. Thus, ignoring the problem of convergence, we have the following formal equation
\begin{equation}
f(z)=\frac{\Gamma(1-\gamma)}{(2\pi\sqrt{-1})^{n+k}}\int \exp\left\{\sum_{l=1}^ky_lh_{l,z^{(l)}}(x)\right\}y^{\gamma-{\bf 1}}x^{c-{\bf 1}}dydx.
\end{equation}
\noindent
This suggests that if we introduce an $n\times N_l$ matrix $A_l=({\bf a}^{(l)}(1)\mid \cdots\mid {\bf a}^{(l)}(N_l))$, $f(z)$ is a solution of GKZ system $M_A(d)$ associated to an $n\times N$ $(N=N_1+\dots+N_k)$ matrix
\begin{equation}
A
=
\left(
\begin{array}{ccc|ccc|c|ccc}
1&\cdots&1&0&\cdots&0&\cdots&0&\cdots&0\\
\hline
0&\cdots&0&1&\cdots&1&\cdots&0&\cdots&0\\
\hline
 &\vdots& & &\vdots& &\ddots& &\vdots& \\
\hline
0&\cdots&0&0&\cdots&0&\cdots&1&\cdots&1\\
\hline
 &A_1& & &A_2& &\cdots & &A_k& 
\end{array}
\right),
\end{equation}
\noindent
and a parameter 
\begin{equation}
d=
\begin{pmatrix}
\gamma\\
c
\end{pmatrix}.
\end{equation}
\noindent
We examine the observation above from the view point of $D-$module in \S\ref{SectionEquivalence}. Let us remark, however, that it can directly be confirmed that Euler integral (\ref{EulerInt2}) is annihilated by $M_A(d)$.

\begin{prop}[\cite{GKZ} 2.7. Theorem]\label{prop:GKZ}
(\ref{EulerInt2}) is a solution of $M_A(d).$
\end{prop}
\noindent
The proof of \cref{prop:GKZ} is a direct computation. 

In the following, we construct explicit integration contours for (\ref{EulerInt2}). It is important to note that the construction below is less symmetric than the previous two constructions. This asymmetry comes from the fact that the dimension of the integration variable of (\ref{EulerInt2}) is smaller than the number of rows of the $A$ matrix of the GKZ system. For this reason, we can not directly construct a covering transformation associated to a simplex unlike the previous integral representations (\ref{LaplaceInt2}) and (\ref{RegResInt}). In order to clarify the method of breaking the symmetry, we divide the discussion into three parts. Case 1 deals with the case when $k=1$ and the $A$ matrix has a very convenient form for explicit computations . Case 2 deals with the case when $k=1$ and the $A$ matrix has a general form. Case 3 deals with the most general form (\ref{EulerInt2}). 

\vspace{2em}
\noindent
\underline{Case 1} We begin with the following simple situation: Suppose (\ref{EulerInt2}) takes the form
\begin{equation}\label{SuperSimpleEulerInt}
f_{\s,0}(z_0,z)=\frac{1}{(2\pi\sqrt{-1})^{n+1}}\int_\Gamma h_z(x)^{-\gamma}x^{c-1}dx=\frac{1}{(2\pi\sqrt{-1})^{n+1}}\int_\Gamma \left( z_0+\sum_{j=1}^Nz_jx^{\ca(j)}\right)^{-\gamma}x^{c-1}dx.
\end{equation}
Here, $\Gamma$ is an integration contour which will be specified later. If we put $\mathring{A}=(\ca(1)\mid \cdots \mid\ca(N)),$ the $A$ matrix is given by
\begin{equation}\label{SuperSimpleMatrix}
A
=
\left(
\begin{array}{c|ccc}
1&1&\cdots&1\\
\hline
O& &\mathring{A}& 
\end{array}
\right).
\end{equation}

\noindent
Suppose now an $n+1$ simplex $\sigma\subset\{0,1,\dots,N\}$ is given so that $0\in\sigma$. We put $\tau=\sigma\setminus\{0\}.$ We rewrite (\ref{SuperSimpleEulerInt}) as
\begin{align}
f_{\s,0}(z_0,z)&=\frac{z_0^{-\gamma}}{(2\pi\sqrt{-1})^{n+1}}\int_\Gamma \left( 1+\sum_{j=1}^Nz_0^{-1}z_jx^{\ca(j)}\right)^{-\gamma}x^{c-1}dx\\
 &=\frac{z_0^{-\gamma}}{(2\pi\sqrt{-1})^{n+1}}\int_\Gamma \left( 1+\sum_{j=1}^Nw_jx^{\ca(j)}\right)^{-\gamma}x^{c-1}dx.
\end{align}
Now, we introduce a covering transformation 
\begin{equation}
\mathbb{T}^n_x\overset{p}{\rightarrow}\mathbb{T}^\tau_{\xi_\tau}
\end{equation}
given by
\begin{equation}
x\mapsto \xi_\tau=(e^{-\pi\sqrt{-1}}w_ix^{\ca(i)})_{i\in\tau}.
\end{equation}
This formula is abbreviated as 
\begin{equation}
\xi_\tau=e^{-\pi\sqrt{-1}{}^t{\bf 1}_\tau}w_\tau x^{\cA_\tau}.
\end{equation}
Not that we have put $\bartau=\bs$. As in \S\ref{SectionLaplace} and \S\ref{SectionResidue}, assuming that $\Gamma$ is given by a pull back $\Gamma=p^*\gamma$, we have a formula
\begin{align}
f_{\s,0}(z_0,z)&=\frac{z_0^{-\gamma}}{(2\pi\sqrt{-1})^{n+1}}(e^{-\pi\sqrt{-1}{}^t{\bf 1}_\tau}w_\tau)^{\cA_\tau^{-1}c}\times\nonumber\\
 & \int_\gamma\left(1-\sum_{i\in\tau}\xi_i+\sum_{j\in\bartau}(e^{-\pi\sqrt{-1}{}^t{\bf 1}_\tau}w_\tau)^{-\cA_\tau^{-1}\ca(j)}w_j\xi_\tau^{\cA_\tau^{-1}\ca(j)}\right)^{-\gamma}\xi_\tau^{\cA_\tau^{-1}c-1}d\xi_\tau.\label{AnInt}
\end{align}
At this stage, we choose our integration contour $\gamma$ to be the Pochhammer cycle $P_\tau$. Note that (\ref{AnInt}) is convergent if the quantity 
\begin{equation}\label{ConvCond}
|(e^{-\pi\sqrt{-1}{}^t{\bf 1}_\tau}w_\tau)^{-\cA_\tau^{-1}\ca(j)}w_j|=|z_0^{|\cA_\tau^{-1}{\bf a}(j)|-1}z_\tau^{-\cA_\tau^{-1}\ca(j)}z_j|
\end{equation}
is small enough. Substituting the expansion 
\begin{align}
  &\left(1-\sum_{i\in\tau}\xi_i+\sum_{j\in\bartau}(e^{-\pi\sqrt{-1}{}^t{\bf 1}_\tau}w_\tau)^{-\cA_\tau^{-1}\ca(j)}w_j\xi^{\cA_\tau^{-1}\ca(j)}\right)^{-\gamma}\\
 =&\sum_{{\bf m}\in\Z_{\geq 0}^{\bartau}}\frac{(-1)^{|{\bf m}|}(\gamma)_{|{\bf m}|}}{{\bf m}!}\left((e^{-\pi\sqrt{-1}{}^t{\bf 1}_\tau}w_\tau)^{-\cA_\tau^{-1}\ca(j)}w_j\xi^{\cA_\tau^{-1}\ca(j)}\right)^{\bf m}\left(1-\sum_{i\in\tau}\xi_i\right)^{-\gamma-|{\bf m}|},
\end{align}
we have identities
\begin{align}
f_{\s,0}(z_0,z)&=\frac{1}{(2\pi\ii)^{n+1}}z_0^{-\gamma}e^{\pi\sqrt{-1}|\cA_\tau^{-1}c|}w_\tau^{-\cA_\tau^{-1}c}\sum_{\bf m}\frac{(-1)^{|{\bf m}|}(\gamma)_{|{\bf m}|}}{{\bf m}!}e^{\pi\ii|\cA_\tau^{-1}\cA_{\bt}{\bf m}|}\times\nonumber\\
 & \quad z_0^{|\cA_\tau^{-1}\cA_{\bartau}{\bf m}|-|{\bf m}|}(z_\tau^{-\cA_\tau^{-1}\cA_{\bartau}}z_{\bartau})^{\bf m}\int\left(1-\sum_{i\in\tau}\xi_i\right)^{-\gamma-|{\bf m}|}\xi_\tau^{\cA_\tau^{-1}(c+\cA_{\bartau}{\bf m})-{\bf 1}}d\xi_\tau\\
 &=z_0^{-\gamma}e^{\pi\sqrt{-1}|\cA_\tau^{-1}c|}w_\tau^{-\cA_\tau^{-1}c}\sum_{\bf m}\frac{(-1)^{|{\bf m}|}(\gamma)_{|{\bf m}|}}{{\bf m}!}e^{\pi\ii|\cA_\tau^{-1}\cA_{\bt}{\bf m}|}z_0^{|\cA_\tau^{-1}\cA_{\bartau}{\bf m}|-|{\bf m}|}(z_\tau^{-\cA_\tau^{-1}\cA_{\bartau}}z_{\bartau})^{\bf m}\times\nonumber\\
 & \quad \frac{e^{-\pi\sqrt{-1}(1-\gamma+|\cA_\tau^{-1}c|+|\cA_\tau^{-1}\cA_{\bartau}{\bf m}|-|{\bf m}|)}}{\Gamma({\bf 1}-\cA_\tau^{-1}(c+\cA_{\bartau}{\bf m}))\Gamma(\gamma+|{\bf m}|)\Gamma(1-\gamma+|\cA_\tau^{-1}c|+|\cA_\tau^{-1}\cA_{\bartau}\bm|-|\bm|)\bm !}\\
 &=\frac{e^{-\pi\ii(1-\gamma)}}{\Gamma(\gamma)}z_0^{|\cA_\tau^{-1}c|-\gamma}z_\tau^{-\cA_\tau^{-1}c}\times\nonumber\\
 & \quad\sum_{\bf m}
\frac{z_0^{|\cA_\tau^{-1}\cA_{\bartau}{\bf m}|-|{\bf m}|}(z_\tau^{-\cA_\tau^{-1}\cA_{\bartau}}z_{\bartau})^{\bf m}}{\Gamma({\bf 1}_\tau-\cA_\tau^{-1}(c+\cA_{\bartau}{\bf m}))\Gamma(1-\gamma+|\cA_\tau^{-1}c|+|\cA_\tau^{-1}\cA_{\bartau}\bm|-|\bm|)\bm !}.\label{SimpleSum}
\end{align}
Now, let us compute the explicit form of $\Gamma$-series associated to the simplex $\sigma$. By a direct computation, we have
\begin{equation}
(A_{\s})^{-1}=
\left(
\begin{array}{c|c}
1&-{}^t{\bf 1}_\tau \cA_\tau^{-1}\\
\hline
O&\cA_\tau^{-1}
\end{array}
\right)
\end{equation}
and
\begin{equation}
(A_{\s})^{-1}A_{\bs}=
\begin{pmatrix}
{}^t{\bf 1}_{\bs}-{}^t{\bf 1}_\tau \cA_\tau^{-1}\cA_{\bt}\\
\hline
\cA_\tau^{-1}\cA_{\bt}
\end{pmatrix}.
\end{equation}
Therefore, in view of (\ref{seriesphi}), for any ${\bf k}\in\Z^{\bt}(=\Z^{\bs}),$ we have a formula
\begin{align}
 \varphi_{\s,\bf k}(z_0,z)=&z_0^{|\cA_\tau^{-1}c|-\gamma}z_\tau^{-A_\tau^{-1}c}\sum_{{\bf k+m}\in\Lambda_{\bf k}}\frac{z_0^{|A_\tau^{-1}A_{\bartau}({\bf k+m})|-|{\bf k+m}|}(z_\tau^{-A_\tau^{-1}A_{\bartau}}z_{\bartau})^{\bf k+m}}{({\bf k}+\bm) !}\times\nonumber\\
 & \frac{1}{\Gamma\left({\bf 1}_\tau-A_\tau^{-1}(c+A_{\bartau}({\bf k+m}))\right)\Gamma\left(1-\gamma+|A_\tau^{-1}c|+|A_\tau^{-1}A_{\bartau}({\bf k}+\bm)|-|{\bf k}+\bm|\right)},\label{SimpleSeries}
\end{align}
where the set $\Lambda_{\bf k}$ in this case is given by
\begin{equation}
\Lambda_{\bf k}=\{{\bf k+m}\in\Z_{\geq 0}^{\bs}\mid A_{\bs}\bm\in\Z A_{\s}\}
=\{{\bf k+m}\in\Z_{\geq 0}^{\bt}\mid \cA_{\bt}\bm\in\Z \cA_{\tau}\}.
\end{equation}
Note that the projection $\pi:\Z^{(n+1)\times1}\ni
\begin{pmatrix}
v_0\\
v
\end{pmatrix}
\mapsto v\in\Z^{n\times 1}$ induces an isomorphism of finite abelian groups
\begin{equation}
\Z^{(n+1)\times 1}/\Z \cA_\tau\overset{\sim}{\rightarrow}\Z^{n\times 1}/\Z \cA_\tau
\end{equation}

\noindent
Thus, combining (\ref{SimpleSum}), (\ref{SimpleSeries}), and the identity $\det A_{\s}=\det \cA_\tau$, we obtain the basic formula
\begin{equation}\label{SimpleFormula}
f_{\s,0}(z_0,z)=\frac{e^{-\pi\ii(1-\gamma)}}{\Gamma(\gamma)}\sum_{i=1}^r\varphi_{\s,{\bf k}(i)}(z_0,z),
\end{equation}
where ${\bf k}(1),\dots,{\bf k}(r)\in\Z^{\bs}$ are chosen so that we have identities
\begin{equation}
\left\{ [A_{\bs}{\bf k}(i)]\right\}_{i=1}^r=\Z^{(n+1)\times 1}/\Z A_{\s}\simeq\Z^{n\times 1}/\Z \cA_\tau=\{ [\cA_{\bt}{\bf k}(i)]\}_{i=1}^r.
\end{equation}
Note that from the computation above, the convergence condition (\ref{ConvCond}) of Euler integral (\ref{SuperSimpleEulerInt}) is equivalent to 
\begin{equation}
(z_0,z)\in U_\s.
\end{equation}
As in \S\ref{SectionLaplace} and \S\ref{SectionResidue}, we take any integer vector $\tilde{\bf k}\in\Z^{\tau}$ and consider a deck transformation $\Gamma_{\tilde{\bf k}}$ of $\Gamma_0$ associated to $\xi_\tau\mapsto e^{2\pi\ii{}^t\tilde{\bf k}}\xi_\tau=(e^{2\pi\ii\tilde{k}_i}\xi_i)_{i\in\tau}.$
Then, from (\ref{AnInt}), we have a formula
\begin{align}
f_{\s,\tilde{\bf k}}(z)=&\frac{1}{(2\pi\ii)^{n+1}}\int_{\Gamma_{\tilde{\bf k}}}\left( z_0+\sum_{j=1}^Nz_jx^{\ca(j)}\right)^{-\gamma}x^{c-{\bf 1}}dx\nonumber\\
  =&\frac{e^{-\pi\ii(1-\gamma)}e^{2\pi\ii{}^t\tilde{\bf k}\cA_\tau^{-1}c}}{\Gamma(\gamma)}\sum_{i=1}^r\exp\{ 2\pi\ii{}^t\tilde{\bf k}\cA_\tau^{-1}\cA_{\bt}{\bf k}(i)\}\varphi_{\s,{\bf k}(i)}(z_0,z).\label{DeckTransEuler}
\end{align}
Observe that the map $\Z^{n\times1}/\Z \transp{\cA_\tau}\ni \tilde{\bf k}\mapsto
\begin{pmatrix}
0\\
\tilde{\bf k}
\end{pmatrix}
\in\Z^{(n+1)\times 1}/\Z\transp{A_{\s}}
$ is an isomorphism. We can rewrite the formula (\ref{DeckTransEuler}) in the form

\begin{equation}
f_{\s,\tilde{\bf k}}(z_0,z)=\frac{e^{-\pi\ii(1-\gamma)}\exp\left\{2\pi\ii
\transp{
\begin{pmatrix}
0\\
\tilde{\bf k}
\end{pmatrix}
}
(A_{\s})^{-1}d
\right\}
}
{
\Gamma(\gamma)
}
\sum_{i=1}^r\exp\left\{ 2\pi\ii
\transp{
\begin{pmatrix}
0\\
\tilde{\bf k}
\end{pmatrix}
}
(A_{\s})^{-1}A_{\bs}{\bf k}(i)\right\}\varphi_{\s,{\bf k}(i)}(z_0,z),
\end{equation}
which only contains the information of $A$. By \cref{lem:pairing}, if we take $\tilde{\bf k}(1),\dots,\tilde{\bf k}(r)$ so that we have identities 
\begin{equation}
\left\{ \left[\begin{pmatrix}
0\\
\tilde{\bf k}(i)
\end{pmatrix}\right]\right\}_{i=1}^r=\Z^{(n+1)\times 1}/\Z {}^tA_{\s}\simeq\Z^{n\times 1}/\Z {}^t\cA_\tau=\{ [\tilde{\bf k}(i)]\}_{i=1}^r,
\end{equation}
and if $\gamma\notin\Z_{\leq 0}$, we can conclude that $\{f_{\s,\tilde{\bf k}(i)}\}_{i=1}^r$ is a set of $r$ linearly independent solutions of $M_{A}(d)$.

\vspace{2em}
\noindent
\underline{Case 2} Now, consider a slightly more general case. Namely, suppose our A matrix is of the form
\begin{equation}
A
=
\left(
\begin{array}{c|ccc}
1&1&\cdots&1\\
\hline
{\bf a}(0)& &\cA& 
\end{array}
\right),
\end{equation}
and we are given a $n+1$ simplex $\sigma$ which contains $0$. Our integral is given by the formula
\begin{equation}
f_{\s,0}(z_0,z)=\frac{1}{(2\pi\ii)^{n+1}}\int_\Gamma\left( z_0x^{\ca(0)}+\sum_{j=1}^Nz_jx^{\ca(j)}\right)^{-\gamma}x^{c-1}dx.
\end{equation}
Observe that, rewriting $f_{\s,0}(z_0,z)$ as
\begin{equation}
f_{\s,0}(z_0,z)=\frac{1}{(2\pi\ii)^{n+1}}\int\left( z_0+\sum_{j=1}^Nz_jx^{\ca(j)-\ca(0)}\right)^{-\gamma}x^{c-\gamma\ca(0)-1}dx
\end{equation}
amounts to considering an identity
\begin{equation}
M_{A}(d)=M_{Q_0A}(Q_0d),
\end{equation}
where $Q_0$ is a $n+1$ square matrix 
\begin{equation}
Q_0=
\left(
\begin{array}{c|c}
1&O\\
\hline
-\ca(0)& I_n 
\end{array}
\right).
\end{equation}
We put $\tau=\s\setminus\{ 0\}$. Since the matrix 
\begin{equation}
Q_0A_{\s}=
\left(
\begin{array}{c|ccc}
1&1&\cdots&1\\
\hline
O& &\cA_\tau-\ca(0)\cdot {}^t{\bf 1}_\tau& 
\end{array}
\right)
\end{equation}
has the same form as (\ref{SuperSimpleMatrix}), putting $w_j=z_0^{-1}z_j$ and introducing a covering transformation
\begin{equation}
\T^n_x\ni x\mapsto \xi_\tau=e^{-\pi\ii{}^t{\bf 1}_\tau}w_\tau x^{\cA_\tau-\ca(0)\cdot{}^t{\bf 1}_\tau}\in\T^\tau_{\xi_\tau}
\end{equation}
yield to the identity
\begin{align}
f_{\s,0}(z_0,z)&=\frac{z_0^{-\gamma}(e^{-\pi\sqrt{-1}{}^t{\bf 1}_\tau}w_\tau)^{-(\cA_\tau-\ca(0)\cdot{}^t{\bf 1}_\tau)^{-1}(c-\gamma\ca(0))}}{(2\pi\sqrt{-1})^{n+1}}\times\nonumber\\
 & \quad \int_\gamma\left(1-\sum_{i\in\tau}\xi_i+\sum_{j\in\barsigma}(e^{-\pi\sqrt{-1}{}^t{\bf 1}_\tau}w_\tau)^{-(\cA_\tau-\ca(0)\cdot{}^t{\bf 1}_\tau)^{-1}(\ca(j)-\ca(0))}w_j\xi_\tau^{(\cA_\tau-\ca(0)\cdot{}^t{\bf 1}_\tau)^{-1}(\ca(j)-\ca(0))}\right)^{-\gamma}\times\nonumber\\
 & \quad\xi_\tau^{(A_\tau-\ca(0)\cdot{}^t{\bf 1}_\tau)^{-1}(c-\gamma\ca(0))-1}d\xi_\tau.\label{AnInt2}
\end{align}
Here, we have again assumed that $\Gamma$ is given by a pull back $\Gamma=p^*\gamma$. Note that $\det A_{\s}=\det (A_\s-\ca(0)\cdot{\bf 1}_\tau)$. At this stage, taking $\gamma$ to be the Pochhammer contour $P_\s$, we are able to evaluate (\ref{AnInt2}) as in Case 1. We should be careful about the fact that from the general form of $\Gamma$-series associated to a simplex (\ref{seriesphi}), $\varphi_{\s,{\bf k}}(z)$ is invariant if we replace the matrix $A$ by $Q_0A$ and the parameter $d$ by $Q_0d$. Then, by (\ref{SimpleFormula}), we have a basic formula
\begin{equation}
f_{\s,0}(z_0,z)=\frac{e^{-\pi\ii(1-\gamma)}}{\Gamma(\gamma)}\sum_{i=1}^r\varphi_{\s,{\bf k}(i)}(z_0,z),
\end{equation}
where ${\bf k}(i)$ are chosen so that the identity
\begin{equation}
\{ (\cA_{\bs}-\ca(0)\cdot{}^t{\bf 1}_{\bs}){\bf k}(i)\}_{i=1}^r=\Z^{n\times 1}/\Z (\cA_{\tau}-\ca(0)\cdot{}^t{\bf 1}_{\tau})
\end{equation}
holds. We also know that 
\begin{align}
\Lambda_{\bf k}&=\{ {\bf k+m}\in\Z_{\geq 0}^{\bs}\mid A_{\barsigma}{\bf m}\in \Z A_{\sigma}\}\\
&=\{ {\bf k+m}\in\Z_{\geq 0}^{\bs}\mid Q_0A_{\barsigma}{\bf m}\in \Z Q_0A_{\sigma}\}\\
&=\{ {\bf k+m}\in\Z^{\bs}_{\geq 0}\mid (\cA_{\bs}-\ca(0)\cdot{\bf 1}_{\bs}){\bf m}\in \Z (\cA_{\tau}-\ca(0)\cdot{\bf 1}_\tau)\}.
\end{align}
We can also check that $\{A_{\bs}{\bf k}(i)\}_{i=1}^r=\Z^{(n+1)\times 1}/\Z A_{\s}$ if and only if $\{ (\cA_{\bs}-\ca(0)\cdot{}^t{\bf 1}_{\bs}){\bf k}(i)\}_{i=1}^r=\Z^{n\times 1}/\Z (\cA_{\tau}-{\bf a}(0)\cdot{}^t{\bf 1}_{\tau})$. On the level of finite abelian groups, we have isomorphisms
\begin{equation}
\Z^{(n+1)\times 1}/\Z A_{\s}\overset{Q_0\times}{\overset{\sim}{\rightarrow}}\Z^{(n+1)\times 1}/\Z Q_0A_{\s}\overset{\pi}{\overset{\sim}{\rightarrow}}\Z^{n\times 1}/\Z (A_{\tau}-\ca(0)\cdot{}^t{\bf 1}_{\tau}),
\end{equation}
and the corresponding isomorphisms of dual groups
\begin{equation}\label{DualIsom}
\Z^{(n+1)\times 1}/\Z {}^tA_{\s}=\Z^{(n+1)\times 1}/\Z {}^tA_{\s}\cdot{}^tQ_0\overset{P\times}{\overset{\sim}{\rightarrow}}\Z^{(n+1)\times 1}/\Z P{}^tA_{\s}\cdot{}^tQ_0\overset{\pi}{\overset{\sim}{\rightarrow}}\Z^{n\times 1}/\Z {}^t(\cA_{\tau}-\ca(0)\cdot{}^t{\bf 1}_{\tau}),
\end{equation}
where $\pi$ is a projection which truncates the first coordinate and P is an invertible $\Z$ matrix
\begin{equation}
P=
\left(
\begin{array}{c|ccc}
1&0&\cdots&0\\
\hline
-1&1& & \\
\vdots& &\ddots& \\
-1& & &1
\end{array}
\right).
\end{equation}
These isomorphisms preserve the duality pairing in the following sense: for any elements $[v]\in \Z^{(n+1)\times 1}/\Z {}^tA_{\s}$ and $[w]\in\Z^{(n+1)\times 1}/\Z A_{\s},$ we have an identity
\begin{equation}
{}^tvA_{\s}^{-1}w=\pi(Pv)(\cA_\tau-\ca(0)\cdot{}^t{\bf 1}_\tau)^{-1}\pi(Q_0w)\mod \Z.
\end{equation}
Indeed,
\begin{align}
{}^tvA_{\s}^{-1}w&={}^t(Pv)(Q_0A_{\s}{}^tP)^{-1}(Q_0w)\\
                                            &={}^t(Pv)
\left(
\begin{array}{c|c}
1& \\
\hline
 &(\cA_\tau-\ca(0)\cdot{}^t{\bf 1}_\tau)^{-1}
\end{array}
\right)
(Q_0w)\\
                                            &=\pi(Pv)(\cA_\tau-\ca(0)\cdot{}^t{\bf 1}_\tau)^{-1}\pi(Q_0w)\mod \Z.
\end{align}
Moreover, the inverse of the isomorphism (\ref{DualIsom}) is given by
\begin{equation}
\Z^{n\times 1}/\Z {}^t(A_{\tau}-\ca(0)\cdot{}^t{\bf 1}_{\tau})\ni\tilde{\bf k}\mapsto
\begin{pmatrix}
0\\
\tilde{\bf k}
\end{pmatrix}
\in
\Z^{(n+1)\times 1}/\Z {}^tA_{\s}.
\end{equation}
Therefore, performing a deck transformation associated to $\xi_\tau\mapsto e^{2\pi\ii{}^t\tilde{\bf k}}\xi_\tau$ gives rise to an equation
\begin{equation}
f_{\s,\tilde{\bf k}}(z_0,z)=\frac{
e^{-\pi\ii(1-\gamma)}\exp\{
2\pi\ii
\transp{
\begin{pmatrix}
0\\
\tilde{\bf k}
\end{pmatrix}
}
(A_{\s})^{-1}d
\}
}
{\Gamma(\gamma)}\sum_{i=1}^r
\exp\{
2\pi\ii
\transp{
\begin{pmatrix}
0\\
\tilde{\bf k}
\end{pmatrix}
}
(A_{\s})^{-1}A_{\bs}{\bf k}(i)
\}
\varphi_{\s,{\bf k}(i)}(z_0,z)
\end{equation}
Thus, again by \cref{lem:pairing}, if we take $\tilde{\bf k}(1),\dots,\tilde{\bf k}(r)$ so that we have identities 
\begin{equation}
\left\{ \left[\begin{pmatrix}
0\\
\tilde{\bf k}(i)
\end{pmatrix}\right]\right\}_{i=1}^r=\Z^{(n+1)\times 1}/\Z {}^tA_{\s}\simeq\Z^{n\times 1}/\Z {}^t(A_\tau-\ca(0)\cdot{\bf 1}_\tau)=\{ [\tilde{\bf k}(i)]\}_{i=1}^r,
\end{equation}
and if $\gamma\notin\Z_{\leq 0}$, we can conclude that $\{f_{\s,\tilde{\bf k}(i)}\}$ is a set of $r$ linearly independent solutions of $M_{A}(d)$.

\vspace{2em}
\noindent
\underline{Case 3} We are now at the position of providing a method for constructing integration contours in the general case. Suppose we are given Euler integral (\ref{EulerInt2}). Denoting ${\bf e}_l\;(l=1,\dots,k)$ the standard basis of $\Z^{k\times 1}$,  we put $I_l=\left\{
\begin{pmatrix}
{\bf e}_l\\
\hline
{\bf a}^{(l)}(j)
\end{pmatrix}
\right\}_{j=1}^{N_l}$. We also put $\cA=(A_1\mid\cdots\mid A_k)$. This induces a partition of indices 
\begin{equation}\label{partition}
\{1,\dots,N\}=I_1\cup\dots\cup I_k.
\end{equation}
For any $l=1\dots,k$, we define $\ca(j)\;(j\in I_l)$ by $\ca(j)={\bf a}^{(l)}(j)$ so that we have an equality
\begin{equation}
{\bf a}(j)=
\begin{pmatrix}
{\bf e}_l\\
\hline
\ca(j)
\end{pmatrix}.
\end{equation}
Take an $n+k$ simplex $\s\subset\{1,\dots,N\}$. According to the partition (\ref{partition}), we have an induced partition $\s=\s^{(1)}\cup\dots\cup\s^{(k)}$, where $\s^{(l)}=\s\cap I_l$. Since $\det A_\s\neq 0$, we have, for any $l=1,\dots,k$, $\s^{(l)}\neq 0$. Now choose $k$ labeled points $i^{(l)}\in\s^{(l)}\;(l=1,\dots,k)$ and set $\tau^{(l)}=\s^{(l)}\setminus\{ i^{(l)}\}$, $\tau=\tau^{(1)}\cup\cdots\cup\tau^{(k)}$, and $\s_0=\{i^{(1)},\dots,i^{(k)}\}$. Introducing a new variable $w_j=z_{i^{(l)}}^{-1}z_j\;(j\in I_l)$, we rewrite (\ref{EulerInt2}) into a convenient form:
\begin{equation}
f(z)=\frac{z_{\s_0}^{-\gamma}}{(2\pi\ii)^{n+k}}\int \prod_{l=1}^k\left( 1+\sum_{j\in I_l\setminus\{ i^{(l)}\}}w_jx^{\ca(j)-\ca(i^{(l)})}\right)^{-\gamma_l}x^{c-\cA_{\s_0}\gamma-{\bf 1}}dx.
\end{equation}

\noindent
Now, the covering transform associated to $(\s,\s_0)$ is defined by
\begin{equation}
p:\T^n\ni x\mapsto \xi_\tau=\left((e^{-\pi\ii}w_i x^{\ca(i)-\ca(i^{(1)})})_{i\in\tau_1},\dots,(e^{-\pi\ii}w_i x^{\ca(i)-\ca(i^{(k)})})_{i\in\tau_k}\right)\in\T^\tau.
\end{equation}
This is also abbreviated as
\begin{equation}
p:\T^n\ni x\mapsto \xi_\tau=e^{-\pi\ii{}^t{\bf 1}_\tau}z_{\s_0}^{-S}z_\tau x^{\cA_\tau-\cA_{\s_0}S}\in\T^\tau,
\end{equation}
where $S$ is a stair matrix
\begin{equation}
S=(\overbrace{{\bf e_1}\mid\cdots\mid{\bf e}_1}^{|\tau^{(1)}|\text{ times}}
\mid\cdots\mid
\overbrace{{\bf e_k}\mid\cdots\mid{\bf e}_k}^{|\tau^{(k)}|\text{ times}}
)\in\Z^{k\times \tau}.
\end{equation}
For brevity, we put $w_\tau=z_{\s_0}^{-S}z_\tau.$ Straightforward computations as Case 1 and Case 2 give rise to a formula
\begin{align}
f(z)=&\frac{z_{\s_0}^{-\gamma}(e^{-\pi\ii{\bf 1}_\tau}w_\tau)^{-(\cA_\tau-\cA_{\s_0}S)^{-1}(c-\cA_{\s_0}\gamma)}}{(2\pi\ii)^{n+k}}\times\nonumber\\
      &\int_\gamma\prod_{l=1}^k\left( 1-\sum_{i\in\tau^{(l)}}\xi_i+\sum_{j\in \bs^{(l)}}(e^{-\pi\ii{\bf 1}_\tau}w_\tau)^{-(\cA_\tau-\cA_{\s_0}S)^{-1}\left(\ca(j)-\ca(i^{(l)})\right)}w_j\xi_\tau^{(\cA_\tau-\cA_{\s_0}S)^{-1}\left(\ca(j)-\ca(i^{(l)})\right)}\right)^{-\gamma_l}\times\nonumber\\
      &\xi_\tau^{(\cA_\tau-\cA_{\s_0}S)^{-1}(c-\cA_{\s_0}\gamma)-{\bf 1}}d\xi_\tau.\label{ConvEulerInt}
\end{align}
Note that these computations correspond to an identity
\begin{equation}
M_A(d)=M_{Q_0A}(Q_0d),
\end{equation}
where $Q_0$ is given by
\begin{equation}
Q_0=
\left(
\begin{array}{c|c}
I_k&O\\
\hline
-A_{\s_0}&I_n
\end{array}
\right).
\end{equation}
\begin{lem}\label{lem:matrix}
For any $j\in I_l$ one has an equality
\begin{equation}
A_\s^{-1}{\bf a}(j)=
\begin{pmatrix}
{\bf e}_l-S(\cA_\tau-\cA_{\s_0}S)^{-1}\left(\mathring{\bf a}(j)-\ca(i^{(l)})\right)\\
(\cA_\tau-\cA_{\s_0}S)^{-1}\left(\ca(j)-\ca(i^{(l)})\right)
\end{pmatrix}.
\end{equation}
In particular, one has
\begin{align}
w_\tau^{-(\cA_\tau-\cA_{\s_0}S)^{-1}\left({\bf a}(j)-{\bf a}(i^{(l)})\right)}w_j&=z_{i^{(l)}}^{-1}z_{\s_0}^{S\left(\cA_\tau-\cA_{\s_0}S\right)^{-1}(\ca(j)-\ca(i^{(l)}))}z_\tau^{-\left(\cA_\tau-\cA_{\s_0}S\right)^{-1}(\ca(j)-\ca(i^{(l)}))}z_j\\
&=z_\s^{-A_\s^{-1}{\bf a}(j)}z_j
\end{align}
\end{lem}

\begin{proof}
Since $A_\s^{-1}{\bf a}(j)=(Q_0A_\s)^{-1}Q_0{\bf a}(j)$, we are reduced to compute $Q_0A_\s^{-1}$ and $Q_0{\bf a}(j)$. By definition, we have
\begin{equation}
Q_0{\bf a}(j)=
\begin{pmatrix}
{\bf e}_l\\
\hline
\ca(j)-\ca(i^{(l)})
\end{pmatrix}.
\end{equation}
On the other hand, we have
\begin{equation}
Q_0A_\s=
\left(
\begin{array}{c|c}
I_k&O\\
\hline
-\cA_{\s_0}& I_n
\end{array}
\right)
\left(
\begin{array}{c|c|c}
{}^t{\bf 1}_{\s^{(1)}}&\ldots&O\\
\hline
\vdots&\ddots&\vdots\\
\hline
O&\ldots&{}^t{\bf 1}_{\s^{(k)}}\\
\hline
\cA_{\s^{(1)}}& \ldots&\cA_{\s^{(k)}}
\end{array}
\right)
=
\left(
\begin{array}{c|c}
I_{\s_0}&S\\
\hline
O_{\s_0}& \cA_\tau-\cA_{\s_0}S
\end{array}
\right),
\end{equation}
which leads to a formula
\begin{equation}
(Q_0A_\s)^{-1}=
\left(
\begin{array}{c|c}
I_{\s_0}&-S(\cA_\tau-\cA_{\s_0}S)^{-1}\\
\hline
O_{\s_0}&(\cA_\tau-\cA_{\s_0}S)^{-1}
\end{array}
\right).
\end{equation}
\end{proof}
\noindent
In the same manner, we can prove 
\begin{equation}\label{exponent}
z_{\s_0}^{-\gamma}w_\tau^{-(\cA_\tau-\cA_{\s_0}S)^{-1}(c-\cA_{\s_0}\gamma)}=z_\s^{-A_\s d}.
\end{equation}
\noindent
Now, we take our cycle as a product of Pochhammer contours $\gamma=P_{\xi_{\tau^{(1)}}}\times\dots\times P_{\xi_{\tau^{(k)}}}$. From \cref{lem:matrix} and (\ref{exponent}), we can confirm that the integral (\ref{ConvEulerInt}) is convergent if $z\in U_\s$. For convenience, let us introduce a matrix 
\begin{equation}
T=(\overbrace{{\bf e_1}\mid\cdots\mid{\bf e}_1}^{|\bs^{(1)}|\text{ times}}
\mid\cdots\mid
\overbrace{{\bf e_k}\mid\cdots\mid{\bf e}_k}^{|\bs^{(k)}|\text{ times}}
)\in\Z^{k\times \bs}.
\end{equation}
Expanding the integrand into a series as in Case 1 and Case 2, we obtain a formula
\begin{align}
f(z)=&\frac{z_{\s}^{-A_\s d}}{(2\pi\ii)^{n+k}}\sum_{{\bf m}\in\Z^{\bs}_{\geq 0}}\frac{(-1)^{|{\bf m}|}\displaystyle\prod_{l=1}^k(\gamma_l)_{|{\bf m}_l|}}{{\bf m}!}
e^{\pi\ii{}^t{\bf 1}_\tau(\cA_\tau-\cA_{\s_0}S)^{-1}(\cA_{\bs}-\cA_{\s_0}T){\bf m}}
(z_\s^{-A_\s^{-1}A_{\bs}}z_{\bs})^{\bf m}\times\nonumber\\
&\prod_{l:\tau^{(l)}\neq\varnothing}{\rm I}_l({\bf m}),
\end{align}
where ${\rm I}_l({\bf m})$ is defined by the formula
\begin{equation}
{\rm I}_l({\bf m})=
\begin{cases}
\int_{P_{\xi_{\tau^{(l)}}}}\left( 1-\displaystyle\sum_{i\in\tau^{(l)}}\xi_i\right)^{-\gamma_l-|{\bf m}_l|}\displaystyle\prod_{i\in\tau^{(l)}}\xi_i^{{}^t{\bf e}_i(\cA_\tau-\cA_{\s_0}S)^{-1}(c-\cA_{\s_0}\gamma+(\cA_{\bs}-\cA_{\s_0}T){\bf m})}d\xi_{\tau^{(l)}}&(\bs^{(l)}\neq\varnothing)\\
\int_{P_{\xi_{\tau^{(l)}}}}\left( 1-\displaystyle\sum_{i\in\tau^{(l)}}\xi_i\right)^{-\gamma_l}\displaystyle\prod_{i\in\tau^{(l)}}\xi_i^{{}^t{\bf e}_i(\cA_\tau-\cA_{\s_0}S)^{-1}(c-\cA_{\s_0}\gamma+(\cA_{\bs}-\cA_{\s_0}T){\bf m})}d\xi_{\tau^{(l)}}&(\bs^{(l)}=\varnothing)
\end{cases}
\end{equation}
Computing as in \S\ref{SectionResidue} and substituting the formula
\begin{equation}
(\gamma_l)_{|{\bf m}_l|}=\frac{2\pi\ii e^{-\pi\ii\gamma_l}(-1)^{|{\bf m}_l|}}{\Gamma(\gamma_l)\Gamma(1-\gamma_l-|{\bf m}_l|)(1-e^{-2\pi\ii\gamma_l})},
\end{equation}
we obtain the basic formula
\begin{equation}
f(z)=\frac{\displaystyle\prod_{l:\tau^{(l)}\neq\varnothing}e^{-\pi\ii(1-\gamma_l)}\displaystyle\prod_{l:\tau^{(l)}=\varnothing}e^{-\pi\ii\gamma_l}}{\Gamma(\gamma_1)\dots\Gamma(\gamma_k)\displaystyle\prod_{l:\tau^{(l)}=\varnothing}(1-e^{-2\pi\ii\gamma_l})}\sum_{i=1}^r\varphi_{\s,{\bf k}(i)}(z).
\end{equation}
\noindent
We set $\Gamma_{\ts,\s_0,0}$ to be $p^*(P_\s)$. As before, for any $\tilde{\bf k}\in\Z^{\s}$, we can define the deck transformation $\Gamma_{\ts,\s,\tilde{\bf k}}$ associated to $\xi_\tau\mapsto e^{2\pi\ii{}^t\tilde{\bf k}}\xi_\tau$. Through the isomorphism 
\begin{equation}
\Z/\Z {}^t(\cA_\tau-\cA_{\s_0}S)\ni\tilde{\bf k}\mapsto 
\begin{pmatrix}
0_{\s_0}\\
\tilde{\bf k}
\end{pmatrix}
\in\Z^{(n+k)}/\Z {}^tA_{\s},
\end{equation}
we can show that 
\begin{align}
 &f_{\tilde{\bf k}}(z)\\
=&\frac{1}{(2\pi\ii)^{n+k}}\int_{\Gamma_{\ts,\s,\tilde{\bf k}}} h_{1,z^{(1)}}(x)^{-\gamma_1}\cdots h_{k,z^{(k)}}(x)^{-\gamma_k}x^{c-{\bf 1}}dx\\
=&\frac{\displaystyle\prod_{l:\tau^{(l)}\neq\varnothing}e^{-\pi\ii(1-\gamma_l)}\displaystyle\prod_{l:\tau^{(l)}=\varnothing}e^{-\pi\ii\gamma_l}
\exp\left\{
2\pi\ii\transp{
\begin{pmatrix}
0_{\s_0}\\
\tilde{\bf k}
\end{pmatrix}
}
A_{\s}^{-1}d
\right\}
}{\Gamma(\gamma_1)\dots\Gamma(\gamma_k)\displaystyle\prod_{l:\tau^{(l)}=\varnothing}(1-e^{-2\pi\ii\gamma_l})}\times\nonumber\\
 &\sum_{i=1}^r
\exp\left\{
2\pi\ii\transp{
\begin{pmatrix}
0_{\s_0}\\
\tilde{\bf k}
\end{pmatrix}
}
A_{\s}^{-1}A_{\bs}{\bf k}(i)
\right\}
\varphi_{\s,{\bf k}}(z).
\end{align}
Summing up, we obtain

\begin{thm}\label{thm:fundamentalthm3}
Take a regular triangulation $T$ of $A$. Assume that the parameter vector $d$ is very generic with respect to any $\sigma\in T$ and that for any $l=1,\dots,k$, one has $\gamma_l\notin\Z_{\leq 0}$. For each simplex $\s$, one chooses $k$ labeled points $i^{(l)}\in\s^{(l)}\;(l=1,\dots,k)$ and set $\s_0=\{i^{(1)},\dots,i^{(k)}\}$. Then, if one puts
\begin{equation}
f_{\s,\tilde{\bf k}(j)}(z)=\frac{1}{(2\pi\ii)^{n+k}}\int_{\Gamma_{\s,\s_0,\tilde{\bf k}(j)}} h_{1,z^{(1)}}(x)^{-\gamma_1}\cdots h_{k,z^{(k)}}(x)^{-\gamma_k}x^{c-{\bf 1}}dx,
\end{equation}
$\bigcup_{\s\in T}\{ f_{\sigma,\tilde{\bf k}(j)}(z)\}_{j=1}^{r}$ is a basis of solutions of $M_A(d)$ on the non-empty open set $U_T$, where $\{\tilde{\bf k}(j)\}_{j=1}^{r}$ is a complete system of representatives $\Z^{n}/\Z{}^t(\cA_\tau-\cA_{\s_0}S)$. Moreover, for each $\sigma\in T,$ one has a transformation formula 
\begin{equation}
\begin{pmatrix}
f_{\sigma,\tilde{\bf k}(1)}(z)\\
\vdots\\
f_{\sigma,\tilde{\bf k}(r)}(z)
\end{pmatrix}
=
T_\sigma
\begin{pmatrix}
\varphi_{\sigma,{\bf k}(1)}(z)\\
\vdots\\
\varphi_{\sigma,{\bf k}(r)}(z)
\end{pmatrix}.
\end{equation}
Here, $T_\sigma$ is an $r\times r$ matrix given by 
\begin{align}
T_\sigma=&\frac{\displaystyle\prod_{l:\tau^{(l)}\neq\varnothing}e^{-\pi\ii(1-\gamma_l)}\displaystyle\prod_{l:\tau^{(l)}=\varnothing}e^{-\pi\ii\gamma_l}}{\Gamma(\gamma_1)\dots\Gamma(\gamma_k)\displaystyle\prod_{l:\tau^{(l)}=\varnothing}(1-e^{-2\pi\ii\gamma_l})}
\diag\Big( \exp\left\{
2\pi\ii\transp{
\begin{pmatrix}
0_{\s_0}\\
\tilde{\bf k}(i)
\end{pmatrix}
}
A_{\s}^{-1}d
\right\}\Big)_{i=1}^{r}\times\nonumber\\
 &\Big(
\exp\left\{
2\pi\ii\transp{
\begin{pmatrix}
0_{\s_0}\\
\tilde{\bf k}(i)
\end{pmatrix}
}
A_{\s}^{-1}A_{\bs}{\bf k}(j)
\right\}
\Big)_{i,j=1}^{r}.
\end{align}
\end{thm}

\begin{exa}
We consider a $5\times 9$ matrix 
\begin{equation}
A=
\left(
\begin{array}{ccc|ccc|ccc}
1&1&1&0&0&0&0&0&0\\
\hline
0&0&0&1&1&1&0&0&0\\
\hline
0&0&0&0&0&0&1&1&1\\
\hline
0&1&0&0&1&0&0&1&0\\
0&0&1&0&0&1&0&0&1
\end{array}
\right).
\end{equation}
We take a simplex $\s=\{ 2,4,5,6,7\}$. It is easy to confirm that $\det A_\s=1$. The associated Euler integral which is known to be a solution of $E(3,6)$ (\cite{AK}), is given by
\begin{align}
f_{\s}(z)&=\frac{1}{(2\pi\ii)^5}\int_\Gamma(z_1+z_2x+z_3y)^{-\gamma_1}(z_4+z_5x+z_6y)^{-\gamma_2}(z_7+z_8x+z_9y)^{-\gamma_3}x^{c_1-1}y^{c_2-1}dxdy\\
 &=\frac{z_2^{-\gamma_1}z_4^{-\gamma_2}z_7^{-\gamma_3}}{(2\pi\ii)^5}\int_\Gamma(1+w_1x^{-1}+w_3x^{-1}y)^{-\gamma_1}(1+w_5x+w_6y)^{-\gamma_2}(1+w_8x+w_9y)^{-\gamma_3}\times\nonumber\\
 &\hspace{1.5em}x^{c_1-\gamma_1-1}y^{c_2-1}dxdy.
\end{align}
Here, $w_1,w_3,w_5,w_6,w_8,w_9$ are define by
\begin{equation}
w_1=z_2^{-1}z_1,\;\;w_3=z_2^{-1}z_3,\;\;w_5=z_4^{-1}z_5,\;\;w_6=z_4^{-1}z_6,\;\;w_8=z_7^{-1}z_8,\;\;w_9=z_7^{-1}z_9
\end{equation}
Introducing a new coordinate $(\xi,\eta)$ by
\begin{equation}
\begin{cases}
\xi=e^{-\pi\ii}w_3x\\
\eta=e^{-\pi\ii}w_6y,
\end{cases}
\end{equation}
we have
\begin{align}
f_\s(z)&=\frac{e^{\pi\ii(c_1+c_2-\gamma_1)}z_2^{-\gamma_1}z_4^{c_1+c_2-\gamma_1-\gamma_2}z_7^{-\gamma_3}z_5^{\gamma_1-c_1}z_6^{-c_2}}{(2\pi\ii)^5}\times\nonumber\\
 &\hspace{1.5em}\int_\gamma(1+e^{-\pi\ii}w_1w_5\xi^{-1}+w_3w_5w_6^{-1}\xi^{-1}\eta)^{-\gamma_1}(1+e^{\pi\ii}w_5^{-1}w_8\xi+w_6^{-1}w_9\eta)^{-\gamma_3}\times\nonumber\\
 &\hspace{1.5em}(1-\xi-\eta)^{-\gamma_2}\xi^{c_1-\gamma_1-1}\eta^{c_2-1}d\xi d\eta.
\end{align}
We put $l_1=\{ 1+e^{-\pi\ii}w_1w_5\xi^{-1}+w_3w_5w_6^{-1}\xi^{-1}\eta=0\}$, $l_2=\{1=\xi+\eta\}$, $l_3=\{ 1+e^{\pi\ii}w_5^{-1}w_8\xi+w_6^{-1}w_9\eta=0\}.$ The arrangement of branching locus of the integrand is described in the following figure.

\begin{figure}[t]
\begin{center}
\begin{tikzpicture}
\draw (0,0) node[below right]{O}; 
\draw[thick, ->] (-4,0)--(3.2,0) node[right]{$\xi$} ; 
\draw[thick, ->] (0,-1.2)--(0,3.2) node[above]{$\eta$} ; 
\draw[-] (3,-1)--(-1,3) node[left]{$\{ \xi+\eta=1\}=l_2$};
\draw[-] (1,4)--(-3.5,-0.5) node[left]{$l_3$};
\draw[-] (-0.2,3)--(0.2,-1) node[right]{$l_1$};
\end{tikzpicture}
\caption{arrangement of hyperplanes}
\end{center}
\end{figure}

\noindent
Taking limit $w_1w_5,w_3w_5w_6^{-1}\rightarrow0$ amounts to taking limit $l_1\rightarrow \{ \xi=0\}$, while taking limit $w_5^{-1}w_8,w_6^{-1}w_9\rightarrow0$ amounts to taking limit $l_3\rightarrow\{\text{hyperplane at }\infty \}.$ Thus, our Pochhammer cycle $\gamma$ associated to the simplex $\{(\xi,\eta)\mid \xi\geq 0,\eta\geq 0,\xi+\eta\leq 1\}$ encircles the divisor $l_1$ while it does not encircle the divisor $l_3$. The $\Gamma$-series associated to $\s$ is explicitly given by a formula
\begin{align}
\varphi_{\s}(z)&=\sum_{m_1,\dots,m_4\in\Z_{\geq 0}}
\frac{1}{\Gamma(1-\gamma_1-m_1-m_2)\Gamma(1-\gamma_1-\gamma_2+c_1+c_2-m_1+m_3+m_4)\Gamma(1-\gamma_3-m_3-m_4)}\times\nonumber\\
 &\hspace{1.5em}\frac{
(w_1w_5)^{m_1}(w_3w_5w_6^{-1})^{m_2}(w_5^{-1}w_8)^{m_3}(w_6^{-1}w_9)^{m_4}
}
{
\Gamma(1+\gamma_1-c_1+m_1+m_2-m_3)\Gamma(1-c_2-m_2-m_4)m_1!m_2!m_3!m_4!
}.
\end{align}
The relation between $f_\s(z)$ and $\varphi_\s(z)$ is given by the formula
\begin{equation}
f_\s(z)=\frac{e^{-\pi\ii(1+\gamma_1-\gamma_2+\gamma_3)}}{\Gamma(\gamma_1)\Gamma(\gamma_2)\Gamma(\gamma_3)(1-e^{-2\pi\ii\gamma_1})(1-e^{-2\pi\ii\gamma_3})}\varphi_\s(z).
\end{equation}

\end{exa}

\end{section}

\begin{section}{Equivalence of various integral representations}\label{SectionEquivalence}
In this section, we provide a $D$-module theoretic background of integral representations discussed so far. The form of the integrals we consider in this section are the followings:

\begin{equation}\label{MixedInt}
\int h_{1,z^{(1)}}(x)^{-\gamma_1}\cdots h_{k,z^{(k)}}(x)^{-\gamma_k}x^{c-1} e^{h_{0,z^{(0)}}(x)}dx,
\end{equation}
\begin{equation}\label{MixedLaplaceInt}
\int  \exp\Big\{h_{0,z^{(0)}}(x)+\sum_{l=1}^k y_lh_{l,z^{(l)}}(x)\Big\}y^{\gamma-1}x^{c-1}dydx,
\end{equation}
\begin{equation}\label{MixedResidueInt}
\int\frac{e^{h_{0,z^{(0)}}(x)}y^{\gamma -1}x^{c-1}}{\Big(1-y_1h_{1,z^{(1)}}(x)\Big)\cdots \Big(1-y_kh_{k,z^{(k)}}(x)\Big)}dydx.
\end{equation}

\noindent
Here, $h_{l,z^{(l)}}(x)$ $(l=0,\dots, k)$ are Laurent polynomials of the form
\begin{equation}
h_{l,z^{(l)}}(x)=\displaystyle\sum_{j=1}^{N_l}z^{(l)}_jx^{{\bf a}^{(l)}(j)}.
\end{equation}
\noindent
In order to formulate our result, let us revise some basic notation and results of algebraic $\DD$-modules. For their proofs, see \cite{Bo} or \cite{HTT}. Let $X$ and $Y$ be smooth algebraic varieties over $\C$ and let $f:X\rightarrow Y$ be a morphism. We denote $\DD_X$ the sheaf of linear partial differential operators on $X$ and denote $D^b_h(\DD_X)$ (resp. $D^b_{r.h.}(\DD_X)$) the derived category of bounded complexes of left $\DD_X$-modules whose cohomologies are holonomic (resp. regular holonomic). For any object $N\in D^b_h(\DD_Y)$ (resp. $D^b_{r.h.}(\DD_Y)$) on $Y$, we define its inverse image $\LL f^*N\in D^b_h(\DD_X)$ (resp. $D^b_{r.h.}(\DD_X)$) and its shifted inverse image $f^\dagger N\in D^b_h(\DD_X)$ (resp. $D^b_{r.h.}(\DD_X)$) with respect to $f$ by the formula
\begin{equation}
\LL f^*N=\DD_{X\rightarrow Y}\overset{\LL}{\underset{f^{-1}\DD_Y}{\otimes}}f^{-1}N \;\;\;(\text{resp. } f^{\dagger}N=\LL f^*N[\dim X-\dim Y]),
\end{equation} 
where $\DD_{X\rightarrow Y}$ is the transfer module $\mathcal{O}_X\otimes_{f^{-1}\mathcal{O}_Y}f^{-1}\DD_Y.$ Similarly, for any object $M\in D^b_h(\DD_X)$ (resp. $D^b_{r.h.}(\DD_X)$), we define its direct image $\int_fM\in D^b_h(\DD_Y)$ (resp. $D^b_{r.h.}(\DD_Y)$) by
\begin{equation}
\int_fM=\R f_*(\DD_{Y\leftarrow X}\overset{\LL}{\underset{\DD_X}{\otimes}}M),
\end{equation}
where $\DD_{Y\leftarrow X}$ is the transfer module $\Omega_X\otimes_{\mathcal{O}_X}\DD_{X\rightarrow Y}\otimes_{f^{-1}\mathcal{O}_Y}f^{-1}\Omega_Y$. If $X=Y\times Z$ and $f:Y\times Z\rightarrow Y$ is the natural projection, the direct image can be computated in terms of (algebraic) relative de Rham complex
\begin{equation}
\int_fM\simeq \R f_*(\DR_{X/Y}(M)).
\end{equation}
In particular, if $Y=\{*\}$ (one point), and $M$ is a connection $M=(E,\nabla)$ on $Z$, then for any integer $p$, we have a canonical isomorphism
\begin{equation}
\Homo^p\left( \int_fM\right)\simeq\mathbb{H}_{dR}^{p+\dim Z}(Z,(E,\nabla)),
\end{equation} 
where $\mathbb{H}_{dR}$ denotes the algebraic de-Rham cohomology group.
For objects $M,M^\prime\in D^b_h(\DD_X)$ (resp. $D^b_{r.h.}(\DD_X)$) and $N\in D^b_h(\DD_Y)$ (resp. $D^b_{r.h.}(\DD_Y)$), the tensor product $M\overset{\mathbb{D}}{\otimes}M^\prime\in D^b_h(\DD_X)$ (resp. $D^b_{r.h.}(\DD_X)$) and external tensor product $M\boxtimes N\in D^b_h(\DD_{X\times Y})$ (resp. $D^b_{r.h.}(\DD_{X\times Y})$) are defined by
\begin{equation}
M\overset{\mathbb{D}}{\otimes}M^\prime=M\overset{\mathbb{L}}{\underset{\mathcal{O}_X}{\otimes}}M^\prime,\;\; M\boxtimes N=M\underset{\C}{\otimes}N.
\end{equation}
Let $Z$ be a smooth closed subvariety of $X$ and let $i:Z\hookrightarrow X$ and $j:X\setminus Z\hookrightarrow X$ be natural inclusions. Then, for any object $M\in D_h^b(\DD_X)$, there is a standard distinguished triangle
\begin{equation}\label{sdt}
\int_ii^\dagger M\rightarrow M\rightarrow\int_j j^\dagger M\overset{+1}{\rightarrow}.
\end{equation}

\noindent
For any (possibly multivalued) function $\varphi$ on $X$ such that $\varphi$ is nowhere-vanishing and that $\frac{d\varphi}{\varphi}$ belongs to $\Omega^1(X)$, we define a $\DD_X$-module $\mathcal{O}_X\varphi$ by twisting its action as
\begin{equation}
\theta\cdot h=\Big\{\theta-\Big(\frac{\theta\varphi}{\varphi}\Big)\Big\}h\;\;\;(h\in\mathcal{O}_X,\;\theta\in\Theta_X).
\end{equation}
For any $\DD_X$-module $M,$ we define $M\varphi$ by $M\varphi=M\underset{\mathcal{O}_X}{\otimes}\mathcal{O}_X\varphi.$

We begin with the equivalence of (\ref{MixedInt}) and (\ref{MixedLaplaceInt}). This isomorphism is nothing but an algebraic interpretation of the so-called Cayley trick. We will prove the following identity which is ``obvious'' from the definition of $\Gamma$ function.

\begin{prop}\label{prop:GammaInt}
Let $h:X\rightarrow\A^1$ be a non-zero regular function such that $h^{-1}(0)$ is smooth, $\pi:X\times (\Gm)_y\rightarrow X$ be the canonical projection, $j:X\setminus h^{-1}(0)\hookrightarrow X$ and $i:h^{-1}(0)\hookrightarrow X$ be inclusions, and let $\gamma\in\C\setminus\Z$ be a parameter. One has a canonical isomorphism

\begin{equation}
\int_{\pi}\mathcal{O}_{X\times(\Gm)_y}y^\gamma e^{yh}\simeq\int_j\mathcal{O}_{X\setminus h^{-1}(0)}h^{-\gamma}.
\end{equation}

\end{prop}
\noindent
For the proof, we insert the following elementary

\begin{lem}\label{gamma}
Let $pt:(\Gm)_y\rightarrow\{ *\}$ be the trivial morphism. If $\gamma\in\C\setminus \Z$ and $h\in\C,$ one has
\begin{equation}
\text{(1) }h=0\Rightarrow \int_{pt}\mathcal{O}_{(\Gm)_y}y^\gamma e^{hy}=0
\end{equation}
\begin{equation}
\text{(2) }h\neq 0\Rightarrow \int_{pt}\mathcal{O}_{(\Gm)_y}y^\gamma e^{hy}=\C.
\end{equation}
\end{lem}

\begin{proof}
Remember first that 
\begin{equation}
\int_{pt}\mathcal{O}_{(\Gm)_y}y^\gamma e^{hy}=\Big(\Omega^{\bullet+1}((\Gm)_y), \nabla\Big)=\Big(0\rightarrow \overset{\overset{-1}{\smile}}{\C[y^{\pm}]}\overset{\nabla}{\rightarrow}\overset{\overset{0}{\smile}}{\C[y^\pm]}\rightarrow 0\Big)
\end{equation}
where $\nabla=\frac{\partial}{\partial y}+\frac{\gamma}{y}+h.$ Take any element $\displaystyle\sum_{n\in \Z}a_ny^n\in\C[y^\pm]$. By the definition of $\nabla$, we have
\begin{equation}
\nabla\Big(\sum_{n\in\Z}a_ny^n\Big)=\sum_{n\in\Z}\lef (n+\gamma)a_n+ha_{n-1}\righ y^{n-1}.
\end{equation}
If $h=0$, it is clear that $\nabla$ is an isomorphism of $\C$ vector spaces since $\gamma\notin\Z$. Thus, (1) is true. In the following, we assume $h\neq 0$. If $\displaystyle\sum_{n\in\Z}a_ny^n\in\Ker\nabla,$ we have $a_{n-1}=-\frac{(n+\gamma)}{h}a_n$ for any $n$, hence $a_n=0$ for all $n$. Now, let us show that $1\notin\im\nabla.$ 
Suppose the converse. Then, $\nabla\Big(\displaystyle\sum_{n\in\Z}a_ny^n\Big)=1$ implies
\begin{equation}
(1+\gamma)a_1+ha_0=1,\;\; (n+\gamma)a_n+ha_{n-1}=0 \; (\forall n\neq 1).
\end{equation}
Suppose that $a_1\neq 0.$ Then, for any $n\in\Z_{\geq 2},$ 
\begin{equation}
a_n=-\frac{h}{n+\gamma}a_{n-1}=\dots=\frac{(-1)^{n-1}h^{n-1}}{(n+\gamma)\dots (2+\gamma)}a_1\neq 0.
\end{equation}
This contradicts the fact that $\displaystyle\sum_{n\in\Z}a_ny^n\in\C[y^\pm].$ Similarly, if we assume that $a_0\neq 0,$ the formula $a_{n-1}=-\frac{(n+\gamma)}{h}a_n$ implies a contradiction.

On the other hand, since
\begin{equation}
\nabla(y^n)=ny^{n-1}+\gamma y^{n-1}+hy^n,
\end{equation}
we can conclude that 
\begin{equation}
\C\ni 1\mapsto [1]\in \C[y^\pm]/\im\nabla\neq 0
\end{equation}
is surjective. Hence this is an isomorphism.
\end{proof}

\noindent
(Proof of proposition)

\noindent
Consider the following cartesian diagram:

\begin{equation}
\xymatrix{
& h^{-1}(0)\times(\Gm)_y \ar[r]^{\tilde{i}} \ar[d]_{\tilde{\pi}}&X\times(\Gm)_y \ar[d]^{\pi}\\
& h^{-1}(0)   \ar[r]^{i}                             &X.}
\end{equation}

\noindent
By base change formula and Lemma \ref{gamma} (1), we have
\begin{equation}
i^\dagger\int_\pi\mathcal{O}_{X\times (\Gm)_y}y^\gamma e^{yh}=\int_{\tilde{\pi}}\tilde{i}^\dagger\mathcal{O}_{X\times (\Gm)_y}y^\gamma e^{yh}=\int_{\tilde{\pi}}\mathcal{O}_{h^{-1}(0)\times (\Gm)_y}y^\gamma [-1]=0.
\end{equation}
By the standard distinguished triangle (\ref{sdt}), we have a canonical isomorphism
\begin{equation}
\int_\pi\mathcal{O}_{X\times (\Gm)_y}y^\gamma e^{yh}\simeq\int_jj^\dagger\int_\pi\mathcal{O}_{X\times (\Gm)_y}y^\gamma e^{yh}.
\end{equation}
We are going to compute the latter complex. Consider the following cartesian square:

\begin{equation}
\xymatrix{
& \Big(X\setminus h^{-1}(0)\Big)\times(\Gm)_y \ar[r]^{\;\;\;\;\;\;\tilde{j}} \ar[d]_{\tilde{\pi}^\prime}&X\times(\Gm)_y \ar[d]^{\pi}\\
& X\setminus h^{-1}(0)   \ar[r]^{   \;\;\;\;\;\;   j}                             &X.}
\end{equation}
Again by projection formula, we have
\begin{equation}
j^\dagger\int_\pi\mathcal{O}_{X\times (\Gm)_y}y^\gamma e^{yh}\simeq\int_{\tilde{\pi}^\prime}\tilde{j}^\dagger\mathcal{O}_{X\times (\Gm)_y}y^\gamma e^{yh}.
\end{equation}

\noindent
We consider an isomorphism $\varphi: \Big( X\setminus h^{-1}(0)\Big)\times(\Gm)_y\tilde{\rightarrow}\Big( X\setminus h^{-1}(0)\Big)\times(\Gm)_y$ defined by $\varphi(x,y)=(x,\frac{y}{h(x)}).$ Since $\tilde{\pi}^\prime=\tilde{\pi}^\prime\circ\varphi,$ one has 

\begin{equation}
\int_{\tilde{\pi}^\prime}\tilde{j}^\dagger\mathcal{O}_{X\times (\Gm)_y}y^\gamma e^{yh}\simeq\int_{\tilde{\pi}^\prime}\int_\varphi\mathcal{O}_{\Big(X\setminus h^{-1}(0)\Big)\times (\Gm)_y}y^\gamma e^{yh}\simeq\int_{\tilde{\pi}^\prime}\mathcal{O}_{X\setminus h^{-1}(0)}h^{-\gamma}\boxtimes \mathcal{O}_{(\Gm)_y}y^\gamma e^{y}\simeq \mathcal{O}_{X\setminus h^{-1}(0)}h^{-\gamma}.
\end{equation}
Thus, Proposition follows.\qed

\begin{rem}
In the proof above, we have used the following simple fact: Let $X$ be a smooth algebraic variety, and $f:X\rightarrow X$ be an isomorphism. Then, we have an identity
\begin{equation}\label{Cart}
\int_f\simeq (f^{-1})^{\dagger}=\mathbb{L}(f^{-1})^*.
\end{equation}
Indeed, base change formula applied to the following Cartesian diagram gives the identity (\ref{Cart}):
\begin{equation}
\xymatrix{
& X \ar[r]^{{\rm id}_X } \ar[d]_{f^{-1}}&X\ar[d]^{{\rm id}_X}\\
& X \ar[r]^{f}                             &X.
}
\end{equation}
\end{rem}

\begin{cor}\label{cor:gamma}
Let $X$ be a smooth algebraic variety, $h_l:X\rightarrow\A^1$ $(l=1,\cdots, k)$ be non-zero regular functions such that $h_l^{-1}(0)$ are smooth, $\pi:X\times (\Gm)_y^k\rightarrow X$ be the canonical projection, $j:X\setminus \lef h_1\dots h_k=0\righ\hookrightarrow X$ be the inclusion, and let $\gamma_l\in\C\setminus\Z$ be parameters. One has a canonical isomorphism

\begin{equation}
\int_{\pi}\mathcal{O}_{X\times(\Gm)_y^k}y_1^{\gamma_1}\dots y_k^{\gamma_k} e^{y_1h_1+\dots+y_kh_k}\simeq\int_j\mathcal{O}_{X\setminus \lef h_1\dots h_k=0\righ}h_1^{-\gamma_1}\cdots h_k^{-\gamma_k}.
\end{equation}
\end{cor}

\begin{proof}
First, let us remark that, for any smooth algebraic varieties $X_i,Y_i\;(i=1,2)$, morphisms $f_i:X_i\rightarrow Y_i\; (i=1,2),$ and $\DD_{X_i}$-modules $M_i\;(i=1,2)$, there is a natural isomorphism of $\DD_{Y_1\times Y_2}$-modules (Proposition 1.5.30. in \cite{HTT}):
\begin{equation}\label{product}
\int_{f_1\times f_2}M_1\boxtimes M_2\simeq\int_{f_1}M_1\boxtimes\int_{f_2}M_2.
\end{equation}
Now, consider the following cartesian square:

\begin{equation}
\xymatrix@C=40pt{
& X_x\times(\Gm)_y^k \ar[r]^{\hspace{-8mm}\diag_x\times {\rm id}_y } \ar[d]_{\pi_x}&X_{x_1}\times\dots X_{x_k}\times(\Gm)_y^k \ar[d]^{\pi_{x_1,\dots ,x_k}}\\
& X_x   \ar[r]^{\hspace{-8mm}\diag_x}                             &X_{x_1}\times\dots\times X_{x_k}.
}
\end{equation}
Here, subindices denotes the coordinate of the spaces we consider. For example, $(\Gm)^k_y$ denotes an algebraic $k$ torus whose coordinate ring is $\C[y^\pm]=\C[y_1^\pm,\dots,y_k^\pm]$ etc. 
If we denote by $j_l:X_{x_l}\setminus h_l^{-1}(0)\hookrightarrow X_{x_l}$ the open embedding, base change formula in view of (\ref{product}) gives a sequence of isomorphisms
\begin{align}
  &\int_{\pi}\mathcal{O}_{X\times(\Gm)_y^k}y_1^{\gamma_1}\dots y_k^{\gamma_k} e^{y_1h_1(x)+\dots+y_kh_k(x)}\\
 \simeq\hspace{2.4em}&\int_{\pi}\mathbb{L}(\diag_x\times {\rm id}_y)^*\Big[\mathcal{O}_{X_{x_1}\times(\Gm)_{y_1}}y_1^{\gamma_1}e^{y_1h_1(x_1)}\boxtimes\dots\boxtimes\mathcal{O}_{X_{x_k}} y_k^{\gamma_k} e^{y_kh_k(x_k)}\Big]\\
 \simeq\hspace{2.4em}&\mathbb{L}\diag_x^*\int_{\pi_{x_1,\dots,x_k}}\Big[\mathcal{O}_{X_{x_1}\times(\Gm)_{y_1}}y_1^{\gamma_1}e^{y_1h_1(x_1)}\boxtimes\dots\boxtimes\mathcal{O}_{X_{x_k}\times(\Gm)_{y_k}} y_k^{\gamma_k} e^{y_kh_k(x_k)}\Big]\\
 \overset{\cref{prop:GammaInt}}{\simeq}&\mathbb{L}\diag_x^*\Big[ \int_{j_1}\mathcal{O}_{X_{x_1}\setminus h_1^{-1}(0)}h_1(x_1)^{-\gamma_1}\boxtimes\dots\boxtimes\int_{j_k}\mathcal{O}_{X_{x_k}\setminus h_k^{-1}(0)}h_k(x_k)^{-\gamma_k}\Big]\\
 \simeq\hspace{2.4em}&\mathbb{L}\diag_x^*\int_{j_1\times\dots\times j_k}\Big[ \mathcal{O}_{X_{x_1}\setminus h_1^{-1}(0)}h_1(x_1)^{-\gamma_1}\boxtimes\dots\boxtimes\mathcal{O}_{X_{x_k}\setminus h_k^{-1}(0)}h_k(x_k)^{-\gamma_k}\Big].
\end{align}

\noindent
Finally, base change formula applied to the cartesian square

\begin{equation}
\xymatrix{
& X_x\setminus \lef h_1\cdots h_k=0\righ \ar[r]^{\hspace{-15mm}\widetilde{\diag_x} } \ar[d]_{j}&\Big(X_{x_1}\setminus h_1^{-1}(0)\Big)\times\dots\times \Big(X_{x_k}\setminus h_k^{-1}(0)\Big) \ar[d]^{j_1\times\cdots\times j_k}\\
& X_x   \ar[r]^{\hspace{-15mm}\diag_x}                             &X_{x_1}\times\dots\times X_{x_k}
}
\end{equation}
\noindent
gives isomorphisms

\begin{align}
  &\mathbb{L}\diag_x^*\int_{j_1\times\dots\times j_k}\Big[ \mathcal{O}_{X_{x_1}\setminus h_1^{-1}(0)}h_1(x_1)^{-\gamma_1}\boxtimes\dots\boxtimes\mathcal{O}_{X_{x_k}\setminus h_k^{-1}(0)}h_k(x_k)^{-\gamma_k}\Big]\\
 \simeq&\int_j\mathbb{L}\widetilde{\diag_x}^*\Big[ \mathcal{O}_{X_{x_1}\setminus h_1^{-1}(0)}h_1(x_1)^{-\gamma_1}\boxtimes\dots\boxtimes\mathcal{O}_{X_{x_k}\setminus h_k^{-1}(0)}h_k(x_k)^{-\gamma_k}\Big]\\
 \simeq&\int_j\mathcal{O}_{X\setminus \lef h_1\dots h_k=0\righ}h_1(x)^{-\gamma_1}\cdots h_k(x)^{-\gamma_k}.
\end{align}

\end{proof}

The following theorem proves the equivalence of (\ref{MixedInt}) and (\ref{MixedLaplaceInt}).
\begin{thm}[Cayley trick for mixed integrals]\label{thm:CayleyTrick}
Let $h_{l,z^{(l)}}(x)=\displaystyle\sum_{j=1}^{N_l}z_j^{(l)}x^{{\bf a}^{(l)}(j)}$ $(l=0,1,\dots,k)$ be Laurent polynomials. We put $N=N_0+\cdots+N_k$, $z=(z^{(0)},\dots,z^{(k)})$, and $A_l=({\bf a}^{(l)}(1)\mid\cdots\mid {\bf a}^{(l)}(N_l))$. Let $\pi:\A^N_z\times (\Gm)_x^n\setminus\lef h_{1,z^{(1)}}\cdots h_{k,z^{(k)}}=0\righ\rightarrow\A^N_z$ and $\varpi: \A^N_z\times(\Gm)_y^k\times(\Gm)_x^n\rightarrow\A^N_z$ be projections and $\gamma_l\in\C\setminus\Z$ be parameters. Then, one has an isomorphism

\begin{equation}\label{Desired}
\int_\pi\mathcal{O}_{\A^N_z\times (\Gm)_x^n\setminus\lef h_{1,z^{(1)}}\cdots h_{k,z^{(k)}}=0\righ }h_1^{-\gamma_1}\cdots h_k^{-\gamma_k}x^c e^{h_{0,z^{(0)}}(x)}\simeq\int_\varpi\mathcal{O}_{\A^N_z\times(\Gm)_y^k\times(\Gm)_x^n}y^\gamma x^c e^{h_z(y,x)},
\end{equation}
where $h_z(y,x)=h_{0,z^{(0)}}(x)+\displaystyle\sum_{l=1}^ky_lh_{l,z^{(l)}}(x).$
\end{thm}
\begin{proof}
Note first that hypersurfaces $\{h_{l,z^{(l)}}=0\}\subset\A^N_z\times(\Gm)^n_x$ $(l=1,\dots,k)$ are smooth. Now, consider the following commutative diagram:

\begin{equation}
\xymatrix{
& \A^N_z\times(\Gm)_x^n\setminus \lef h_{1,z^{(1)}}\cdots h_{k,z^{(k)}}=0\righ \ar[d]_{\pi } \ar[dr]_{j}& \\
& \A^N_z  & \A^N_z\times(\Gm)_x^n \ar[l]_{\tilde{\pi}}\\
& \A^N_z\times(\Gm)_y^k\times(\Gm)_x^n   \ar[ur]^{p}   \ar[u]_{\varpi}    & .
}
\end{equation}
\noindent
By projection formula,
\begin{align}
 &\int_j\mathcal{O}_{\A^N_z\times (\Gm)_x^n\setminus\lef h_{1,z^{(1)}}\cdots h_{k,z^{(k)}}=0\righ }h_1^{-\gamma_1}\cdots h_k^{-\gamma_k}x^c e^{h_{0,z^{(0)}}(x)}\nonumber\\
 \simeq&\int_j\Big( \mathcal{O}_{\A^N_z\times (\Gm)_x^n\setminus\lef h_{1,z^{(1)}}\cdots h_{k,z^{(k)}}=0\righ }h_1^{-\gamma_1}\cdots h_k^{-\gamma_k}\Big)\overset{\mathbb{D}}{\otimes}\mathcal{O}_{\A^N_z\times(\Gm)^n_x}x^ce^{h_{0,z^{(0)}}(x)}.\label{CIsom1}
\end{align}
\noindent
By \cref{cor:gamma}, we have
\begin{equation}
\int_j\Big( \mathcal{O}_{\A^N_z\times (\Gm)_x^n\setminus\lef h_{1,z^{(1)}}\cdots h_{k,z^{(k)}}=0\righ }h_1^{-\gamma_1}\cdots h_k^{-\gamma_k}\Big)\simeq\int_p\mathcal{O}_{\A^N_z\times(\Gm)_y^k\times(\Gm)_x^n}y^\gamma e^{y_1h_{1,z^{(1)}}+\cdots+y_kh_{k,z^{(k)}}}.
\end{equation}
Again by projection formula, we have
\begin{equation}\Big(\int_p\mathcal{O}_{\A^N_z\times(\Gm)_y^k\times(\Gm)_x^n}y^\gamma e^{y_1h_{1,z^{(1)}}+\cdots+y_kh_{k,z^{(k)}}}\Big)\overset{\mathbb{D}}{\otimes}\mathcal{O}_{\A^N_z\times(\Gm)^n_x}x^ce^{h_{0,z^{(0)}}(x)}\simeq\int_p\mathcal{O}_{\A^N_z\times(\Gm)_y^k\times(\Gm)_x^n}y^\gamma x^c e^{h_z(y,x)}\label{CIsom2}
\end{equation}

\noindent
Since one has canonical isomorphisms 
\begin{equation}
\int_\pi\simeq\int_{\tilde{\pi}}\circ\int_j\;\;\;\;\;\;\;\;\int_\varpi\simeq\int_{\tilde{\pi}}\circ\int_p\;\;\;,
\end{equation}
applying the functor $\int_{\varpi}$ to the left hand side of (\ref{CIsom1}) and to the right hand side of (\ref{CIsom2}) yields to the desired formula (\ref{Desired}).
\end{proof}

Next, we establish an equivalence between (\ref{MixedInt}) and (\ref{MixedResidueInt}). The fundamental idea behind the construction of the canonical isomorphism is Leray's theory of residues \cite{L}. Let us briefly explain the most important formula in his theory. Suppose we are given a complex manifold $X$ and a smooth hypersurface $Y$ in $X$, For any $p$ cycle $[\gamma]\in\Homo_p(Y,\C)$ in $Y$, one can define the coboundary $\delta \gamma$ of $\gamma$ as a $(p+1)$ cycle $[\delta\gamma]\in\Homo_{p+1}(X\setminus Y,\C)$ in $X\setminus Y$. J. Leray's original construction is purely geometric. Namely, we first take a tubular neighbourhood $N$ and a projection $\pi:N\rightarrow Y$. Since $\pi$ is a $\Delta^1$ bundle, we can naturally consider a $\partial\Delta$ bundle $\pi^\prime:\partial N\rightarrow Y$. Then the coboundary cycle $\delta\gamma$ is obtained as $\delta\gamma=(\pi^{\prime})^{-1}(\gamma)$ equipped with a suitable orientation. On the other hand, for any $p+1$ cocycle $[\omega]\in\Homo^{p+1}(X\setminus Y,\C)$ in $X\setminus Y$, one can naturally define the residue ${\rm res}(\omega)$ as a $p$ cocycle $[{\rm res}(\omega)]\in\Homo^{p}(X\setminus Y,\C)$ in $Y$.  Leray's residue theorem states that coboundary operation $\delta$ and residue operation ${\rm res}$ are dual to each other, i.e.,
\begin{equation}
\frac{1}{2\pi\ii}\int_{\delta\gamma}\omega=\int_\gamma{\rm res}(\omega).
\end{equation}
Even more generally, one can define the composed coboundary and composed residue operations. Suppose we are given $k$ smooth hyperpersurfaces $Y_1,\dots,Y_k$. Suppose that these hypersurfaces are normal crossings. Then, we can consider a sequence of embeddings of smooth hypersurfaces 
\begin{align*}
Y_1\cap\cdots \cap Y_k&\rightarrow Y_2\cap\cdots \cap Y_k\setminus Y_1\\
                                 &\rightarrow Y_3\cap\cdots \cap Y_k\setminus (Y_1\cup Y_2)\\
                                 &\rightarrow\cdots\\
                                 &\rightarrow X\setminus (Y_1\cup\cdots\cup Y_k).
\end{align*}
If we denote $\delta^{(1)},\dots,\delta^{(k)}$ (resp. ${\rm res}^{(1)},\dots,{\rm res}^{(k)}$) the corresponding coboundary operations (resp. residue operations), Leray's composed residue theorem states that for any $p$ cycle $[\gamma]\in\Homo_p(Y_1\cap\dots\cap Y_k,\C)$ in $Y_1\cap\dots\cap Y_k$ and for any $p+k$ cocycle $[\omega]\in\Homo^{p+k}(X\setminus (Y_1\cup\dots\cup Y_k),\C)$ in $X\setminus (Y_1\cup\cdots\cup Y_k)$, we have an identity
\begin{equation}
\frac{1}{(2\pi\ii)^k}\int_{\delta^{(k)}\circ\dots\circ\delta^{(1)}\gamma}\omega=\int_\gamma{\rm res}^{(1)}\circ\dots\circ{\rm res}^{(k)}(\omega).
\end{equation}
 
Based on this theory, let us turn back to the equivalence between (\ref{MixedInt}) and (\ref{MixedResidueInt}).
Let us consider a wider space
\begin{equation}
Y=\mathbb{A}^N_z\times(\Gm)^k_y\times(\Gm)_x^n
\end{equation}
and divisors 
\begin{equation}
S_l=\{ 1=y_lh_{l,z^{(l)}}(x)\}\subset Y.
\end{equation}
Here, $\A^N=\A^{N_0}_{z^{(0)}}\times\dots\times\A^{N_k}_{z^{(k)}}$ as in \cref{thm:CayleyTrick}. We consider the following sequence of Leray's coboundary operation:
\begin{align*}
S_1\cap\cdots \cap S_k&\rightarrow S_2\cap\cdots \cap S_k\setminus S_1\\
                                 &\rightarrow S_3\cap\cdots \cap S_k\setminus (S_1\cup S_2)\\
                                 &\rightarrow \cdots\\
                                 &\rightarrow Y\setminus (S_1\cup\cdots\cup S_k).
\end{align*}

\noindent
Put $X=\A^N_z\times(\Gm)^n_x\setminus\lef   h_{1,z^{(1)}}\dots h_{k,z^{(k)}}=0\righ,$ and 

\noindent
$\tilde{Y}=\A^N_z\times(\Gm)^k_y\times(\Gm)_x^n\setminus\lef   (1-y_1h_{1,z^{(1)}})\dots(1- y_kh_{k,z^{(k)}})=0\righ.$ Let $\pi_z:X\rightarrow \A^N_z$ and $\tilde{\varpi}:\tilde{Y}\rightarrow\A^N_z$ be the natural projections. The equivalence is formulated as follows:

\begin{thm}[Composed residue isomorphism]\label{thm:ComposedResidue}
Assume $\gamma_l\notin\mathbb{Z}.$ Then, there is a canonical isomorphism of $\DD_{\A^N}$-modules
\begin{equation}
\int_{\pi_z}\mathcal{O}_Xh_{1,z^{(1)}}^{-\gamma_1}\dots h_{k,z^{(k)}}^{-\gamma_k}x^ce^{h_{0,z^{(0)}}(x)}\simeq\int_{\tilde{\varpi}_z}\mathcal{O}_{\tilde{Y}}y^\gamma x^ce^{h_{0,z^{(0)}}(x)}.
\end{equation}
\end{thm}

\begin{proof}
Consider an integrable connection
\begin{equation}
L_k=\mathcal{O}_Yy^{\gamma}x^{c}e^{h_{0,z^{(0)}}(x)}
\end{equation}
on $Y$.
We put 
\begin{equation}
X_m=\displaystyle \bigcap_{l=m+1}^kS_l\setminus\bigcup_{l=1}^{m}S_l,\; Y_m=\displaystyle\bigcap_{l=m+1}^kS_l\setminus\bigcup_{l=1}^{m-1}S_l.
\end{equation}
Define $L_m$ by
\begin{equation}
L_{m}=\iota^\dagger_m L_k,
\end{equation}
where $\displaystyle \iota_m:Y_m\hookrightarrow Y$
is an inclusion.
If we denote by $j_m:X_m\hookrightarrow Y_m$ the open immersion and by $i_m:X_{m-1}\hookrightarrow Y_m$ the closed immersion, we have a triangle

\begin{equation}\label{sdt2}
\int_{i_m}i_m^{\dagger}L_m\rightarrow L_m\rightarrow\int_{j_m}j_m^\dagger L_m\overset{+1}{\rightarrow}.
\end{equation}

\noindent
and a commutative diagram
\begin{equation}
\xymatrix{
& X_{m-1} \ar[r]^{i_m} \ar[dr]_{\pi_{m-1}}&Y_m \ar[d]^{\tilde{\pi}_m}& X_m\ar[dl]^{\pi_m} \ar[l]_{j_m}\\
&                                 &\A^N_z& .}
\end{equation}

\noindent
Here, $\pi_m$ and $\tilde{\pi}_m$ are restrictions of the canonical projection $\varpi:Y\rightarrow\mathbb{A}^N_z.$ Applying $\int_{\tilde{\pi}_{m}}$ to $(\ref{sdt2})$ yields to a distinguished triangle

\begin{equation}
\int_{\pi_{m-1}}i_m^\dagger L_m\rightarrow \int_{\tilde{\pi}_{m}}L_m\rightarrow\int_{\pi_m}j^\dagger_mL_{m}\overset{+1}{\rightarrow}.
\end{equation}

\noindent
Remember that 
\begin{equation}
\displaystyle Y_m=(\Gm)_{y_m}\times\left[ \left( (\Gm)^{m-1}_{y_1,\cdots,y_{m-1}}\times (\mathbb{A}^N_z\times(\Gm)_x^n\setminus\{ h_{m+1,z^{(m+1)}}\cdots h_{k,z^{(k)}}=0\})\right)\setminus\bigcup_{l=1}^{m-1}\{ 1=y_lh_{l,z^{(l)}}\}\right]
\end{equation}
and 
\begin{equation}
L_m=\mathcal{O}y_m^{\gamma_m}\boxtimes\mathcal{O}y_1^{\gamma_1}\cdots y_{m-1}^{\gamma_{m-1}}h_{m+1,z^{(m+1)}}^{-\gamma_{m+1}}\cdots h_{k,z^{(k)}}^{-\gamma_k}x^c e^{h_{0,z^{(0)}}(x)}.
\end{equation}
Therefore, $\tilde{\pi}_m$ decomposes as 
\begin{equation}
Y_m\overset{p_m}{\rightarrow}\left( (\Gm)^{m-1}_{y_1,\cdots,y_{m-1}}\times (\mathbb{A}^N_z\times(\Gm)_x^n\setminus\{ h_{m+1,z^{(m+1)}}\cdots h_{k,z^{(k)}}=0\})\right)\setminus\bigcup_{l=1}^{m-1}\{ 1=y_lh_{l,z^{(l)}}\}\overset{p_m^\prime}{\rightarrow}\A^N_z
\end{equation}

\noindent
If we denote by $pt:(\Gm)_{y_m}\rightarrow\{ *\}$ the morphism to one point, we have
\begin{equation}
\int_{p_m}L_m\simeq \left(\int_{pt}\mathcal{O}y_m^{\gamma_m}\right)\boxtimes\mathcal{O}y_1^{\gamma_1}\cdots y_{m-1}^{\gamma_{m-1}}h_{m+1,z^{(m+1)}}^{-\gamma_{m+1}}\cdots h_{k,z^{(k)}}^{-\gamma_k}x^c e^{h_{0,z^{(0)}}(x)}.
\end{equation}
\noindent
On the other hand, if $\gamma_m\notin\mathbb{Z},$ we have by Lemma \ref{gamma},
\begin{equation}
\Homo^*\left(\int_{pt}\mathcal{O}y_m^{\gamma_m}\right)=\mathbb{H}_{dR}^{1+*}\Big((\Gm)_{y_m},\mathcal{O}y_m^{\gamma_m}\Big)=0.
\end{equation}
In summary, if $\gamma_m\notin\mathbb{Z},$ we have a canonical isomorophism
\begin{equation}
\int_{\tilde{\pi}_m}L_m\simeq \int_{p_m^\prime}\int_{p_m}L_m\simeq 0.
\end{equation}

\noindent
This implies that 

\begin{equation}
\int_{\pi_m}j^\dagger_mL_{m}=\int_{\pi_{m-1}}i_m^\dagger L_m[1]
\end{equation}
In view of the relation
$
i_m^{\dagger}L_m=j^{\dagger}_{m-1}L_{m-1},
$
 we inductively have a formula
\begin{equation}
\int_{\pi_k}j_k^\dagger=\int_{\pi_0}i_1^\dagger L_1.
\end{equation}

\noindent
Now we have
\begin{equation}
X_0=\displaystyle\bigcap_{l=1}^k S_l\simeq \mathbb{A}^N_z\times(\Gm)^n_x\setminus\{ h_{1,z^{(1)}}\cdots h_{k,z^{(k)}}=0\}
\end{equation}
and 
\begin{equation}
X_k=Y\setminus (S_1\cup\cdots\cup S_k).
\end{equation}
Thus, we obtained the desired formula.
\end{proof}

Finally, let us refer to the result of Schulze and Walther (\cite{SW}) which relates $M_A(c)$ for non-resonant parameters to Laplace-Gauss-Manin connection. It is stated in the following form.

\begin{thm}[\cite{SW}]
Let $\phi:(\Gm)^n_x\rightarrow\A^N$ be a morphism defined by $\phi(x)=(x^{{\bf a}(1)},\dots,x^{{\bf a}(N)})$. If $c$ is non-resonant, one has a canonical isomorphism 
\begin{equation}\label{SWisom}
M_A(c)\simeq \FL\circ\int_\phi\mathcal{O}_{(\Gm)^n}x^c,
\end{equation}
where $\FL$ stands for Fourier-Laplace transform.
\end{thm}

For readers' convenience, we include a proof of an isomorphism which rewrites the right-hand side of (\ref{SWisom}) as a direct image of an integrable connection. The readers may find a similar argument in \cite{ET}.

\begin{prop}
Let $f_j\in\mathcal{O}(X)\setminus\C$ $(j=1,\dots,p)$ be non-constant regular functions. Put $f=(f_1,\dots,f_p):X\rightarrow\A^p_\zeta$. Define the Fourier-Laplace transform $\FL:D^b_{q.c.}(\mathcal{D}_{\A^p_\zeta})\rightarrow D^b_{q.c.}(\mathcal{D}_{\A^p_z})$ by the formula 
\begin{equation}
\FL(N)=\int_{\pi_z}(\mathbb{L}\pi^*_\zeta N)\overset{\mathbb{D}}{\otimes}\mathcal{O}_{\A^p_\zeta\times\A^p_z}e^{z\cdot\zeta},
\end{equation}
where 
\begin{equation}
\A^p_z\overset{\pi_z}{\leftarrow}\A^p_z\times\A^p_\zeta\overset{\pi_\zeta}{\rightarrow}\A^p_\zeta.
\end{equation}
Let $\pi:X\times\A^p_z\rightarrow\A^p_z$ be the canonical projection. Under these settings, we have, for any $M\in D^b_{q.c.}(\DD_X),$
\begin{equation}
\FL\Big(\int_f M\Big)\simeq\int_\pi\left\{ (M\boxtimes\mathcal{O}_{\A^p_z})\overset{\mathbb{D}}{\otimes}(\mathcal{O}_{X\times\A^p_z}e^{\sum_{j=1}^pz_jf_j})\right\}.
\end{equation}
\end{prop}
\begin{proof}
Consider the following commutative diagram

\begin{equation}
\xymatrix{
& X\times\A^p_z \ar[d]_{\pi } \ar[r]^{f\times \id}& \A^p_\zeta\times\A^p_z \ar[dl]_{\pi_z}\\
& \A^p_z  &
}.
\end{equation}
\noindent
By the projection formula, we have a canonical isomorphism
\begin{align}
\FL\Big(\int_fM\Big)&\simeq\int_{\pi_z}\left\{ \Big(\left(\int_fM\right)\boxtimes\mathcal{O}_{\A^p_z}\Big)\overset{\mathbb{D}}{\otimes}\mathcal{O}_{\A^p_z\times\A^p_\zeta}e^{z\cdot\zeta }\right\}\\
 &\simeq\int_{\pi_z}\left\{ \Big(\int_{f\times\id_z}M\boxtimes\mathcal{O}_{\A^p_z}\Big)\overset{\mathbb{D}}{\otimes}\mathcal{O}_{\A^p_z\times\A^p_\zeta}e^{z\cdot\zeta }\right\}\\
 &\simeq\int_\pi\left\{ \left(M\boxtimes\mathcal{O}_{\A^p_z}\right)\overset{\mathbb{D}}{\otimes}(\mathcal{O}_{X\times\A^p_z}e^{\sum_{j=1}^pz_jf_j})\right\}.
\end{align}

\end{proof}
\noindent
If we take $X$ to be $(\Gm)^n_x,$ $M$ to be $\mathcal{O}_{(\Gm)^n_x}x^c,$ and $f$ to be $f=(x^{{\bf a}(1)},\dots,x^{{\bf a}(N)}),$ we have
\begin{equation}
\FL\Big(\int_f\mathcal{O}_{(\Gm)^n_x}x^c\Big)\simeq\int_\pi\mathcal{O}_{(\Gm)^n_x\times\A^N_z}x^ce^{h_{z}(x)},
\end{equation}
where $h_z(x)=\displaystyle\sum_{j=1}^Nz_jx^{{\bf a}(j)}.$

\begin{cor}
If $c$ is non-resonant, one has a canonical isomorphism
\begin{equation}
M_A(c)\simeq\int_\pi\mathcal{O}_{(\Gm)^n\times\A^N}x^ce^{h_z(x)}.
\end{equation}
\end{cor}

\noindent
Summing up all the argument above, we have the following

\begin{thm}\label{thm:mainDresult}
Suppose that the parameter 
$d=\begin{pmatrix}
\gamma\\
c
\end{pmatrix}
$ is non-resonant and $\gamma_l\notin\Z$ for $l=1,\dots,k$. We put $N=N_0+N_1+\cdots+N_k$ and define a $(n+k)\times N$ matrix $A$ by 

\begin{equation}
A
=
\left(
\begin{array}{ccc|ccc|ccc|c|ccc}
0&\cdots&0&1&\cdots&1&0&\cdots&0&\cdots&0&\cdots&0\\
\hline
0&\cdots&0&0&\cdots&0&1&\cdots&1&\cdots&0&\cdots&0\\
\hline
&\vdots& & &\vdots& & &\vdots& &\ddots& &\vdots& \\
\hline
0&\cdots&0&0&\cdots&0&0&\cdots&0&\cdots&1&\cdots&1\\
\hline
&A_0& & &A_1& & &A_2& &\cdots & &A_k& 
\end{array}
\right).
\end{equation}
Then, one has a sequence of canonical isomorphisms
\begin{align}
M_A(d)&\simeq\int_\varpi\mathcal{O}_{\A^N_z\times(\Gm)_y^k\times(\Gm)_x^n}y^\gamma x^c e^{h_z(y,x)}\\
 &\simeq\int_\pi\mathcal{O}_{\A^N_z\times (\Gm)_x^n\setminus\lef h_{1,z^{(1)}}\cdots h_{k,z^{(k)}}=0\righ }h_1^{-\gamma_1}\cdots h_k^{-\gamma_k}x^c e^{h_{0,z^{(0)}}(x)}\\
 &\simeq\int_{\tilde{\varpi}_z}\mathcal{O}_{\A^N_z\times(\Gm)^k_y\times(\Gm)_x^n\setminus\lef   (1-y_1h_{1,z^{(1)}})\cdots(1- y_kh_{k,z^{(k)}})=0\righ}y^\gamma x^ce^{h_{0,z^{(0)}}(x)}.
\end{align}
\end{thm}

As an application of this isomorphism, we can prove the following result which is analogous to Theorem 2.10. of \cite{GKZ} and whose proof was sketched in \cite{B}\S6. In order to formulate and prove it, we revise some notation. For any smooth algebraic varieties $X$ and $Y$ and for any morphism $f:X\rightarrow Y$, a functor $\int_{f!}:D^b_h(\DD_X)\rightarrow D^b_h(\DD_Y)$ (resp. $\int_{f!}:D^b_{r. h.}(\DD_X)\rightarrow D^b_{r.h.}(\DD_Y)$) is defined by
\begin{equation}\label{proper}
\int_{f!}=\mathbb{D}_Y\circ\int_f\circ\mathbb{D}_X,
\end{equation}
where $\mathbb{D}_X$ is the holonomic dual functor. If $DR^{an}_X$ denotes analytic de-Rham functor, we have a relation
\begin{equation}\label{commutative}
\sol_X[\dim X]\simeq DR^{an}\circ\mathbb{D}_X
\end{equation}
as a functor from $D^b_h(\DD_X)$ to $D^b_c(\C_X)$ and a commutativity relation
\begin{equation}\label{RH}
DR^{an}_Y\int_{f!}\simeq\R f_!\circ DR^{an}_X
\end{equation}
as a functor from $D^b_{r.h.}(\DD_X)$ to $D^b_c(\C_Y)$. The second commutativity is a part of Riemann-Hilbert correspondence. 

Now we need to remember the definition of the Newton non-degenerate locus.
\begin{defn} 
Suppose $A$ is an $n\times N$ integer matrix $A=({\bf a}(1)\mid \cdots\mid{\bf a}(N))$. We put $h_z(x)=\displaystyle\sum_{j=1}^Nz_jx^{{\bf a}(j)}.$ For any face $\Gamma<\Delta_A\overset{\rm def}{=}c.h.\{0,{\bf a}(1),\dots,{\bf a}(N)\}$, we put $h_z^\Gamma(x)=\displaystyle\sum_{{\bf a}(j)\in\Gamma}z_jx^{{\bf a}(j)}$. Then, we say $z\in\mathbb{A}^N$ belongs to the Newton non-degenerate locus of $A$ if, for any face $\Gamma<\Delta_A$ such that $0\notin\Gamma$, we have an identity
\begin{equation}\label{NewtonNonDeg}
\left\{ x\in(\Gm)^n\mid \frac{\partial h_z^\Gamma}{\partial x_1}(x)=\cdots=\frac{\partial h_z^\Gamma}{\partial x_n}(x)=0\right\}=\varnothing.
\end{equation}
The set of Newton non-degenerate points $z$ of $A$ is denoted by $\Omega$.
\end{defn}
The following result was first proved by A. Adolphson.
\begin{thm}[\cite{A}, COROLLARY 3.8.]
For any parameter vector $c\in\C^{n\times 1}$, the GKZ system $M_A(c)$ is a connection on $\Omega$.
\end{thm}

We are now able to state and prove
\begin{thm}\label{thm:ResidueIntegral}
Suppose $h_{0,z^{(0)}}(x)=0.$ If the parameter $d$ is non-resonant and if there is a simplex $\sigma$ such that $\displaystyle\sum_{i\in\sigma^{(l)}}{}^t{\bf e}_iA_\sigma^{-1}d\notin\Z_{\leq 0}$ for any $l=1,\dots,k$ and that $A_\s^{-1}(d+A_{\bs}\Z^{\bs}_{\geq 0})\cap(\C\setminus\Z_{>0})^\s\neq\varnothing$, then, one has an isomorphism of local systems
\begin{equation}\label{Integration}
\int:R^{n+k}\varpi_!(\C_{\tilde{Y}^{an}} y^\gamma x^c)\restriction_{\Omega^{an}}\ni \Gamma\mapsto\int_\Gamma\frac{y^{\gamma-1} x^{c-1}dydx}{(1-y_1h_{1,z^{(1)}}(x))\cdots(1-y_kh_{k,z^{(k)}}(x))}\in\sol_{M_A(c)}\restriction_{\Omega^{an}}.
\end{equation}
\end{thm}

\begin{proof}
By (\ref{proper}), (\ref{commutative}), and (\ref{RH}), we have
\begin{align}
\sol_{\C^N}\left(\int_{\tilde{\varpi}_z}\mathcal{O}_{\tilde{Y}}y^\gamma x^c\right)&\simeq DR^{an}_{\C^N}\int_{\tilde{\varpi}_{z!}}\mathbb{D}_{\tilde{Y}}\left(\mathcal{O}_{\tilde{Y}}y^\gamma x^c\right)[-N]\\
 &\simeq\R_{\tilde{\varpi}_{z!}}DR^{an}_{\tilde{Y}}\mathbb{D}_{\tilde{Y}}\left(\mathcal{O}_{\tilde{Y}}y^\gamma x^c\right)[-N]\\
 &\simeq\R_{\tilde{\varpi}_{z!}}\sol_{\tilde{Y}}\left(\mathcal{O}_{\tilde{Y}}y^\gamma x^c\right)[k]\\
 &\simeq\R_{\tilde{\varpi}_{z!}}\left(\C_{\tilde{Y}}y^\gamma x^c\right)[k]\label{SheafIsom}.
\end{align}
(\ref{SheafIsom}) combined with \cref{thm:mainDresult} shows that $R^{n+k}_{\tilde{\varpi}_{z!}}\left(\C_{\tilde{Y}}y^\gamma x^c\right)\restriction_{\Omega^{an}}$ and $\sol_{M_A(d)}\restriction_{\Omega^{an}}$ are local systems isomorphic to each other and the morphism (\ref{Integration}) is well-defined by \cref{prop:ResIntIsSol}. Moreover, by Theorem 2.15 of \cite{GKZ} and \cref{thm:mainDresult}, we know that they are both irreducible since the parameter $d$ is non-resonant. Thus, it is enough to prove that the morphism (\ref{Integration}) is non-zero by Schur's lemma. The morphism is indeed non-zero by \cref{thm:fundamentalthm2}.
\end{proof}

\begin{rem}
We can prove an isomorphism of constructible (perverse) sheaves
\begin{equation}
\R\pi_{z!}\left(\C_{X^{an}}\cdot h_{1,z^{(1)}}^{-\gamma_1}\cdots h_{k,z^{(k)}}^{-\gamma_k}x^c\right)\simeq\R\tilde{\varpi}_{z!}\left( \C_{\tilde{Y}^{an}}y^\gamma x^c\right)[k]
\end{equation}
in the same way. This corresponds to the classical composed coboundary operation of J.Leray (\cite{L}).
\end{rem}

\end{section}

\begin{section}{Construction of integration contours for Mixed type integrals}

We conclude this paper with a formula which relates a basis of cycles of (\ref{MixedResidueInt}) (resp. (\ref{MixedInt})) to $\Gamma$-series. The computations can be carried out as in \S\ref{SectionLaplace}, \S\ref{SectionResidue} with a slight modification.

\begin{prop}
(\ref{MixedResidueInt}) is a solution of $M_A(d).$
\end{prop} 
\end{section}

\begin{proof}
Same as \cref{prop:ResIntIsSol}.
\end{proof}

Now we are going to construct a standard basis of integral contours associated to a simplex. Take any $n+k$ simplex $\s\subset\{1,\dots,N\}$ such that $s_j\leq 1$ for all $j\in\bs$. We put

\begin{equation}
f(z)_{\s,0}=\frac{1}{(2\pi\sqrt{-1})^{m_\s}}\int_\Gamma\frac{e^{h_{0,z^{(0)}}(x)}y^{\gamma -1}x^{c-1}}{\Big(1-y_1h_{1,z^{(1)}}(x)\Big)\cdots \Big(1-y_kh_{k,z^{(k)}}(x)\Big)}dydx,
\end{equation}
where $\Gamma$ is an integration cycle to be specified and $m_\s$ is an integer defined by the formula
\begin{equation}
m_\s=
\begin{cases}
n+2k&(\s^{(0)}=\varnothing)\\
n+2k+1&(\s^{(0)}\neq\varnothing).
\end{cases}
\end{equation}
We consider a covering map
\begin{equation}
\T^{n+d}_{\tilde{x}}\rightarrow\T^\sigma_{\xi_\sigma}
\end{equation}
given by
\begin{equation}
\tilde{x}\mapsto \xi_\sigma=z_\s\tilde{x}^{A_\sigma}. 
\end{equation}
\noindent
Then, if we write $\tilde{x}=(y,x)$ and $\tilde{h}_{l,z^{(l)}}(\tilde{x})=y_lh_{l,z^{(l)}}(x)$, and if we assume our integration cycle is a pullback $\Gamma=p^*\gamma$, we have
\begin{equation}\label{TheIntegral}
f_{\s,0}(z)=\frac{z_\sigma^{-A_\sigma^{-1}d}}{(2\pi\sqrt{-1})^{m_\s}}\int_\gamma
\frac{
\exp\left\{\displaystyle\sum_{i\in\sigma^{(0)}}\xi_i+\sum_{j\in\bar{\sigma}^{(0)}}(z_\sigma^{-A_\sigma^{-1}{\bf a}(j)}z_j)\xi_\sigma^{A_\sigma^{-1}{\bf a}(j)}\right\}
\xi_\sigma^{A_\sigma^{-1}d-{\bf 1}}
}
{
\displaystyle
\prod_{l=1}^k
\left(
1-\sum_{i\in\sigma^{(l)}}\xi_i-\sum_{j\in\bar{\sigma}^{(l)}}(z_\sigma^{-A_\sigma^{-1}{\bf a}(j)}z_j)\xi_\sigma^{A_\sigma^{-1}{\bf a}(j)}
\right)
}d\xi_\sigma.
\end{equation}
\noindent
At this stage, we can observe that the integral (\ref{TheIntegral}) has a divergent nature in $\xi_{\sigma^{(0)}}$ direction and convergent nature in $\xi_{\sigma^{(l)}}\;\;(l=1,\dots,k)$ directions. Let us introduce the plane wave coordinate with respect to $\xi_{\sigma^{(0)}},$ namely we introduce new coordinate $(\rho,u_{\sigma^{(0)}})$ of $\T^{\sigma^{(0)}}$ defined by 
\begin{equation}
\xi_i=\rho u_i\;\;(i\in\sigma^{(0)}),\;\;\displaystyle\sum_{i\in\sigma^{(0)}}u_i=1.
\end{equation}
(\ref{TheIntegral}) is transformed into the form
\begin{align}
 &\frac{z_\sigma^{-A_\sigma^{-1}d}}{(2\pi\sqrt{-1})^{m_\s}}\times\nonumber\\
 &\int_\gamma
\frac{
\exp\left\{\rho+\displaystyle\sum_{j\in\bar{\sigma}^{(0)}}(z_\sigma^{-A_\sigma^{-1}{\bf a}(j)}z_j)\rho^{\sum_{i\in\sigma^{(0)}}{}^t{\bf e}_iA_\sigma^{-1}{\bf a}(j)}\displaystyle\prod_{i\in\sigma^{(0)}}u_{i}^{{}^t{\bf e}_iA_\sigma^{-1}{\bf a}(j)}\prod_{l=1}^k\displaystyle\prod_{i\in\sigma^{(l)}}\xi_{i}^{{}^t{\bf e}_iA_\sigma^{-1}{\bf a}(j)}\right\}
}
{
\displaystyle
\prod_{l=1}^k
\left(
1-\sum_{i\in\sigma^{(l)}}\xi_i-\sum_{j\in\bar{\sigma}^{(l)}}(z_\sigma^{-A_\sigma^{-1}{\bf a}(j)}z_j)\rho^{\sum_{i\in\sigma^{(0)}}{}^t{\bf e}_iA_\sigma^{-1}{\bf a}(j)}\prod_{i\in\sigma^{(0)}}u_{i}^{{}^t{\bf e}_iA_\sigma^{-1}{\bf a}(j)}\prod_{l=1}^k\prod_{i\in\sigma^{(l)}}\xi_{i}^{{}^t{\bf e}_iA_\sigma^{-1}{\bf a}(j)}
\right)
}\times\nonumber\\
 &\rho^{\sum_{i\in\sigma^{(0)}}{}^t{\bf e}_iA_\sigma^{-1}{\bf d}-1}\displaystyle\prod_{i\in\sigma^{(0)}}u_{i}^{{}^t{\bf e}_iA_\sigma^{-1}{\bf d}-1}\prod_{l=1}^k\prod_{i\in\sigma^{(l)}}\xi_{i}^{{}^t{\bf e}_iA_\sigma^{-1}{\bf d}-1}
d\rho\wedge du_{\sigma^{(0)}}\wedge d\xi_{\sigma^{(1)}}\wedge\cdots\wedge d\xi_{\sigma^{(k)}}.\label{TheIntegral2}
\end{align}
The following lemma is analogous to \cref{lem:sum}.

\begin{lem}\label{lem:sum2}
For any $l=1,\cdots,k$ and for any $j\in\barsigma^{(l)}$, one has
\begin{equation}
\sum_{i\in\sigma^{(m)}}{}^t{\bf e}_iA_\sigma^{-1}{\bf a}(j)=
\begin{cases}
1\; (m=l)\\
0\; (m\neq 0,l).
\end{cases}
\end{equation}
Moreover, if $j\in\barsigma^{(0)}$, one has
\begin{equation}
\sum_{i\in\sigma^{(m)}}{}^t{\bf e}_iA_\sigma^{-1}{\bf a}(j)=
0\;\; (m=1,\dots,k). 
\end{equation}
\end{lem}
\noindent
The proof is same as \cref{lem:sum}. From \cref{lem:sum2} and the equality
\begin{equation}
\displaystyle\sum_{m=0}^k\sum_{i\in\sigma^{(m)}}{}^t\tilde{\bf e}_iA_\sigma^{-1}{\bf a}(j)=|A_\sigma^{-1}{\bf a}(j)|=s_j,
\end{equation}
we obtain two equations on the degree of divergence
\begin{equation}\label{equality1}
\displaystyle\sum_{i\in\sigma^{(0)}}{}^t{\bf e}_iA_\sigma^{-1}{\bf a}(j)=s_j-1\;\;(j\in\barsigma^{(l)},\; l=1,\dots,k)
\end{equation}
and
\begin{equation}\label{equality2}
\displaystyle\sum_{i\in\sigma^{(0)}}{}^t{\bf e}_iA_\sigma^{-1}{\bf a}(j)=s_j\;\;(j\in\barsigma^{(0)}).
\end{equation}
Then we are naturally led to take the integration contour $\gamma$ as a product 
\begin{equation}
\gamma=\Gamma_{\sigma^{(0)},0}\times P_{\sigma^{(1)}}\times\cdots\times P_{\sigma^{(k)}},
\end{equation}
where $\Gamma_{\sigma^{(0)},0}$ is the product of Hankel contour and Pochhammer cycle associated to $\sigma^{(0)}$ and $P_{\sigma^{(l)},0}$ is the Pochhammer cycle associated to $\sigma^{(l)}$. By equations (\ref{equality1}) and (\ref{equality2}), we have two inequalities
\begin{equation}\label{equality1}
\displaystyle\sum_{i\in\sigma^{(0)}}{}^t{\bf e}_iA_\sigma^{-1}{\bf a}(j)\leq 0\;\;(j\in\barsigma^{(l)},\; l=1,\dots,k)
\end{equation}
and
\begin{equation}\label{equality2}
\displaystyle\sum_{i\in\sigma^{(0)}}{}^t{\bf e}_iA_\sigma^{-1}{\bf a}(j)\leq 1\;\;(j\in\barsigma^{(0)}),
\end{equation}
which combined with the formula (\ref{TheIntegral2}) ensure the convergence of (\ref{TheIntegral}) when $z\in U_\sigma$.
\noindent
Substituting the formula
\begin{equation}
\frac{1}
{
\displaystyle
\left(
1-\sum_{i\in\sigma^{(l)}}\xi_i-\sum_{j\in\bar{\sigma}^{(l)}}(z_\sigma^{-A_\sigma^{-1}{\bf a}(j)}z_j)\xi_\sigma^{A_\sigma^{-1}{\bf a}(j)}
\right)
}
=
\sum_{{\bf m}_l\in\mathbb{Z}^{\bar{\sigma}_l}_{\geq 0}}
\frac{|{\bf m}_l|!}{{\bf m}_l!}
\frac{
\xi_\sigma^{A_\sigma^{-1}A_{\bar{\sigma}^{(l)}}{\bf m}_l}
(z_\sigma^{-A_\sigma^{-1}A_{\bar{\sigma}^{(l)}}}z_{\bar{\sigma}^{(l)}})^{{\bf m}_l}
}
{
\left(
\displaystyle
1-\sum_{i\in\sigma^{(l)}}\xi_i
\right)^{|{\bf m}_l|+1}
},
\end{equation}
to (\ref{TheIntegral}), and expanding the term

\begin{equation}
\exp\left\{ \displaystyle\sum_{j\in\bar{\sigma}^{(0)}}(z_\sigma^{-A_\sigma^{-1}{\bf a}(j)}z_j)\xi_\sigma^{A_\sigma^{-1}{\bf a}(j)}\right\},
\end{equation}
\noindent
we obtain an expansion

\begin{align}
 f_{\s,0}(z)&=
\frac{
z_\sigma^{-A_\sigma^{-1}d}
}
{
(2\pi\ii)^{m_\s}
}
\displaystyle
\sum_{{\bf m}\in\Z^{\bar{\sigma}}_{\geq 0}}
\int_{\gamma} d\xi_\sigma\Big(\exp\left\{\displaystyle\sum_{i\in\sigma^{(0)}}\xi_i\right\}
\prod_{i\in\sigma^{(0)}}\xi_\sigma^{
{}^t{\bf e}_iA_\sigma^{-1}(d+A_{\bar{\sigma}}
{\bf m})
-1}
\Big)\times\nonumber\\
 & \quad
\displaystyle\prod_{l=1}^k
\Big(
(1-\sum_{i\in\sigma^{(l)}}\xi_i)^{-|{\bf m}_l|-1}
\prod_{i\in\sigma^{(l)}}\xi_\sigma^{
{}^t{\bf e}_iA_\sigma^{-1}(d+A_{\bar{\sigma}}
{\bf m})
-1}
\Big)
\frac{|{\bf m}_1|!\cdots|{\bf m}_k|!}{{\bf m}!}
(z_\sigma^{-A_\sigma^{-1}A_{\bar{\sigma}}}z_{\bar{\sigma}})^{{\bf m}}.
\end{align}

\noindent
As in \S\ref{SectionLaplace} and \S\ref{SectionResidue}, we are reduced to computing the integrals 
\begin{equation}
{\rm I}_0({\bf m})=\int_{\Gamma_{\sigma^{(0)},0}} \exp\left\{\displaystyle\sum_{i\in\sigma^{(0)}}\xi_i\right\}
\prod_{i\in\sigma^{(0)}}\xi_\sigma^{
-{}^t{\bf e}_iA_\sigma^{-1}(d+A_{\bar{\sigma}}
{\bf m})
-1}
d\xi_{\sigma^{(0)}}
\end{equation}
\noindent
and

\begin{equation}
{\rm I}_l({\bf m})=\int_{P_{\sigma^{(l)}}}
(1-\sum_{i\in\sigma^{(l)}}\xi_i)^{-|{\bf m}_l|-1}
\prod_{i\in\sigma^{(l)}}\xi_\sigma^{
{}^t{\bf e}_iA_\sigma^{-1}(d+A_{\bar{\sigma}}
{\bf m})
-1}d\xi_{\sigma^{(l)}}.
\end{equation}

\noindent
By \cref{lemma:lemma} and \cref{lemma:Beukers}, they can be computed as

\begin{equation}
{\rm I}_0({\bf m})=(2\pi\sqrt{-1})^{|\sigma^{(0)}|+1}\frac{
\Bigg(1-\exp\left\{-2\pi\sqrt{-1}\displaystyle\sum_{i\in\sigma^{(0)}}{}^t{\bf e}_iA_\sigma^{-1}(d+A_{\bar{\sigma}}
{\bf m})\right\}\Bigg)
}
{
\displaystyle\prod_{i\in\sigma^{(0)}}\Gamma(1-{}^t{\bf e}_iA_\sigma^{-1}(d+A_{\bar{\sigma}}
{\bf m}))
}
\end{equation}

\noindent
and

\begin{equation}
{\rm I}_l({\bf m})=\frac{\exp\left\{ -\pi\sqrt{-1}\displaystyle\sum_{i\in\sigma^{(l)}}{}^t{\bf e}_iA_\sigma^{-1}d\right\}}{
\Gamma(\displaystyle\sum_{i\in\sigma^{(l)}}{}^t{\bf e}_iA_\sigma^{-1}d)
}
\frac{
(2\pi\sqrt{-1})^{|\sigma^{(l)}|+1}
}
{
|{\bf m}_l|!
}
\frac{
1
}
{
\displaystyle\prod_{i\in\sigma^{(l)}}\Gamma(1-{}^t{\bf e}_iA_\sigma^{-1}(d+A_{\bar{\sigma}}
{\bf m}))
}.
\end{equation}

\noindent
Thus, we have

\begin{align}
f_{\s,0}(z) &=
\frac{
\exp\left\{ -\pi\sqrt{-1}\displaystyle\sum_{l=1}^k\displaystyle\sum_{i\in\sigma^{(l)}}{}^t{\bf e}_iA_\sigma^{-1}d\right\}
}
{
\displaystyle\prod_{l=1}^k\Gamma(\sum_{i\in\sigma^{(l)}}{}^t{\bf e}_iA_\sigma^{-1}d)
}
z_\sigma^{-A_\sigma^{-1}d}
\displaystyle
\sum_{{\bf m}\in\Z^{\bar{\sigma}}_{\geq 0}}\nonumber\\
 & \quad
\Bigg(1-\exp\left\{-2\pi\sqrt{-1}\displaystyle\sum_{i\in\sigma^{(0)}}{}^t{\bf e}_iA_\sigma^{-1}(d+A_{\bar{\sigma}}
{\bf m})\right\}\Bigg)
\frac
{
(z_\sigma^{-A_\sigma^{-1}A_{\bar{\sigma}}}z_{\bar{\sigma}})^{{\bf m}}
}
{
\Gamma({\bf 1}-A_\sigma^{-1}(d+A_{\bar{\sigma}}
{\bf m})){\bf m}!
}\\
 &=
\frac{
\exp\left\{ -\pi\sqrt{-1}\displaystyle\sum_{l=1}^k\displaystyle\sum_{i\in\sigma^{(l)}}{}^t{\bf e}_iA_\sigma^{-1}d\right\}
}
{
\displaystyle\prod_{l=1}^k\Gamma(\sum_{i\in\sigma^{(l)}}{}^t{\bf e}_iA_\sigma^{-1}d)
}\times\nonumber\\
&\quad
\displaystyle
\sum_{i=1}^r
\tiny
\Bigg(1-\exp\left\{-2\pi\sqrt{-1}\sum_{i\in\sigma^{(0)}}{}^t{\bf e}_iA_\sigma^{-1}(d+A_{\bar{\sigma}}
{\bf k}(i))\right\}\Bigg)
\normalsize
\varphi_{\sigma,{\bf k}(i)}.
\end{align}

\noindent
We denote $\Gamma_{\s,0}$ the integration cycle above. Just as in \S\ref{SectionLaplace}, \S\ref{SectionResidue}, and \S\ref{SectionEuler}, for any $\tilde{\bf k}\in\Z^{n+k}$, we consider a deck transformation $\Gamma_{\s,\tilde{\bf k}}$ of $\Gamma_{\s,0}$ associated to $\xi_\s\mapsto e^{2\pi\ii{}^t\tilde{\bf k}}\xi_\s$ and put 
\begin{equation}
f_{\s,\tilde{\bf k}}=\frac{1}{(2\pi\sqrt{-1})^{m_\s}}\int_{\Gamma_{\s,\tilde{\bf k}}}\frac{e^{h_{0,z^{(0)}}(x)}y^{\gamma -1}x^{c-1}}{\Big(1-y_1h_{1,z^{(1)}}(x)\Big)\cdots \Big(1-y_kh_{k,z^{(k)}}(x)\Big)}dydx.
\end{equation}
Computing as in \S\ref{SectionLaplace}, \S\ref{SectionResidue}, and \S\ref{SectionEuler}, we can obtain a mixed version of \cref{thm:fundamentalthm1} and \cref{thm:fundamentalthm2}.

\begin{thm}\label{thm:fundamentalthm4}
Take a regular triangulation $T$ of $A$ such that for any simplex $\s\in T$, one has $s_j\leq 1$ for any $j\in\barsigma$. Assume that $\Z A=\Z^{n+k}$, the parameter vector $d$ is very generic with respect to any $\sigma\in T$, and that for any $l=1,\dots,k$, one has $\displaystyle\sum_{i\in\sigma^{(l)}}{}^t{\bf e}_iA_\sigma^{-1}d\notin\Z_{\leq 0}$. Then, if one puts
\begin{equation}
f_{\sigma,\tilde{\bf k}(j)}(z)=\frac{1}{(2\pi\sqrt{-1})^{m_\s}}\int\frac{e^{h_{0,z^{(0)}}(x)}y^{\gamma -1}x^{c-1}}{(1-y_1h_{1,z^{(1)}}(x))\cdots (1-y_kh_{k,z^{(k)}}(x))}dydx,
\end{equation}
$\bigcup\{ f_{\sigma,\tilde{\bf k}(j)}(z)\}_{j=1}^{r(\sigma)}$ is a basis of solutions of $M_A(c)$ on the non-empty open set $U_T$, where $\{\tilde{\bf k}(j)\}_{j=1}^{r}$ is a complete system of representatives $\Z^{n+k}/\Z{}^tA_\sigma$. Moreover, for each $\sigma\in T,$ one has a transformation formula 
\begin{equation}
\begin{pmatrix}
f_{\sigma,\tilde{\bf k}(1)}(z)\\
\vdots\\
f_{\sigma,\tilde{\bf k}(r)}(z)
\end{pmatrix}
=
T_\sigma
\begin{pmatrix}
\varphi_{\sigma,v^{{\bf k}(1)}}(z)\\
\vdots\\
\varphi_{\sigma,v^{{\bf k}(r)}}(z)
\end{pmatrix}.
\end{equation}
Here, $T_\sigma$ is an $r\times r$ matrix given by 
\begin{align}
T_\sigma&=\frac{
\exp\left\{ -\pi\sqrt{-1}\displaystyle\sum_{l=1}^k\displaystyle\sum_{i\in\sigma^{(l)}}{}^t{\bf e}_iA_\sigma^{-1}d\right\}
}
{
\displaystyle
\prod_{l=1}^k
\Gamma\left(
\sum_{i\in\sigma^{(l)}}{}^t\tilde{\bf e}_iA_\sigma^{-1}d
\right)
}
\diag\Big( e^{2\pi\sqrt{-1}{}^t\tilde{\bf k}(i)A_\sigma^{-1}d}\Big)_{i=1}^{r}
\Big(e^{2\pi\sqrt{-1}{}^t\tilde{\bf k}(i)A_\sigma^{-1}{\bf k}(j)}\Big)_{i,j=1}^{r}\times\nonumber\\
 & \quad \diag\Bigg(1-\exp\left\{-2\pi\sqrt{-1}\displaystyle\sum_{i\in\sigma^{(0)}}{}^t{\bf e}_iA_\sigma^{-1}(d+A_{\bar{\sigma}}{\bf k}(i))\right\}\Bigg)_{i=1}^{r}.
\end{align}
\end{thm}

The computations for (\ref{MixedInt}) is carried out in a similar way. Let us put
\begin{equation}
f_{\s,0}(z)=\frac{1}{(2\pi\ii)^{m_\s}}\int_\Gamma h_{1,z^{(1)}}(x)^{-\gamma_1}\cdots h_{k,z^{(k)}}(x)^{-\gamma_k}x^{c-1} e^{h_{0,z^{(0)}}(x)}dx,
\end{equation}
where we specify $\Gamma$ later and by abuse of notation, $m_\s$ denotes an integer
\begin{equation}
m_\s=
\begin{cases}
n+k&(\s^{(0)}=\varnothing)\\
n+k+1&(\s^{(0)}\neq\varnothing).
\end{cases}
\end{equation}
As in \S\ref{SectionEuler}, the index set $\{1,\dots,N\}$ naturally splits into several blocks as $\{1,\dots,N\}=I_0\cup\cdots\cup I_k$. Take any $n+k$ simplex $\s\subset\{1,\dots,N\}$ such that $|A_\s^{-1}{\bf a}(j)|\leq 1\;(j\in\bs)$. We put $\s^{(l)}=\s\cap I_l\;(l=0,\dots,k)$. Fix a distinguished element $i^{(l)}\in\s^{(l)}\;(l=1,\dots, k)$ and put $\s_0=\{ i^{(l)}\}_{l=1}^k$ and $\tau^{(l)}=\s^{(l)}\setminus\{ i^{(l)}\}\;(l=1,\dots,k)$. Finally, we put $\tau=\tau^{(1)}\cup\dots\cup\tau^{(k)}$Introducing a new coordinate $w_j=z_{i^{(l)}}^{-1}z_j\;(l=1,\dots,k,\; j\in\s^{(l)})$ we rewrite $f$ into a convenient form:
\begin{equation}
f_{\s,0}(z)=\frac{z_{\s_0}^{-\gamma}}{(2\pi\ii)^{m_\s}}\int_\Gamma\prod_{l=1}^k\left( 1+\sum_{j\in I_l\setminus\{ i^{(l)}\}}w_jx^{\ca(j)-\ca(i^{(l)})}\right)^{-\gamma_l}x^{c-\cA_{\s_0}\gamma-1} e^{h_{0,z^{(0)}}(x)}dx. 
\end{equation}
This corresponds to the identity $M_A(d)=M_{Q_0A}(Q_0d)$ where $Q_0$ is given by
\begin{equation}
Q_0=
\left(
\begin{array}{c|c}
I_k&O\\
\hline
-\cA_{\s_0}&I_n
\end{array}
\right).
\end{equation}
If we introduce a stair matrix

\begin{equation}
S=(\overbrace{{\bf 0}\mid\cdots\mid{\bf 0}}^{|\s^{(0)}|\text{ times}}\mid\overbrace{{\bf e_1}\mid\cdots\mid{\bf e}_1}^{|\tau^{(1)}|\text{ times}}
\mid\cdots\mid
\overbrace{{\bf e_k}\mid\cdots\mid{\bf e}_k}^{|\tau^{(k)}|\text{ times}}
)\in\Z^{k\times(\s^{(0)},\tau)},
\end{equation}
it is straightforward to check the formula
\begin{equation}
Q_0A_\s=
\left(
\begin{array}{c|c}
I_{\s_0}&S\\
\hline
O_{\s_0}&\cA_{\s^{(0)},\tau}-\cA_{\s_0}S
\end{array}
\right).
\end{equation}
We introduce a covering coordinate defined by 
\begin{numcases}{}
\xi_i=z_ix^{\ca(i)}&($i\in\s^{(0)}$)\\
\xi_i=e^{-\pi\ii}w_ix^{\ca(i)-\ca(i^{(l)})}&($i\in\tau^{(l)},\;l=1,\dots,k$).
\end{numcases}
This is abbreviated as 
\begin{equation}
p:\T^n_x\ni x\mapsto\xi_{\s^{(0)},\tau}=(z_{\s^{(0)}}x^{\cA_{\s^{(0)}}},e^{-\pi\ii{}^t{\bf 1}_\tau}w_\tau x^{\cA_\tau})\in\T^{\s^{(0)},\tau}.
\end{equation}
Suppose that $\Gamma$ is a pull-back $\Gamma=p^*\gamma$. By a direct computation, we have
\begin{align}
  &f_{\s,0}(z)\\
=&\frac{z_{\s_0}^{-\gamma}(z_{\s^{(0)}},e^{-\pi\ii{}^t{\bf 1}_\tau}w_\tau )^{-(\cA_{\s^{(0)},\tau}-\cA_{\s_0}S)^{-1}c}}{(2\pi\ii)^{m_\s}}\int_\gamma\prod_{l=1}^k\left( 1-\sum_{i\in\tau^{(l)}}\xi_i+\right.\nonumber\\
 &\left.\sum_{j\in\bs^{(l)}}(z_{\s^{(0)}},e^{-\pi\ii{}^t{\bf 1}_\tau}w_\tau )^{-(\cA_{\s^{(0)},\tau}-\cA_{\s_0}S)^{-1}(\ca(j)-\ca(i^{(l)}))}w_j\xi^{(\cA_{\s^{(0)},\tau}-\cA_{\s_0}S)^{-1}(\ca(j)-\ca(i^{(l)}))}\right)^{-\gamma_l}\nonumber\\
 &\exp\left\{   \sum_{i\in\s^{(0)}}\xi_i+\sum_{j\in\bs^{(0)}}(z_{\s^{(0)}},e^{-\pi\ii{}^t{\bf 1}_\tau}w_\tau )^{-(\cA_{\s^{(0)},\tau}-\cA_{\s_0}S)^{-1}\ca(j)}w_j\xi^{(\cA_{\s^{(0)},\tau}-\cA_{\s_0}S)^{-1}\ca(j)}\right\}\nonumber\\
 &\xi_{\s^{(0)},\tau}^{(\cA_{\s^{(0)},\tau}-\cA_{\s_0}S)^{-1}c-{\bf 1}}d\xi_{\s^{(0)},\tau}
\end{align}
Now we are going to integrate it over the cycle
\begin{equation}
\gamma=\Gamma_{\s^{(0)}}\times P_{\tau^{(1)}}\times\dots\times P_{\tau^{(k)}}.
\end{equation}
In order to ensure the convergence, we need the following
\begin{lem}
For any $l=1,\cdots,k$ and for any $j\in\barsigma^{(l)}$, one has
\begin{numcases}{}
\sum_{i\in\sigma^{(0)}}{}^t{\bf e}_i(\cA_{\s^{(0)},\tau}-\cA_{\s_0}S)^{-1}\ca(j)\leq 1& ($j\in\bs^{(0)}$) \\
\sum_{i\in\sigma^{(0)}}{}^t{\bf e}_i(\cA_{\s^{(0)},\tau}-\cA_{\s_0}S)^{-1}\ca(j)\leq 0 &
($j\in\bs^{(l)}$,\;l=1,\dots,k). 
\end{numcases}
\end{lem}
\begin{proof}
As for the first case, let us note that
\begin{equation}
A_\s^{-1}{\bf a}(j)=
\left(
\begin{array}{c|c}
I_{\s_0}&-S(\cA_{\s^{(0)},\tau}-\cA_{\s_0}S)^{-1}\\
\hline
O&(\cA_{\s^{(0)},\tau}-\cA_{\s_0}S)^{-1}
\end{array}
\right)
\begin{pmatrix}
O\\
\hline
\ca(j)
\end{pmatrix}
=
\begin{pmatrix}
-S(\cA_{\s^{(0)},\tau}-\cA_{\s_0}S)^{-1}\ca(j)\\
\hline
(\cA_{\s^{(0)},\tau}-\cA_{\s_0}S)^{-1}\ca(j)
\end{pmatrix}.
\end{equation}
By the definition of the matrix $S$, we have
\begin{equation}
|S(\cA_{\s^{(0)},\tau}-\cA_{\s_0}S)^{-1}\ca(j)|=\sum_{l=1}^k\sum_{i\in\tau^{(l)}}{}^t{\bf e}_i(\cA_{\s^{(0)},\tau}-\cA_{\s_0}S)^{-1}\ca(j).
\end{equation}
On the other hand, we have
\begin{equation}
|(\cA_{\s^{(0)},\tau}-\cA_{\s_0}S)^{-1}\ca(j)|=\sum_{i\in\s^{(0)}}{}^t{\bf e}_i(\cA_{\s^{(0)},\tau}-\cA_{\s_0}S)^{-1}\ca(j)+\sum_{l=1}^k\sum_{i\in\tau^{(l)}}{}^t{\bf e}_i(\cA_{\s^{(0)},\tau}-\cA_{\s_0}S)^{-1}\ca(j).
\end{equation}
Combining these, we obtain an identity
\begin{equation}
|A_\s^{-1}{\bf a}(j)|=\sum_{i\in\s^{(0)}}{}^t{\bf e}_i(\cA_{\s^{(0)},\tau}-\cA_{\s_0}S)^{-1}\ca(j),
\end{equation}
which proves the first inequality since we assumed $|A_\s^{-1}{\bf a}(j)|\leq 1$. As for the second one, we should be careful that
\begin{equation}
A_\s^{-1}{\bf a}(j)=
\left(
\begin{array}{c|c}
I_{\s_0}&-S(\cA_{\s^{(0)},\tau}-\cA_{\s_0}S)^{-1}\\
\hline
O&(\cA_{\s^{(0)},\tau}-\cA_{\s_0}S)^{-1}
\end{array}
\right)
\begin{pmatrix}
{\bf e}_l\\
\hline
\ca(j)
\end{pmatrix}
=
\begin{pmatrix}
{\bf e}_l-S(\cA_{\s^{(0)},\tau}-\cA_{\s_0}S)^{-1}\ca(j)\\
\hline
(\cA_{\s^{(0)},\tau}-\cA_{\s_0}S)^{-1}\ca(j)
\end{pmatrix}.
\end{equation}
Repeating the same argument as the first case, we obtain the desired inequality.
\end{proof}

\noindent
Thus, we can confirm that the integral is convergent if $z\in U_\s$. Expanding the integrand, we obtain the basic formula
\begin{equation}
f_{\s,0}(z)=\frac{\displaystyle\prod_{l:\tau^(l)\neq\varnothing}e^{-\pi\ii(1-\gamma_l)}\displaystyle\prod_{l:\tau^(l)=\varnothing}e^{-\pi\ii\gamma_l}}{\Gamma(\gamma_1)\dots\Gamma(\gamma_k)\displaystyle\prod_{l:\tau^(l)=\varnothing}(1-e^{-2\pi\ii\gamma_l})}\sum_{i=1}^r\left( 1-\exp\left\{ -2\pi\ii\sum_{i\in\s^{(0)}}{}^t{\bf e}_iA_\s^{-1}(d+A_{\bs}{\bf k}(i))\right\}\right)\varphi_{\s,{\bf k}(i)},
\end{equation}
where $\{ A_{\bs}{\bf k}(i)\}_{i=1}^r=\Z^{(n+k)\times 1}/\Z A_\s$. As before, we denote the integration cycle above $\Gamma_{\s,\s_0, 0}$. For any $\tilde{\bf k}\in\Z^{n}$, we consider a deck transformation $\Gamma_{\s,\s_0,\tilde{\bf k}}$ of $\Gamma_{\s,0}$ associated to $\xi_\tau\mapsto e^{2\pi\ii{}^t\tilde{\bf k}}\xi_\tau$ and put 
\begin{equation}
f_{\s,\tilde{\bf k}}(z)=\frac{1}{(2\pi\ii)^{m_\s}}\int_{\Gamma_{\s,\s_0,\tilde{\bf k}}} h_{1,z^{(1)}}(x)^{-\gamma_1}\cdots h_{k,z^{(k)}}(x)^{-\gamma_k}x^{c-1} e^{h_{0,z^{(0)}}(x)}dx.
\end{equation}

\noindent
Computing as in \S\ref{SectionEuler}, we obtain the formula

\begin{align}
f_{\s,\tilde{\bf k}}(z)=&\frac{\displaystyle\prod_{l:\tau^(l)\neq\varnothing}e^{-\pi\ii(1-\gamma_l)}\displaystyle\prod_{l:\tau^(l)=\varnothing}e^{-\pi\ii\gamma_l}}{\Gamma(\gamma_1)\dots\Gamma(\gamma_k)\displaystyle\prod_{l:\tau^(l)=\varnothing}(1-e^{-2\pi\ii\gamma_l})}
\exp\{2\pi\ii
\transp{
\begin{pmatrix}
O\\
\tilde{\bf k}
\end{pmatrix}
}
A_\s^{-1}
d
\}
\sum_{i=1}^r\exp\{2\pi\ii
\transp{
\begin{pmatrix}
O\\
\tilde{\bf k}
\end{pmatrix}
}
A_\s^{-1}
{\bf k}(i)
\}\times\nonumber\\
 &\left( 1-\exp\left\{ -2\pi\ii\sum_{i\in\s^{(0)}}{}^t{\bf e}_iA_\s^{-1}(d+A_{\bs}{\bf k}(i))\right\}\right)\varphi_{\s,{\bf k}(i)}.
\end{align}

\begin{thm}\label{thm:fundamentalthm5}
Take a regular triangulation $T$ of $A$ such that $s_j\leq 1$ for any $j\in\barsigma$. Assume that $\Z A=\Z^{n+k}$, the parameter vector $d$ is very generic with respect to any $\sigma\in T$, and that for any $i=1,\dots,r$, one has $\displaystyle\sum_{i\in\sigma^{(0)}}{}^t{\bf e}_iA_\sigma^{-1}(d+A_{\bs}{\bf k}(i))\notin\Z$. Then, if one puts
\begin{equation}
f_{\sigma,\tilde{\bf k}(j)}(z)=\frac{1}{(2\pi\ii)^{m_\s}}\int_{\Gamma_{\s,\s_0,\tilde{\bf k}(j)}} h_{1,z^{(1)}}(x)^{-\gamma_1}\cdots h_{k,z^{(k)}}(x)^{-\gamma_k}x^{c-1} e^{h_{0,z^{(0)}}(x)}dx,
\end{equation}
$\bigcup\{ f_{\sigma,\tilde{\bf k}(j)}(z)\}_{j=1}^{r}$ is a basis of solutions of $M_A(d)$ on the non-empty open set $U_T$, where $\{\tilde{\bf k}(j)\}_{j=1}^{r}$ is a complete system of representatives $\Z^{n}/\Z{}^t(\cA_\tau-\cA_{\s_0}S)$. Moreover, for each $\sigma\in T,$ one has a transformation formula 
\begin{equation}
\begin{pmatrix}
f_{\sigma,\tilde{\bf k}(1)}(z)\\
\vdots\\
f_{\sigma,\tilde{\bf k}(r)}(z)
\end{pmatrix}
=
T_\sigma
\begin{pmatrix}
\varphi_{\sigma,v^{{\bf k}(1)}}(z)\\
\vdots\\
\varphi_{\sigma,v^{{\bf k}(r)}}(z)
\end{pmatrix}.
\end{equation}
Here, $T_\sigma$ is an invertible $r\times r$ matrix given by 
\begin{align}
T_\sigma&=\frac{\displaystyle\prod_{l:\tau^{(l)}\neq\varnothing}e^{-\pi\ii(1-\gamma_l)}\displaystyle\prod_{l:\tau^{(l)}=\varnothing}e^{-\pi\ii\gamma_l}}{\Gamma(\gamma_1)\dots\Gamma(\gamma_k)\displaystyle\prod_{l:\tau^{(l)}=\varnothing}(1-e^{-2\pi\ii\gamma_l})}
\diag\Big( \exp\{2\pi\ii
\transp{
\begin{pmatrix}
O\\
\tilde{\bf k}(i)
\end{pmatrix}
}
A_\s^{-1}
d
\}\Big)_{i=1}^{r}
\times\nonumber\\
 & \quad\Big(\exp\{2\pi\ii
\transp{
\begin{pmatrix}
O\\
\tilde{\bf k}(i)
\end{pmatrix}
}
A_\s^{-1}
{\bf k}(j)
\}\Big)_{i,j=1}^{r}
\diag\Bigg(1-\exp\left\{-2\pi\sqrt{-1}\displaystyle\sum_{i\in\sigma^{(0)}}{}^t{\bf e}_iA_\sigma^{-1}(d+A_{\bar{\sigma}}{\bf k}(i))\right\}\Bigg)_{i=1}^{r}.\nonumber
\end{align}
\end{thm}

\end{document}